\documentclass[11pt]{amsart}

\usepackage[OT2, T1]{fontenc}
\usepackage{url}
\usepackage{amsmath}
\usepackage{array}
\usepackage{graphicx}
\usepackage{amsfonts}
\usepackage{amssymb}
\usepackage{amstext}
\usepackage{amsthm}
\usepackage{hyperref}
\usepackage{colonequals}
\usepackage{enumitem}
\usepackage[alphabetic,lite]{amsrefs}
\usepackage{cleveref}
\usepackage[all,cmtip]{xy}
\usepackage{fullpage}
\usepackage{appendix}
\usepackage{scalerel}
\usepackage{mathrsfs}
\usepackage{comment}

\numberwithin{equation}{section}

\newtheorem{theorem}{Theorem}
\newtheorem{lemma}[theorem]{Lemma}
\newtheorem{proposition}[theorem]{Proposition}
\newtheorem{corollary}[theorem]{Corollary}

\newtheorem{conj}[theorem]{Conjecture}

\theoremstyle{definition}
\newtheorem{defn}[theorem]{Definition}

\newtheorem{para}[theorem]{}

\theoremstyle{remark}
\newtheorem{remark}[theorem]{Remark}

\newtheorem*{claim}{Claim}

\newcommand{\pdiv}{\mathscr{G}}
\newcommand{\bA}{\mathbb{A}}

\newcommand{\bC}{\mathbb{C}}
\newcommand{\bD}{\mathbb{D}}

\newcommand{\bF}{\mathbb{F}}
\newcommand{\bG}{\mathbb{G}}
\newcommand{\bH}{\mathbb{H}}
\newcommand{\bI}{\mathbb{I}}

\newcommand{\bK}{\mathbb{K}}
\newcommand{\bL}{\mathbb{L}}
\newcommand{\bN}{\mathbb{N}}

\newcommand{\bQ}{\mathbb{Q}}
\newcommand{\bR}{\mathbb{R}}
\newcommand{\bS}{\mathbb{S}}

\newcommand{\bZ}{\mathbb{Z}}

\newcommand{\bbB}{\mathbf{B}}

\newcommand{\bbH}{\mathbf{H}}

\newcommand{\bbL}{\mathbf{L}}

\newcommand{\bbQ}{\mathbf{Q}}

\newcommand{\cA}{\mathcal{A}}

\newcommand{\cD}{\mathcal{D}}
\newcommand{\cE}{\mathcal{E}}

\newcommand{\cL}{\mathcal{L}}
\newcommand{\cM}{\mathcal{M}}

\newcommand{\cO}{\mathcal{O}}

\newcommand{\cS}{\mathcal{S}}

\newcommand{\cV}{\mathcal{V}}

\newcommand{\cZ}{\mathcal{Z}}

\newcommand{\Sh}{\mathcal{S}}

\newcommand{\fe}{\mathfrak{e}}

\newcommand{\fp}{\mathfrak{p}}

\newcommand{\F}{F_{\infty}}

\newcommand{\univ}{\textrm{univ}}
\newcommand{\et}{{\text{\'et}}}
\newcommand{\dR}{{\mathrm{dR}}}
\newcommand{\cris}{{\mathrm{cris}}}

\newcommand{\thh}{^{\text {th}}}

\newcommand{\la}{\langle}
\newcommand{\ra}{\rangle}
\newcommand{\Imin}{\bI^{\min}}
\newcommand{\BB}{{\mathrm{BB}}}

\DeclareMathOperator{\GL}{GL}

\DeclareMathOperator{\GSp}{GSp}

\DeclareMathOperator{\SO}{SO}
\DeclareMathOperator{\GSpin}{GSpin}

\DeclareMathOperator{\Spf}{Spf}
\DeclareMathOperator{\Mp}{Mp}
\DeclareMathOperator{\mt}{mt}
\DeclareMathOperator{\err}{err}

\DeclareMathOperator{\End}{End}
\DeclareMathOperator{\Hom}{Hom}

\DeclareMathOperator{\Res}{Res}
\DeclareMathOperator{\Spec}{Spec}

\DeclareMathOperator{\Span}{Span}
\DeclareMathOperator{\der}{der}
\DeclareMathOperator{\Fil}{Fil}
\DeclareMathOperator{\disc}{disc}

\DeclareMathOperator{\Id}{Id}

\DeclareMathOperator{\rk}{rk}
\DeclareMathOperator{\tor}{tor}
\DeclareMathOperator{\bb}{bb}
\DeclareMathOperator{\grph}{graph}
\DeclareMathOperator{\Frob}{Frob}
\DeclareMathOperator{\gr}{gr}
\DeclareMathOperator{\im}{im}

\usepackage[usenames,dvipsnames]{color}  

\begin{document}

\title{Picard ranks of K3 surfaces over function fields and the Hecke orbit conjecture}

\author{Davesh Maulik}
\author{Ananth N. Shankar}
\author {Yunqing Tang}

\maketitle

\numberwithin{theorem}{section}
\begin{abstract}

Let $\mathscr{X} \rightarrow C$ be a non-isotrivial and generically ordinary family of K3 surfaces over a proper curve $C$ in characteristic $p \geq 5$.  We 
prove that the geometric Picard rank jumps at infinitely many closed points of $C$.
 More generally, suppose that we are given the canonical model of a Shimura variety $\Sh$ of orthogonal type, associated to a lattice of signature $(b,2)$ that is self-dual at $p$.  We prove that any generically ordinary proper curve $C$ in $\Sh_{\overline{\bF}_p}$ intersects \emph{special divisors} of $\Sh_{\overline{\bF}_p}$ at infinitely many points.
 As an application, we prove the ordinary Hecke orbit conjecture of Chai--Oort in this setting; that is, we show that ordinary points in $\Sh_{\overline{\bF}_p}$ have Zariski-dense Hecke orbits.  We also deduce the ordinary Hecke orbit conjecture for certain families of unitary Shimura varieties.

\end{abstract}
\section{Introduction}
\subsection{Families of K3 surfaces}

Given a family of complex K3 surfaces, it is a well-known fact that the Picard ranks of the fibers will jump at infinitely many special points, so long as the family is not isotrivial.
More precisely, let us recall the following Hodge-theoretic result, due to Green \cite{Voisin} and Oguiso \cite{Oguiso}.
Let $\Delta$ be the unit disc in $\bC$, and let 
$\mathscr{X} \rightarrow \Delta$ be a non-isotrivial family of (compact) K3 surfaces. If $\rho$ denotes the minimal Picard rank of $\mathscr{X}_s$, $s\in \Delta$, then the set of points $t\in \Delta$ for which the Picard rank of $\mathscr{X}_t$ is greater than $\rho$ is a countable, dense subset of $\Delta$.  In particular, there are infinitely many such points.

In positive characteristic, this question is more subtle. Suppose we are given
$\mathscr{X}\rightarrow C$, where $C/\overline{\bF}_p$ is a curve and $\mathscr{X}$ is a non-isotrivial family of K3 surfaces.  It is now no longer the case that the Picard rank has to jump at infinitely many points of $C$.  
For example, there exist families where every fiber $X_t$ is a supersingular K3 surface, in which case the rank is always $22$.  By studying families of non-ordinary Kummer surfaces, one can produce families where the Picard rank jumps, but only at finitely many points ((see \ref{ss_intro_ss} for an example).

In both of these examples, the generic fiber is not \emph{ordinary}.  The first main result of this paper shows that, under additional hypotheses, if the generic fiber is ordinary, then there will be infinitely many points where the Picard rank jumps.  That is, we show the following:



\begin{theorem}\label{thm_K3}
Let $C/\overline{\bF}_p$ denote a smooth proper curve where $p \geq 5$ is a prime number, and let $\mathscr{X}\rightarrow C$ denote a generically ordinary non-isotrivial family of K3 surfaces. Suppose that the discriminant\footnote{Note that the Picard lattice of a K3 surface is equipped with a non-degenerate quadratic form arising from the intersection pairing. The discriminant of the Picard lattice is defined to be the discriminant of this quadratic form.} of the generic Picard lattice is prime to $p$. Then there exist infinitely many points $c\in C(\overline{\bF}_p)$ such that the Picard rank of $\mathscr{X}_c$ is greater than the generic Picard rank of $\mathscr{X}$. 
\end{theorem}

Broadly speaking, Theorem \ref{thm_K3} is proved by studying moduli spaces of K3 surfaces, viewed as (canonical integral models of) GSpin Shimura varieties $\Sh$ associated to quadratic $\bZ$-lattices $(L,Q)$ having signature $(b,2)$. These Shimura varieties admit families of ``special divisors'', which are themselves GSpin Shimura varieties associated to sublattices of $(L,Q)$ having signature $(b-1,2)$, whose points parameterize K3 surfaces with Picard rank greater than those parameterized by ``generic points'' of the ambient Shimura variety. 

The notion of special divisors makes sense in the more general setting of GSpin Shimura varieties $\Sh$ associated to quadratic lattices $(L,Q)$ having signature $(b,2)$ for all positive integers $b$. For every positive integer $m$, there exists a divisor $\cZ(m) \subset \Sh$ which, if not empty, is also (the integral model of) a GSpin Shimura variety. We prove the following theorem which directly implies Theorem \ref{thm_K3}.

\begin{theorem}\label{thm_int}
Let $\Sh$ denote the canonical integral model over $\bZ_p$ of the GSpin Shimura variety associated to a quadratic $\bZ$-lattice $(L,Q)$ of signature $(b,2)$, such that $p$ does not divide the discriminant of $(L,Q)$. Assume that $b\geq 3, p\geq 5$. Let $C$ be an irreducible smooth proper curve with a finite morphism $C\rightarrow \Sh_{\bar{\bF}_p}$ such that the generic point of $C$ is ordinary and that the image of $C$ does not lie in any special divisors $Z(m):=\cZ(m)_{\bar{\bF}_p},m\in \bZ_{>0}$. Then there exist infinitely many $\bar{\bF}_p$-points on $C$ which lie $\displaystyle\cup_{m\in \bN, p\nmid m}Z(m)$.
\end{theorem}

In the case of $\cS$ being Hilbert modular surfaces or Siegel modular threefold (with $b=2,3$ respectively), \Cref{thm_int} follows from our earlier paper \cite{MST}, but the general setting considered here requires additional techniques.

\subsection{The Hecke orbit conjecture.}

The second goal of this paper is to apply Theorem \ref{thm_int} to the study of Hecke orbits in characteristic $p$.  In general, Shimura varieties are naturally equipped with a set of correspondences, known as Hecke correspondences. Roughly speaking, these Hecke correspondences permute\footnote{The special divisors $Z(d)$ and $Z(m^2d)$ are in the same Hecke orbit.} the set of special divisors. In characteristic zero, the dynamics of Hecke correspondences are well-understood.  For example, work of Clozel--Oh--Ullmo \cite{COU} proves that the Hecke orbit of a point equidistributes in the analytic topology. However, in characteristic $p$, the behavior of the Hecke orbit of a point is still far from understood. 

The first result along these lines is due to Chai \cite{Chai95}, who proved that the prime-to-$p$ Hecke orbit of an ordinary point is Zariski dense in $\cA_{g,\bF_p}$, the moduli space of principally polarized abelian varieties over $\bF_p$. Guided by this, Chai and Oort have the following more general conjecture for arbitrary Shimura varieties. 
\begin{conj}[Chai--Oort]\label{conjHO}
Let $\Sh$ denote the canonical integral model of a Shimura variety of Hodge type (with hyperspecial level) and let $\Sh_{\overline{\bF}_p}$ denotes its special fiber.\footnote{The conjecture was made in the PEL case, but is expect to hold for Hodge type Shimura varieties too, which includes the case of PEL Shimura varieties.} Then the prime-to-$p$ Hecke orbit of a $\mu$-ordinary point is Zariski dense in $\Sh_{\overline{\bF}_p}$.\footnote{Being $\mu$-ordinary means that this point lies in the open Newton stratum of $\Sh_{\bar{\bF}_p}$ and it means ordinary if the ordinary locus in $\Sh_{\bar{\bF}_p}$ is nonempty, which will be the case for us in the rest of the paper.}
\end{conj}

There is a further generalization to non-ordinary points, which we do not discuss here. For more about the Hecke orbit conjecture and generalizations, see \cite{Ch03}, \cite{Chai05},\cite{Chai06}, \cite{CO06}, \cite{CO09} and \cite{CO19}.
Using Theorem \ref{thm_int} as our main input, we establish the ordinary Hecke orbit conjecture for GSpin Shimura varieties, as well as certain unitary Shimura varieties. 
\begin{theorem}\label{thm_Hecke}
Let $\Sh_{\bF_p}$ denote the mod $p$ (where $p \geq 5$) fiber of the canonical integral Shimura variety associated to one of the following data: 
\begin{enumerate}
    \item[The orthogonal case.] A quadratic $\bZ$-lattice with signature $(b,2)$ having discriminant prime to $p$ with the associated Shimura variety defined in \S\ref{def_Sh}.
    \item[The unitary case.] An imaginary quadratic field $K$ split at $p$, and an $\cO_K$-Hermitian lattice having signature $(n,1)$, with discriminant prime to $p$ with the associated Shimura variety defined in \cite[\S 2]{KR14} (see also \cite[\S 9.3]{SSTT}). 
\end{enumerate}
Then the prime-to-$p$ Hecke orbit of an ordinary point is dense in $\Sh_{\bF_p}$. 
\end{theorem}
As far as we know, this result is the first of its kind towards settling the Hecke orbit conjecture in the setting of orthogonal Shimura varieties.

\subsection{Outline of the proof of Theorem \ref{thm_int}}
 There are two broad steps in our proof: 
\begin{enumerate}

    \item We use Borcherds theory to compute the asymptotic of the intersection numbers $(C .Z(m))$ as $m\rightarrow \infty$; 
    
    \item We then prove that given finitely many points $P_1, \cdots, P_n \in C(\overline{\bF}_p)$, the local contributions have the property $\sum_{p\nmid m, \ 1\leq m\leq X}\sum_{i=1}^n i_{P_i}(C.Z(m)) < \sum_{1\leq m\leq X} C. Z(m)$ for large enough $X\in \bZ$.

\end{enumerate}
These two steps together prove that as $X\rightarrow \infty$, more and more points $P\in C(\overline{\bF}_p)$ must contribute to the intersection $C.(\sum_{p\nmid m, 1\leq m\leq X} Z(m))$, thereby yielding Theorem \ref{thm_int}.


The second step involves both local and global techniques.  We use the moduli interpretation of the special divisors $Z(m)$ to express the local contribution $i_P(C.Z(m))$ in terms of a lattice point count in an \emph{infinite} nested sequence of lattices. This is another way in which the characteristic-$p$ nature of this work complicates matters -- the analogous expression in the characteristic $0$ setting involved a lattice point count in a \emph{finite} nested sequence of lattices, which makes matters far more tractable. The most technical part of the paper deals with controlling the main term of $i_P(C.Z(m))$ for supersingular points, which we do over Sections \ref{sec_three},\ref{sec_decay_sg} and \ref{sec_decay_ssp}. This requires using Ogus's and Kisin's work to explicitly understand the equicharacteristic deformation theory of special endomorphisms at supersingular points in terms of crystalline theory.  

One of the difficulties of this result compared to \cite{MST} is that, for $b \leq 3$, it is relatively easy to bound the error terms of 
$i_P(C.Z(m))$.  However, in general, high levels of tangency between $C$ and special divisors could in principle cause
this term to grow uncontrollably.
In order to control this, we use a global argument which first appeared in \cite{SSTT}. Note that $i_P(C .Z(m))$ is necessarily bounded above by the global intersection number $ (C.Z(m))$. We use the fact that the global bound holds for \emph{every positive integer} $m$ (representable by $(L,Q)$) in order to obtain a sufficient control on the error term of the local contribution $i_P(C. Z(m))$ \emph{on average}, as we average over all positive integers $m$. More precisely, we prove that if the error term of $i_P(C. Z(m))$ is too close to the global intersection number for several values of $m$, then there must exist a positive integer $m_0$ for which the local intersection number $i_P(C.Z(m_0))$ is greater than the global intersection number $(C.Z(m_0))$, which is a contradiction. 


It is crucial to our proof of the local bound that the curve $C$ is an algebraic curve. Given a formal curve having the form $\Spf \overline{\bF}_p[[t]] \subset \Sh_{\bF_p}$ with closed point $P$, the term $i_P(\Spf \overline{\bF}_p[[t]]. Z(m))$ is well defined. It is easy to construct examples of formal curves that have the property that $i_P(\Spf \overline{\bF}_p[[t]]. Z(m_i))$ grows exponentially fast for appropriate sequences of integers $m_i$,
as we discuss in Section \ref{formal_example}.  Indeed, the growth rate of $i_P(\Spf \overline{\bF}_p[[t]]. Z(m))$ can be used as a necessary criterion to determine whether or not a formal curve contained in $\Sh$ is algebraizable. 

\subsection{Contributions from supersingular points}\label{ss_intro_ss}
As stated in the outline, the most technical part of our paper is dealing with supersingular points.  The main difficulty is caused by the fact that the local contribution $i_P(C. Z(m))$ from a supersingular point $P\in C(\overline{\bF}_p)$ has the same order of magnitude as the global intersection number $(C.Z(m))$ as $m\rightarrow \infty$; 
Indeed, the global intersection number $(C.Z(m))$ can be expressed in terms of the $m$-th Fourier coefficient of a non-cuspidal modular form of weight $1+b/2$, whose Eisenstein part is well understood (\Cref{asymp_glo} and \Cref{Eis-cof_L}). The main term of the local contribution $i_P(C.Z(m))$ at a supersingular point $P$ is controlled by the $m$-th Fourier coefficients of the theta series associated to a nested sequence of positive definite lattices each having rank $b+2$, and is therefore also asymptotic to the $m$-th Fourier coefficient of an Eisenstein series of weight $1+b/2$ (see \S\ref{sec_lat}).

Therefore, a more refined understanding of the constants involved in the global intersection number and the supersingular contribution is needed to prove our theorem. In fact, this is precisely what goes wrong when $C$ is no longer generically ordinary.  There are examples when finitely many supersingular points can indeed conspire to fully make up the entire global intersection number. We illustrate this with the following example.

Consider the setting of $X\rightarrow C$, where $C/\overline{\bF}_p$ is a curve and $X$ is a non-isotrivial family of Kummer surfaces. Indeed, let $E\rightarrow C$ denote a non-isotrivial family of elliptic curves, and let $E_0\rightarrow C$ denote a supersingular elliptic curve pulled back to $C$. Consider the family $K(E\times_C E_0)\rightarrow C$, where $K(E\times_C E_0)$ denotes the Kummer surface associated to the abelian surface $E\times_C E_0$. The set of points $c\in C(\overline{\bF}_p)$ such that the Picard rank of $K(E\times E_0)_c$ is greater than the generic Picard rank of $K(E\times E_0)$ is precisely the set of $c\in C(\overline{\bF}_p)$ such that the fiber of $E$ at $c$ is supersingular, and therefore the total global intersection number is made up from the local contributions from these finitely many supersingular points.

\subsection{Outline of proof of Theorem \ref{thm_Hecke}}


    
    
    
We now survey the proof of the Hecke orbit conjecture, using  Theorem \ref{thm_int} as input.
We will focus on the orthogonal case, since the unitary case follows by a similar argument.

Let us first observe that Chai's approach for $\cA_g$ does not easily extend to this case. Chai's argument involves several steps, some of which generalize to the case of all Shimura varieties of Hodge type, but there are many ideas in Chai's work which don't generalize to our setting. Indeed, a key step in his paper is the so-called Hilbert trick which states that every $\overline{\bF}_p$-valued point of $\cA_g$ is contained in a positive-dimensional Shimura subvariety of $\cA_g$, namely a Hilbert modular variety. Unfortunately, this fact does not hold for most Shimura varieties.  The case of Hilbert modular varieties is more tractable than $\cA_g$ because the geometrically simple factors of the associated reductive groups have rank 1.  

Instead, our idea is to use an inductive argument on the dimension of $\Sh_{\bF_p}$.
Our argument consists of the following steps: 
\begin{enumerate}
    \item The setting of Shimura varieties associated to quadratic lattices having signature $(1,2)$ follows because these Shimura varieties are one-dimensional, and it is well known that the Hecke orbit of an ordinary point is an infinite set. Now, inductively assume that the Hecke orbit conjecture has been proved for all Shimura varieties associated to quadratic lattices having signature $(b-1,2)$, with discriminant relatively prime to $p$, where $b>1$. 
    \item Let $Z\subset \Sh_{\bF_p}$ denote a generically ordinary Hecke stable subvariety, where $\Sh$ is the canonical integral model of a Shimura variety associated to a quadratic lattice having signature $(b,2)$ with discriminant relatively prime to $p$. Such a subvariety $Z$ necessarily has to be positive dimensional, as the Hecke orbit of an ordinary point is necessarily infinite. 
    
    \item Suppose that $Z$ contains a proper curve $C$ that is generically ordinary. Then, Theorem \ref{thm_int} implies that $C$ intersects the union special divisors $\bigcup_{p\nmid m}Z(m)$ at infinitely many points, and therefore at an ordinary point $x \in Z(m_0)$. The special divisor $Z(m_0)$ is the special fiber of a Shimura variety in its own right, associated to a quadratic lattice having signature $(b-1,2)$ and prime-to-$p$ discriminant (because $p\nmid m_0$), and so the prime-to-$p$ Hecke orbit of $x$ contains a Zariski-dense subset of $Z(m_0)$ by the inductive hypothesis. Therefore, $Z(m_0) \subset Z$, and the result follows from the fact that the Hecke orbit of any special divisor is Zariski dense in $\Sh_{\bF_p}$. 
    
    \item To deal with the case when $Z$ might not contain a proper curve, we directly prove that any generically ordinary Hecke stable subvariety that intersects the boundary of the Baily--Borel compactification of $\Sh_{\bF_p}$ (constructed in \cite{MP19}) must be all of $\Sh_{\bF_p}$. 
\end{enumerate}

In other words, even though a ``generic'' $\overline{\bF}_p$-valued point of $\Sh$ may not lie in a smaller positive-dimensional Shimura variety, we are able to reduce to the case of a smaller Shimura variety using the intersection-theoretic input of Theorem \ref{thm_int}.

\subsection{Previous work}
In addition to the ones mentioned above, we discuss here  other related work in the literature.

Chai and Oort \cite{CO06} proved \Cref{thm_K3} for Kummer surfaces associated to the product of two elliptic curves and \Cref{thm_int} for $\Sh=\cA_1\times \cA_1$ without the assumption that $C$ is proper. 
The number field analogs of \Cref{thm_K3} and \Cref{thm_int} have been proved in \cite{SSTT}, based on the previous work by Charles \cite{Ch} and \cite{ST20} for $\cA_1\times \cA_1$ and Hilbert modular surfaces respectively.
For characteristic zero families, \cite{Tayou} proved an equidistribution result on the the Noether--Lefschetz locus, which is a refinement of the theorem of Green.


For the results on Hecke orbits,
Chai has also proved Conjecture \ref{conjHO} in the setting of Hilbert modular varieties, as well as for some PEL type C Shimura varieties. Building on work of Chai, the second named author \cite{Shankar} proved Conjecture \ref{conjHO} for the ordinary locus in Deligne's mod\`{e}les \'{e}tranges. 

There is also a generalization of \Cref{conjHO} to $\bar{\bF}_p$-points in other Newton strata (see \cite[Conj.~3.2]{Chai06}).  In the case of $\cA_g$, there is extensive work of Chai and Oort studying the properties of Newton strata (see their survey paper \cite{CO19} and the references there); in combination with work of Yu, this gives the full Hecke orbit conjecture  for $\cA_g$ and Hilbert modular varieties (see for instance \cite{Chai05} for the proofs). More recently, Zhou proved \Cref{conjHO} for (the $\mu$-ordinary loci of) quaternionic Shimura varieties associated to quaternion algebras over some totally real fields (\cite[Thm.~3.1.3, Rmk.~3.1.4]{Zhou}); 
and Xiao proved the generalized version for certain PEL Shimura varieties of type A and C and the points in those Newton strata which contain certain hypersymmetric points (\cite[Thm.~7.1, Cors. 7.5, 7.6]{Xiao}). 


\subsection{Organization of paper}
In \S\ref{sec_GSpin}, we recall the definitions of GSpin Shimura varieties, special endomorphisms, and special divisors. 
In \S\ref{sec_heuristic}, we formulate theorems \ref{cor_decay_sg} and \ref{cor_decay_ssp} which
describe the decay of lattices of special endomorphisms at supersingular points.  The proof of these statements occupies the next three sections, which may be skipped on a first reading. 
In \S\ref{sec_three}, we recall from Ogus's work \cite{Ogus79} the explicit description of the lattices of special endomorphisms at supersingular points and we use Kisin's work \cite{Kisin} to compute an $F$-crystal $\bbL_\cris$, which controls the deformation of special endomorphisms. In sections \S\S\ref{sec_decay_sg}-\ref{sec_decay_ssp} we use this explicit description to 
prove the decay results.  In \S\ref{sec_pf}, we prove \Cref{thm_int}
following the outline given above.    In \S\ref{sec_Hecke}, we prove \Cref{thm_Hecke} using \Cref{thm_int}; we only use the statement (not the proof) of \Cref{thm_int} and the reader who is interested in the Hecke orbit conjecture may directly proceed to \S\ref{sec_Hecke} after \S\ref{sec_GSpin}.

\subsection*{Notation} Throughout the paper, $p\geq 5$ is a prime. We write $f\asymp g$ if $f=O(g)$ and $g=O(f)$.

\subsection*{Acknowledgement} We thank George Boxer, Ching-Li Chai, Johan de Jong, Kai-Wen Lan, Keerthi Madapusi Pera, Frans Oort, Arul Shankar, Andrew Snowden, Salim Tayou, and Tonghai Yang for helpful discussions, as well as Arthur and D.W. Read for additional assistance.   Y.T. has been partially supported by the NSF grant DMS-1801237.


\section{GSpin Shimura varieties and special divisors}\label{sec_GSpin}

In this section, we review basic definitions, terminology, and notation for GSpin Shimura varieties, special endomorphisms, and special divisors that we need in the rest of the paper.

Let $(L,Q)$ be a quadratic $\bZ$-lattice of signature $(b,2)$, $b\geq 1$. We assume that $(L,Q)$ is self-dual at $p$. We recall the canonical integral model of the GSpin Shimura variety associated to $(L,Q)$ and the definition of special divisors. The main references are \cite[\S\S 3-5]{MP16} and \cite[\S\S 4.1-4.3]{AGHMP};\footnote{Since we work with the hyperspecial case, all the results listed here are in \cite{MP16} and we follow the convention of using cohomology as in \cite{MP16}.} see also \cite[\S 2]{SSTT} for a brief summary.

\begin{para}\label{def_Sh}
Let $V:=L\otimes_\bZ \bQ$ and let $[-,-]$ denote the bilinear form on $V$ given by $[x,y]=Q(x+y)-Q(x)-Q(y)$. Let $G:=\GSpin(L\otimes \bZ_{(p)},Q)$ be the group of spinor similitudes of $L\otimes \bZ_{(p)}$, which is a reductive group over $\bZ_{(p)}$. The group $G(\bR)$ acts on the Hermitian symmetric domain $D_L=\{z\in V_{\bC}\mid [z,z]=0, [z,\bar{z}]<0\}/\bC^\times$ via $G\rightarrow \SO(V)$. For $[z]\in D_L$ with $z\in V_{\bC}$, let $h_{[z]}: \Res_{\bC/\bR} \bG_m\rightarrow G_\bR$ denote the unique homomorphism which induces the Hodge decomposition on $V_\bC$ given by $V_\bC^{1,-1}=\bC z, V_\bC^{0,0}=(\bC z\oplus \bC \bar{z})^\perp, V_\bC^{-1,1}=\bC \bar{z}$. Thus $(G_\bQ, D_L)$ is a Shimura datum with reflex field $\bQ$. 

Let $\bK\subset G(\bA_f)$ be a compact open subgroup contained in $G(\bA_f)\cap C(L\otimes \widehat{\bZ})^\times$, where $C(L\otimes \widehat{\bZ})$ is the Clifford algebra of $(L\otimes \widehat{\bZ},Q)$ and we assume that $\bK_p=G(\bZ_p)$. Then we have the GSpin Shimura variety $Sh:=Sh(G_{\bQ},D_L)_{\bK}$ over $\bQ$ with $Sh(G_{\bQ},D_L)_{\bK}(\bC)=G(\bQ)\backslash D_L\times G(A_f)/\bK$ and by \cite[Theorem 2.3.8]{Kisin}, $Sh$ admits a canonical smooth integral model $\Sh:=\Sh_\bK$ over $\bZ_{(p)}$.
\end{para}

\begin{para}\label{def_bbL}
Let $H$ denote the Clifford algebra $C(L)$ equipped with the right action by itself via right multiplication and we equip $H\otimes \bZ_{(p)}$ with the action of $G$ by left multiplication. By picking a suitable symplectic form on $H$, we have $G_\bQ\rightarrow \GSp(H\otimes \bQ)$, which induces a morphism of Shimura data and thus, there is the Kuga--Satake abelian scheme $A^\univ\rightarrow Sh$ whose first $\bZ$-coefficient Betti cohomology $\bbH_B$ is the local system induced by $H$ (and its $G_\bQ$-action).
This Kuga--Satake abelian scheme $A^\univ\rightarrow Sh$
extends to an abelian scheme $\cA^\univ\rightarrow \Sh$ equipped with a left $C(L)$-action. Let $\bbH_\dR, \bbH_{\ell, \et}$ denote the first relative de Rham cohomology and $\ell$-adic \'etale cohomology with $\bZ_\ell$-coefficient of $\cA^\univ\rightarrow \Sh$ for $\ell\neq p$, and let $\bbH_\cris$ denote the first relative crystalline cohomology of $\cA^\univ_{\bF_p}\rightarrow \Sh_{\bF_p}$.

The action of $L$ on $H$ via left multiplication induces a $G_{\bQ}$ equivariant map on $L\otimes \bQ \rightarrow \End_{C(L)}(H\otimes \bQ)$ and thus we have a $\bZ$-local system $\bbL_B$ over $Sh$ with a natural embedding $\bbL_B\rightarrow \End_{C(L)}(\bbH_B)$. There are a filtered vector bundle with connection $\bbL_{\dR}\subset \End_{C(L)}(\bbH_{dR})$, a $\bZ_\ell$-lisse sheaf $\bbL_{\ell, \et}\subset \End_{C(L)}(\bbH_{dR})$ and an $F$-crystal $\bbL_\cris \subset \End_{C(L)}(\bbH_\cris)$ such that these embeddings along with $\bbL_B\rightarrow \End_{C(L)}(\bbH_B)$ are compatible under Betti-de Rham, Betti-\'etale, de Rham-crystalline comparison maps (see \cite[Prop.~3.11, 3.12, Prop.~4.7]{MP16}). By \cite[\S 4.3]{AGHMP}, $\bbL_?,?=B,\dR, (\ell,\et), \cris$ are equipped with a natural quadratic form $\bbQ$ given by $f\circ f=\bbQ(f)\cdot \Id$ for a section $f$ of $\bbL_?$.
\end{para}

\begin{defn}[{\cite[Def.~4.3.1]{AGHMP}}] Let $T$ denote an $\Sh$-scheme.
\begin{enumerate}
    \item An endomorphism $v\in \End_{C(L)}(\cA^\univ_T)$ is \emph{special} if all cohomological realizations of $v$ lie in the image of $\bbL_?\rightarrow\End_{C(L)}(\bbH_?)$, where $?=B, \dR, \cris, (\ell, \et) $, for all $\ell\neq p$.\footnote{We drop the ones which do not make sense. For instance, if $p$ is invertible in $T$, we drop $\cris$; if $T_\bQ=\emptyset$, we drop $B$.}
    \item Let $\cA^\univ_T[p^\infty]$ denote the $p$-divisible group associated to $\cA^\univ_T$. An endomorphism $v\in \End_{C(L)}(\cA^\univ_T[p^\infty])$ is \emph{special} if its crystalline realization lies in $\bbL_\cris$.
\end{enumerate}
\end{defn}

\begin{remark}\label{posdef}
For connected $T$, an endomorphism $v\in \End_{C(L)}(\cA^\univ_T)$ or $\End_{C(L)}(\cA^\univ_T[p^\infty])$ is special if and only if there exists a geometric point $t\in T$ such that $v_t\in \End_{C(L)}(\cA^\univ_t)$ or $\End_{C(L)}(\cA^\univ_t[p^\infty])$ is special (see \cite[Prop.~4.3.4, Lem.~4.3.5]{AGHMP} and their proofs). Moreover, if $T_{\bF_p}\neq \emptyset$, then we may pick a geometric point $t\in T_{\bF_p}$ and for such $t$, $v_t\in\End_{C(L)}(\cA^\univ_t)$ is special if and only if the crystalline realization of $v_t$ lies in $\bbL_\cris$ (see \cite[Cor.~5.22, \S 5.24]{MP16}). In this paper, we will mainly work with $T$ which is an $\Sh_{\bF_p}$-scheme and thus we will only use $\bbL_\cris$ to verify special endomorphisms.
\end{remark}

\begin{remark}
By \cite[Lem.~5.2]{MP16}, for $v\in \End_{C(L)}(\cA^\univ_T)$ special, we have $v\circ v=[Q(v)]$ for some $Q(v)\in \bZ_{\geq 0}$ and $v\mapsto Q(v)$ is a positive definite quadratic form on the $\bZ$-lattice of special endomorphisms of $\cA^\univ_T$.
\end{remark}

\begin{defn}\label{def_spdiv}
For $m\in \bZ_{>0}$, the \emph{special divisor} $\cZ(m)$ is the Deligne--Mumford stack over
$\cM$ with functor of points $\cZ(m)(T) = \{v\in \End(\cA^\univ_T) \text{ special } | Q(v) = m\}$ for any $\Sh$-scheme $T$. We use the same notation for the image of $\cZ(m)$ in $\Sh$. By for instance \cite[Prop.~4.5.8]{AGHMP}, $\cZ(m)$ is an effective Cartier divisor and it flat over $\bZ_{(p)}$ and hence $\cZ(m)_{\bF_p}$ is still an effective Cartier divisor of $\Sh_{\bF_p}$; we denote $\cZ(m)_{\bF_p}$ by $Z(m)$.
\end{defn}



\section{Lattice decay statements and heuristics}\label{sec_heuristic}

In this section, we formulate local intersection multiplicities in terms of counting points from a nested sequence of lattices.  In the supersingular case, we then state decay estimates for this nested sequence that will be crucial for controlling the local contributions.  Proving these estimates will occupy \S\S \ref{sec_three}, \ref{sec_decay_sg}, and \ref{sec_decay_ssp}.  
We give a heuristic explanation for why these decay estimates suffice.  Finally, at the end of the section, we construct a formal family where the local multiplicities behave wildly; as a consequence, in our argument, it is necessary to use the global geometry to control the local error terms.

\subsection*{Preliminaries and main statements}
Let $k$ denote $\overline{\bF}_p$ and recall from \Cref{thm_int} that $C \rightarrow \cS_k$ is a smooth proper curve whose generic point maps to the ordinary locus of $\cS_k$. Let $P \in C(k)$, and let $t$ be a local coordinate at $P$ (i.e., $\widehat{C}_P = \Spf k[[t]]$). Let $\cA/k[[t]]$ denote the pullback of the universal abelian scheme $\cA^{\univ}/\cS$. Finally, let $L_n$ denote the $\bZ$-module of special endomorphisms of $\cA \bmod t^n$. The moduli-theoretic description of the special divisors yields the following expression:

\begin{equation}\label{localexpm}
    i_P(C.Z(m)) = \sum_{n=1}^{\infty} \#\{v \in L_n\mid Q(v) = m  \}.  
\end{equation}

As discussed in the introduction, one of the main difficulties in comparing local and global intersections is the contribution of supersingular and especially superspecial points; these are the supersingular points for which the lattice of special endomorphisms is as large as possible (see \S\ref{par_decomp_cL} for a precise definition). We will therefore assume that the image of $P$ in $\cS_k$ is contained in the supersingular locus of $\cS_k$.  We thus have that $\cA$ is generically ordinary and specializes to a supersingular point, and hence the Hasse invariant $H$ on $\cS_k$ must vanish to some order at $P$. 

In order to control the number of points in the nested family of lattices $L_n$, as $n$ grows, we will prove that the covolumes of these lattices grow rapidly; note that the covolume of a lattice determines -- to first order -- the number of lattice points with bounded norm.  

We define $h_P$ to be $v_t(H)$, namely the $t$-adic valuation of $H$ restricted to $\widehat{C}_P$.  Our bounds will be in terms of the quantity $h_P$, and so we make the following definitions.
\begin{defn}
Let $r\geq 0$ denote an integer, and let $a =\frac{h_P}{2}$. Define $h_r = [h_P(p^r+\hdots p+1+1/p)]$, 
$h'_r = [h_P(p^r+\hdots +1) + a/p]$ and $h'_{-1} = [a/p]$.
\end{defn}

Suppose that the point $P$ is supersingular, but not superspecial. Then we have:
\begin{theorem}\label{cor_decay_sg}
The index $|L_1/L_n|$ of $L_n$ inside $L_1$ satisfies the inequality
$$|L_1/L_n| \geq p^{2+2r}$$ 
if $h_r+1\leq n\leq h_{r+1}$.  
\end{theorem}
We remind the reader that $L_1$ contains the lattices $L_n$ with index a power of $p$ (see \cite[Rmk.~7.2.2]{MST}). The content of the above result is that the for any $n$ that is larger than $h_P(1+1/p)$, the abelian scheme $\cA \bmod t^n$ has fewer special endomorphisms than $\cA \bmod t$, and that the index of $L_n$ in $L_1$ is at least $p^2$. For $n$ greater than $h_P(p+1+1/p)$, the $\cA \bmod t^n$ has still fewer special endomorphisms than $\cA \bmod t$, and in fact the index of $L_n$ in $L_1$ is at least $p^4$, etc. 

As the lattice of special endomorphisms at $P$ is maximal when $P$ is superspecial, we need better bounds in this case. In \S\ref{sec_decay_ssp}, we establish the following result:
\begin{theorem}\label{cor_decay_ssp}
When $P$ is superspecial, the index $|L_1/L_n|$ of $L_n$ inside $L_1$ satisfies 
one of the following two inequalities:
\begin{enumerate}
    \item  $|L_1/L_n|\geq p^{1+2r}$ if $h'_{r-1}+ap^r+1\leq n \leq h'_{r}$ and $|L_1/L_n|\geq p^{2+2r}$ if $h'_{r}+1\leq n \leq h'_{r}+ap^{r+1}$.
    \item  $|L_1/L_n|\geq p$ if $h'_{-1}+a+1\leq n \leq h'_{0}$ and $|L_1/L_n|\geq p^{3+2r}$ if $h'_{r}+1\leq n \leq h'_{r+1}$.
\end{enumerate}


\end{theorem}
The above results show that there is a dichotomy between the local behavior at superspecial points, and supersingular points that are not superspecial. This is because the vanishing of the Hasse invariant on $\cS_k$ is singular precisely at superspecial points (see for instance \cite[the proof of Cor.~16]{Ogus01}). This singularity forces the covolume of $L_n$ to increase faster than it otherwise would.

\begin{para}\textbf{A heuristic.}

To motivate our approach, we give a heuristic argument here for the expectation that for $p\gg_\epsilon 1$, the sum of local intersection multiplicities $i_P(C.Z(m))$ at supersingular points on $C$ with $Z(m)$ is at most $(\frac{1}{2}+\epsilon)C.Z(m)$ as $m\rightarrow \infty$ using \Cref{cor_decay_sg,cor_decay_ssp}. The proof of \Cref{ssbound} verifies this expectation when we average over $m$.  In particular, this heuristic explains why we need a stronger decay estimate for superspecial points and why such decay should exist.
In order to just convey the basic idea, we will keep the argument presented here brief, even a little vague;  more precise statements and proofs will come later in \S\ref{sec_pf} and the reader may consult there for the precise statements and proofs. 

\Cref{cor_decay_sg,cor_decay_ssp} imply that for $p\gg 1$, the major contribution in $i_P(C.Z(m))$ comes from the intersection of $k[t]/t^{h_P}$ and $Z(m)$ (as the covolumes of $L_n$ increase). The intersection multiplicity of $k[t]/t^{h_P}$ and $Z(m)$ is at most $h_P$ times the number $b(m,P)$ of branches of the formal completion $\widehat{Z(m)}_P\subset \widehat{\Sh}_{k,P}$. 

Indeed, $b(m,P)=\#\{v\in L_1 \mid Q(v)=m\}$. By studying the theta series associated to $L_1$, we have that $b(m,P)$ is roughly $|q_L(m)|/p^{t_P/2}$, where $q_L(m)$ denotes the $m$-th Fourier coefficient of the vector-valued Eisenstein series $E_0$ of weight $1+b/2$ defined in \S\ref{def_Eis} and $t_P$ (which is an even positive integer) is the type of $P$ defined in \S\ref{par_decomp_bL}. In particular, this bound is the worst when $t_P =2$, which by definition means that $P$ is superspecial. 

We now consider the extreme case when all non-ordinary points on $C$ are superspecial. Since the Hasse invariant is a weight $p-1$ modular form on $\Sh_k$, then we have \[\sum_{P\in C(k) \text{ superspecial}} h_P=(p-1)C.\omega,\] where $\omega$ is the line bundle of modular forms of weight one.
Then without considering the first inequalities in each of (1) and (2) of \Cref{cor_decay_ssp}, an initial estimate of $\sum_{P \text{ supersingular}} i_P(C.Z(m))$ is \[\sum_{P\in C(k) \text{ superspecial}} h_P|q_L(m)|/p=\frac{p-1}{p} |q_L(m)|(C.\omega),\] 
and a priori this should be a lower bound as we have ignored tangencies of order greater than $h_P$. 

On the other hand, as we discuss in \Cref{asymp_glo}, based on Borcherds theory, this is roughly the same size as the global intersection $C.Z(m)$.
 Thus we need some extra input, which is exactly given by \Cref{cor_decay_ssp}; this result lets us replace $h_P$ by $h_P/2$ for the major term in $i_P(C.Z(m))$ and then obtain our expectation.

We can see how this works in the simplest situation, when $C$ intersects all local formal branches of $\widehat{Z(m)}_P$ transversely. Then $i_P(C.Z(m))=b(m,P)$. On the other hand, since the singular locus in the non-ordinary locus in $\Sh_k$ consists of the supersingular points, then $h_P\geq 2$ for all superspecial points and thus the total number of superspecial points on $C$ is at most $\frac{p-1}{2}C.\omega$ and thus by the above estimate of $b(m,P)$, we see that $\sum_{P \text{ supersingular}} i_P(C.Z(m))\leq (\frac{1}{2}+\epsilon)(C.Z(m))$.
\end{para}
\begin{para}\textbf{An example of a formal curve.}\label{formal_example}
We will now construct a formal curve $\Spf k[[t]]\subset \Sh_k$ with closed point $P$ where the local multiplicities $i_P(C.Z(m_i))$ grows exponentially fast for appropriate sequences of integers $m_i$. Our example will in fact be of where $P$ is ordinary. For ease of exposition, we assume that the quadratic lattice has even rank, and consequently let $b = 2c$. We may arrange for $L_1$, the $\bZ$-module of special endomorphisms, to have rank $2c = b$ and to have discriminant prime to $p$. Suppose that $e_1, f_1, e_2, f_2, \hdots e_c,f_c$ is an orthogonal $\bZ$-basis of $L_1$. Let $\cL_n = L_n\otimes \bZ_p$ denote the module of (formal) special endomorphisms $\cA[p^{\infty}] \bmod t^n$. 

 Serre--Tate theory yields the existence of coordinates $\{q_i,q'_i: 1\leq i\leq c \}$ such that the formal neighborhood of $\Sh_{k}$ at $P$ is given by $\Spf k[[q_i-1, q'_i-1]]$. Moreover, the local equation defining the locus where the formal special endomorphism $\sum_{i=1}^c(\lambda_i e_i+ \mu_i f_i) \in \cL_1,\ \lambda_i,\mu_i \in \bZ_p$ deforms is just $\displaystyle \prod_{i=1}^c q_i^{\lambda_i}{q'}^{ \mu_i} -1$. Note that this has following consequence: if $f$ is the local equation defining the locus where some special endomorphism $w$ deforms, then $f^p$ is the equation for $pw$. 
 
 We now choose $\mu_i$ to be irrational $p$-adic integers which are ``very well approximated'' by actual integers. Specifically, choose $\mu_i = \mu = \sum a_np^n $ where $0\leq a_n \leq p-1$ and $a_0 = 1$. We will choose the precise values for $a_n$, $n\geq 1$ below. 
 
 We will now construct our formal curve to satisfy the property that $\cA/\Spf \overline{\bF}_p[[t]]$ admits no non-zero special endomorphisms, but $\cA[p^\infty]/\Spf k[[t]]$ admits special endomorphisms by $\textrm{Span}_{\bZ_p}\{e_i + \mu f_i\}_{i=1}^c \subset \cL_1$. Choosing $\Spf k[[t]] \subset \Sh_k$ to be defined by the quotient map $\rho: k[[q_i,q'_i]]\rightarrow k[[t]]$, with $\rho(q_i) = (1+\alpha_i t)^{-\mu}$ and $\rho(q'_i) = (1+\alpha_i t)$, where $\alpha_i \in k$ are linearly independent over $\bF_p$, is one such example and we will treat this example. 
 
 With this setup, we are now prepared to compute the lattices $\cL_N$, and therefore also $L_N$ and $i_P(\Spf k[[t]].Z(m_i))$. The assumption that the elements $\alpha_i\in k$ are $\bF_p$-linearly independent and $\mu\in \bZ_p^\times$ implies that the local equation defining the locus in $\Spf \overline{\bF}_p[[t]]$ such that any primitive $w\in \textrm{Span}\{e_1\hdots e_c \}$ deforms is just $t$. As the endomorphisms $e_i + \mu f_i, 1\leq i \leq c$ extend to the whole of $\Spf k[[t]]$, we have $\cL_2, \dots, \cL_p = \textrm{Span} \{pe_1, \dots, pe_c, e_1 + \mu f_1, \dots, e_c + \mu f_c \}$; $\cL_{p+1},\dots, \cL_{p^2} = \textrm{Span} \{p^2e_1, \dots, p^2e_c, e_1 + \mu f_1, \dots, e_c + \mu f_c \}$; and we finally have $\cL_{p^{n-1} + a}=\textrm{Span} \{p^ne_1, \dots, p^ne_c, e_1 + \mu f_1, \dots, e_c + \mu f_c \}$, where $a\geq 1$ and $p^{n-1} + a \leq p^{n}$. Finally, we have that $L_N = L_1 \cap \cL_N$ (with the intersection in $\cL_1$). 
 
 The fact that $\displaystyle \mu = \sum_{n\geq 0}a_n p^n$ implies that $v_{i,0} = e_i+a_0f_i \in L_p$, $v_{i,1} = e_i + (a_0+a_1p)f_i\in L_{p^2}, \hdots $, $v_{i,n} = e_i + (a_0 + a_1p + \hdots a_np^n)f_i \in L_{p^{n+1}}$, etc. We finally choose our sequence of $a_n$ -- recall that we have already chosen $a_0 = 1$. To that end, define $n_0 = 0$, and recursively define $n_{j+1} = p^{2n_j}$. We define $a_{n_j} = 1$ and $a_n = 0$ if $n\neq n_j,\, \forall j \in \bZ_{\geq 0}$. For any positive integer $j_0$, we see that $v_{i,n_{j_0+1} -1} = e_i+(\sum_{j=0}^{j_0}p^{n_j})f_j \in L_{p^{n_{j+1}}} $. It is easy to see that $m_{j} :=  Q(v_{i,n_{j+1}-1}) \asymp p^{2n_j}$. Therefore, we have that $i_P(\Spf k[[t]].Z(m_j)) \geq p^{n_{j+1}}$, whose size is clearly exponential in $m_j$! We have therefore constructed an example of a formal curve, as well as a sequence of special divisors $Z(m_j)$, such that $i_P(\Spf k[[t]].Z(m_j))$ is exponential in $m_j$. In fact, $L_{p^{n_{j+1}}}$ contains a rank-$c$ sublattice with discriminant $\asymp p^{2cn_j}$ (spanned by $\{v_{i, n_{j+1}-1}\}_{i=1}^c$). Therefore, when $c>2$, by choosing our initial values $Q(e_i), Q(f_i)$ carefully, we may even arrange for $i_P(\Spf k[[t]]. Z(m))$ growing exponentially in $m_j$ (and therefore growing faster than any polynomial in $m$) for \emph{a density one} set of $m \in [m_j,m_j^N]$. 
 
 In \cite{MST}, we are able to get around this difficulty because $c \leq 2$, and hence our lattices all have relatively small rank. Indeed, in that setting, the lattices $L_{p^{n_{j+1}}}$ may contain sublattices with discriminant logarithmic in $p^{n_{j+1}}$, but these sublattices necessarily have rank bounded above by 2, and the set of integers represented by rank two positive definite lattices has density zero. 

\end{para}

\section{The $F$-crystal $\bbL_\cris$ on local deformation spaces of supersingular points}\label{sec_three}

The goal of this section and \S\S \ref{sec_decay_sg} and \ref{sec_decay_ssp}
is to prove Theorems \ref{cor_decay_sg} and \ref{cor_decay_ssp} by analyzing the deformation behavior of special endomorphisms at supersingular points.

To set up this analysis, in this section, we compute $\bbL_\cris$ over the formal neighborhoods of supersingular points in $\cS_{\bar{\bF}_p}$. As in \cite[\S 3]{MST}, we first compute $\bbL_{\cris,P}(W)$ at a supersingular point $P$, which is a quadratic space over $W:=W(\bar{\bF}_p)$ with a $\sigma$-linear Frobenius action $\varphi$, and then we use Kisin's work \cite{Kisin} to obtain $\bbL_\cris$ over the formal neighborhood of $P$. Here we use the work of Ogus \cite[\S 3]{Ogus79} to compute $\bbL_{\cris,P}(W)$ while we follow \cite{HP} in \cite{MST}; the extra input is \cite[Thm.~3.21]{Ogus79}. 

In \cite{Ogus79}, he uses the notion K3 crystals \cite[Def.~3.1]{Ogus79}, which are of weight $2$; supersingular K3 crystals are equivalent to a Tate twist applied to our $\bbL_{\cris,P}(W)$ (which are weight $0$). Our convention is the same as that in \cite{HP}. In particular, our Frobenius $\varphi$ differs from the Frobenius in \cite{Ogus79} by a factor of $1/p$. For the convenience of the reader, we give references to \cite{Ogus79} whenever possible in this paper and the reader may check \cite{MST} for the references to \cite{HP}.

\subsection*{The $F$-crystal $\bbL_\cris$ at a supersingular point}
\begin{para}\label{par_decomp_cL}
Set $k=\bar{\bF}_p, W=W(k), K=W[1/p]$ and let $\sigma$ denote the usual Frobenius action on $K$. Given a supersingular point $P$, $\bL:=\bbL_{\cris,P}(W)$ is equipped with a quadratic form $\bbQ$ (see \S\ref{def_bbL}) and a $\sigma$-linear Frobenius action $\varphi$. We note that $\varphi$ is not a endomorphism of $\bL$, but is a $\sigma$-linear map $\bL_{\cris,P}(W)\rightarrow \frac{1}{p}\bL_{\cris,P}(W)$. Let $\la-,-\ra$ denote the bilinear form on $\bL$ given by $\la x, y\ra=\bbQ(x+y)-\bbQ(x)-\bbQ(y)$. By the definition of $\bbQ$, we have $\la \varphi(x),\varphi(y)\ra=\sigma(\la x, y\ra)$. 

Let $\cL$ denote the $\bZ_p$-lattice of special endomorphisms the $p$-divisible group $A_P[p^\infty]$, where $A_P:=\cA^\univ_P$. By Dieudonn\'e theory, $\cL=\bL^{\varphi=1}$. Since $P$ is supersingular, $\rk_{\bZ_p} \cL=\rk_W \bL=\rk_{\bZ} L$ and $\bL\subset \cL\otimes_{\bZ_p} K$. 

By \cite[Thm.~3.4]{Ogus}, there is a decomposition of $\bZ_p$-quadratic lattices $(\cL, \la, \ra)=(\cL_0,\la,\ra_0)\oplus (\cL_1,\la,\ra_1)$, where $p\mid \la,\ra_0$, both $\frac{1}{p}\la,\ra_0$ and $ \la,\ra_1$ are perfect, and $2\mid \rk_{\bZ_p} \cL_0$. Thus $p\la,\ra_0$ induces a perfect $\bF_p$-valued quadratic form on the $\bF_p$-vector space $\frac{1}{p}\cL_0/\cL_0$; we also denote this quadratic form by $p\la,\ra_0$.
The type of $P$, denote by $t_P$, is defined to be $\rk_{\bZ_p} \cL_0$;\footnote{By \cite[p.~327]{Ogus01}, $t_P/2$ is the Artin invariant if $A_P$ is the Kuga--Satake abelian variety associated to a $K3$ surface.} by \cite[Cor.~3.11]{Ogus79}, $2\mid t_P, 2\leq t_P\leq \rk L$. We say $P$ is \emph{superspecial} if $t_P=2$; otherwise, we say $P$ is non-superspecial.
\end{para}

\begin{para}\label{par_decomp_bL}
The above decomposition of $\cL$ induces a decomposition of $\bL$, which allows us to compute $\bL$ explicitly.
More precisely, by \cite[Thms.~3.5,3.20]{Ogus79}, the $W$-quadratic lattice $\bL$ with Frobenius action decomposes as $\bL_0\oplus \bL_1$, where $\bL_0 \bmod \cL_0\otimes W \subset (\frac{1}{p}\cL_0/\cL_0)\otimes k$ is totally isotropic subspace with respect to $\la,\ra_0$ of dimension $t_P/2$ satisfying certain conditions and $\bL_1=\cL_1\otimes W$.
\end{para}

We first provide explicit descriptions of $\cL_0$ and $\bL_0$.

\begin{lemma}[Ogus]\label{L''_p0}
Set $n=t_P/2$ and $\lambda\in\bZ_{p^2}^\times$ such that $\lambda^2 \bmod p \in \bF_p$ is a quadratic non-residue.
 There exists a $\bZ_p$-basis $\{e_1, \dots, e_n$, $f_1, \dots, f_n\}$ of $\cL_0$ and the quadratic form $\la,\ra_0$ is given by $\langle e_i,f_i\rangle_0 = p$ for $i > 1$, $\langle e_1,e_1\rangle_0 = 2p$, $\langle  f_1,f_1\rangle_0 = -2 \lambda^2p$, and $\la v, w\ra_0=0$ for all $(v,w)\in \{e_1,\dots,e_n, f_1,\dots,f_n\}^2$ such that $(v,w)\neq (e_i,f_i),(f_i,e_i), i>1$ or $(e_1,e_1),(f_1,f_1)$.
\end{lemma}
\begin{proof}
The assertion follows from Theorem 3.4 and the proof of Lemma 3.15 in \cite{Ogus79}.
\end{proof}

\begin{lemma}[Ogus]\label{vforbL}
Fix the $\bZ_p$-quadratic space $\cL_0$ as in \Cref{L''_p0}. All possible $\bL$ attached to a supersingular point $P$ with $(\bL_0)^{\varphi=1}=\cL_0$ are given by $\bL_0=\Span_W\{v,\sigma(v),\dots,\sigma^{n-1}(v)\}+\cL_0\otimes W$, where $v\in \frac{1}{p}(\cL_0\otimes_{\bZ_p}W)$ satisfying the following conditions: 
\begin{enumerate}
    \item $\Span_W\{v, \sigma(v), \dots, \sigma^{2n-1}(v)\}=\frac{1}{p}\cL_0\otimes_{\bZ_p} W$. 
    
    \item $\Span_W\{v, \sigma(v), \dots, \sigma^{n-1}(v)\}$ is isotropic for $\la,\ra_0$.
    
    \item $\langle v,\sigma^n( v)\rangle_0 = 1/p$,
\end{enumerate}
where we use $\sigma$ to denote the action $1\otimes \sigma$ on $\cL_0\otimes_{\bZ_p} K$.
The quadratic form and $\varphi$ action on $\bL_0$ are induced by those on $\cL_0\otimes_{\bZ_p} K$ via $\bL_0\subset \cL_0\otimes_{\bZ_p} K$, where $\varphi$ on $\cL_0\otimes_{\bZ_p} K$ is given by $1\otimes \sigma$. Finally, the set of vectors $\{v, \sigma(v), \hdots, \sigma^{n-1}(v), p\sigma^{n}(v), \hdots, p\sigma^{2n-1}(v) \}$ forms a $W$-basis for $\bL_0$. 
\end{lemma}
\begin{proof}
Consider the inclusion $\cL_0\otimes W \subset \bL_0\subset \frac{1}{p} \cL_0\otimes W$. Recall from \S\ref{par_decomp_cL} that the quadratic form $p^{-1}\langle,\rangle_0$ yields a perfect bilinear form on $\cL_0\otimes W$. By \cite[Theorem 3.5]{Ogus79}, the data of $\bL_0$ is in bijection with the data of an $n$-dimensional subspace $\overline{H} \subset \cL_0\otimes k$ which is isotropic for $p^{-1}\langle,\rangle_0$, where $\overline{H}$ satisfies conditions 3.5.2 and 3.5.3 of \emph{loc. cit.}.\footnote{Ogus proved that the isomorphism classes so called K3 crystal (\cite[Def.~3.1]{Ogus79}) are in bijection with the data in \cite[Thm.~3.5]{Ogus79} described here; indeed, the isomorphism classes of K3 crystals in Ogus sense are isomorphism classes of $\bL$ for supersingular points by \cite{HP}.} Let $H \subset \cL_0\otimes W$ denote any lift of $\overline{H}$ and then the crystal $\bL_0$ corresponding to $\overline{H}$ is defined to be $\frac{1}{p}H + \cL_0\otimes W$. Note that $\bL_0$ only depends on $\overline{H}$ and not on $H$ itself, and that $\overline{H}$ is indeed the kernel of the natural map  $\cL_0 \otimes k \rightarrow \bL_0\otimes k$. 

The discussion in the paragraph above Theorem 3.21 in \cite{Ogus79} implies that there exists a vector $e'\in \overline{H}$ such that $\{e',\hdots, \sigma^{n-1}(e')\}$ yields a basis of $\overline{H}$, and the set $\{e',\hdots, \sigma^{2n-1}(e')\}$ is a basis of $\cL_0 \otimes k$. Note that although the discussion in \emph{loc. cit.} is in the context of $ \varphi^{-1}(\overline{H}) \subset \cL_0\otimes W$ and not $\overline{H}$, everything applies to our setting too, by defining $e'$ to be $\sigma(e)$, where $e$ is as in \cite[p.~33]{Ogus79}, and note that $\varphi(\cL_0)=\cL_0$. 

A straightforward application of Hensel's lemma yields a specific choice of an \emph{isotropic} $n$-dimensional $H_0 \subset \cL_0\otimes W$ along with a vector $\tilde{e}'$, with the property that $H_0$ and $\tilde{e}'$ reduce to $\overline{H}$ and $e'$ mod $p$ such that $H_0$ is the $W$-span of $\tilde{e}', \sigma(\tilde{e}'), \hdots, \sigma^{n-1}(\tilde{e}')$. It then follows that the $W$-span of $\tilde{e}', \sigma(\tilde{e}'), \hdots, \sigma^{2n-1}(\tilde{e}')$ equals $\cL_0 \otimes W$. By replacing $\tilde{e}'$ by an appropriate $W^{\times}$-multiple, we may also assume that $\frac{1}{p}\langle\tilde{e}',\sigma^{n}(\tilde{e}')\rangle_0 = 1$. The the lemma follows by defining $v = \frac{1}{p}\tilde{e}'$. 
%
\end{proof}

\begin{lemma}\label{cyclicfrob}
Set $v_i=\sigma^{i-1}(v), i=1,\dots, 2n$ for the vector $v$ in \Cref{vforbL}. Then there exist vectors $w_1,\dots, w_n\in \bL_0$ such that
\begin{enumerate}
    \item $v_1,\dots,v_n, w_1\dots,w_n$ form a $W$-basis of $\bL_0$;
    \item The Gram matrix of $\la,\ra_0$ with respect to this basis is $\left[
\begin{array}{c|c}
0&I\\
\hline
I&0
\end{array}
\right]$;
\item The Frobenius $\varphi$ on $\bL_0$ with respect to this basis is of form $B_0\sigma$, where\footnote{All empty entries in the matrix are $0$.} \[
B_0=
\left[
\begin{array}{cccc|cccc}
0&\hphantom{\textrm{-}b_1}&\hphantom{-b_1}&\hphantom{-b_1}&\hphantom{-b_1}&&& p\\
1&&&&&&&pb_1\\
&\ddots&&&&&&\vdots\\
&&1&&&&&pb_{\scaleto{n-1}{4pt}}\\
\hline
&&&p^{-1}&-b_1&\hdots&-b_{\scaleto{n-1}{4pt}}&0\\
&&&&1&&&0\\
&&&&&\ddots&&\vdots\\
&&&&&&1&0\\
\end{array}
\right] \text{ with } b_i\in W.
\]
\end{enumerate}
\end{lemma}
\begin{proof}
By \Cref{vforbL}(1), $\{v_1,\dots,v_n, pv_{n+1},\dots, pv_{2n}\}$ is a basis of $\bL_0$ over $W$. By \Cref{vforbL}(2)(3) and the fact that $\la \varphi(x),\varphi(y)\ra_0=\sigma(\la x, y\ra_0)$, we have that $\la v_i, pv_j\ra_0=0$ for $j\leq i+n-1$ and $\la v_i, pv_{i+n} \ra_0=1$; thus 
by modifying $pv_{n+1}, \dots, pv_{2n}$ by an upper-unipotent matrix, we obtain $w_1,\dots, w_n$ satisfying condition (2). Moreover, the left half of $B_0$ in condition (3) also follows from the definition of $w_i$. 

We now consider the top-right block of $B$. To deduce that the first $n-1$ columns of this block vanish, (2) shows that it suffices to prove $\langle \varphi(w_i),w_j\rangle_0=0$ for $1\leq i\leq n-1$ and $1\leq j \leq n$. By definition, the $w_i, 1\leq i \leq n-1$ are $W$-linear combinations of $pv_{n+1} = p\varphi^n(v), \hdots, pv_{2n-1}=p\varphi^{n-2}(v) $ and thus $\varphi(w_i)$ is contained in $\Span_W\{ pv_{n+1},\hdots, pv_{2n}\}=\Span_W\{w_1,\dots, w_n\}$. Since $\Span_W\{w_1,\dots, w_n\}$ is isotropic by \Cref{vforbL}(2), then $\langle \varphi(w_i),w_j\rangle_0$ for $1\leq i\leq n-1$ and $1\leq j \leq n$ as required. In order to prove that the last column of this block is as claimed in the lemma, it suffices to prove that $\langle \varphi(w_n) , w_1\rangle_0 = \sigma(\langle w_n , \varphi^{-1}(w_1)\rangle_0) =p $ and $p\mid \langle \varphi(w_n) , w_j\rangle_0 = \sigma(\langle w_n , \varphi^{-1}(w_j)\rangle_0)$ for $j\leq n$. Note that $\varphi^{-1}(w_1) = pv_n$ and then the first equality follows. For the rest, note that $\varphi$ gives a $\sigma$-linear endomorphism of $\cL_0\otimes W$ and $p\mid \la, \ra_0$ on $\cL_0$, thus $w_1\hdots w_n, \varphi(w_n)\in \cL_0\otimes W$ and $p\mid \langle \varphi(w_n),w_j \rangle_0 $ for all $j\leq n$.

 Similarly, for the bottom-right part of $B$, it suffices to show that $\la \varphi(w_i)-w_{i+1}, v_j\ra_0=0,$  $p\la \varphi(w_i), v_1\ra_0=-\la \varphi(w_n),w_{i+1} \ra_0, \, \forall 1\leq i\leq n-1, 2\leq j \leq n$ and $\la \varphi(w_n), v_j\ra_0=0, \, \forall 1\leq j \leq n$. Note that $\la \varphi(w_i)-w_{i+1}, v_j\ra_0=\la \varphi(w_i), v_j\ra_0 -\la w_{i+1}, v_j \ra_0=\la w_i, v_{j-1}\ra_0 -\la w_{i+1}, v_j \ra_0=0$ by condition (2). Then $\varphi(w_{i})=w_{i+1}+a_i w_1$ for some $a_i\in W$. Thus $w_{i+1}=\varphi(w_i-\sigma^{-1}(a_i)pv_n)$ and then $\la \varphi(w_n), w_{i+1}\ra_0=\sigma(\la w_n, w_i-\sigma^{-1}(a_i)pv_n \ra_0)=-pa_i$; in other words, $p\la \varphi(w_i), v_1\ra_0=-\la \varphi(w_n),w_{i+1} \ra_0 $. Moreover, $\la \varphi(w_n), v_j \ra_0=\sigma(\la w_n, v_{j-1} \ra_0)=0$ for $j\geq 2$. For $\la \varphi(w_n), v_1 \ra_0=:c$, by the above discussion, $\varphi(w_n)=pv_1+pb_1 v_2 + \cdots + pb_{n-1} v_n +c w_1$ and thus $\la \varphi(w_n), \varphi(w_n)\ra_0=2pc$; on the other hand, $\la \varphi(w_n), \varphi(w_n)\ra_0=\sigma(\la w_n, w_n \ra_0)=0$ and then $c=0$, which finishes the proof of the lemma.
\end{proof}

\begin{para}\label{change-basis}
Let $S_0$ denote the change-of-basis matrix from $\{e_i, f_i\}_{i=1}^n$ to $\{v_i, w_i\}_{i=1}^n$ in \Cref{cyclicfrob}. More precisely, $S_0\in M_{2n}(K)$ whose first (resp. last) $n$ columns are the coordinates of $v_i$ (resp. $w_i$) in terms of the basis $\{e_i,f_i\}_{i=1}^n$. For the simplicity of computations in \S\ref{sec_decay_sg}, let $S'_0$ the change-of-basis matrix from $\{e_i, f_i\}_{i=1}^n$ to $\{pv_i, w_i\}_{i=1}^n$. From the proof of \Cref{cyclicfrob}, $\Span_W\{e_i, f_i\}_{i=1}^n=\Span_W\{pv_i, w_i\}_{i=1}^n$; thus $S_0'\in \GL_{2n}(W)$. Moreover, by definition, $S_0=S_0'\left[
\begin{array}{c|c}
p^{-1}I&0\\
\hline
0&I
\end{array}
\right]$.
\end{para}

\begin{para}\label{embedding}
We now describe $\bL_1$ and $\cL_1$ defined in \S\S\ref{par_decomp_cL},\ref{par_decomp_bL}. Recall that $\bL_1=\cL_1\otimes W$, the Frobenius $\varphi$ on $\bL_1$ is given by $1\otimes \sigma$ and the quadratic form on $\bL_1$ is also induced by the one on $\cL_1$, so we only need to classify $\cL_1$.  Unlike $\cL_0$, which is completely determined by $t_P$ (see \Cref{L''_p0}), the $\bZ_p$-quadratic lattice $\cL_1$ depends on $\dim L$ and $\disc L$ (see \cite[Thm.~3.4]{Ogus79} and \cite[\S 4.3.1]{HP}). Since $\cL_1$ is self-dual, we have the following three cases:\footnote{Comparing to \cite[\S 3]{MST}, \S 3.2.1 in \emph{loc.~cit.~}is a special case of the split even dimensional case, \S 3.2.2 in \emph{loc.~cit.~}is a special case of the non-split even dimensional case, and \S 3.3 in \emph{loc.~cit.~}is a special case of the odd dimensional case.}
\begin{enumerate}
    \item $\dim_{\bZ_p} \cL_1=2m$ and there is an $m$-dimensional isotropic subspace of $\cL_1$ over $\bZ_p$. We call this the \emph{split} case. 
    \item $\dim_{\bZ_p} \cL_1=2m$ and there does \emph{not} exist an $m$-dimensional isotropic subspace of $\cL_1$ over $\bZ_p$.
    \item $\dim_{\bZ_p} \cL_1$ is odd.
\end{enumerate}
Note that for cases (2)(3), one may always embed $\cL_1$ into a split $\bZ_p$-quadratic lattice of larger dimension. Therefore, we deal exclusively with the \emph{split} case and we will remark in the proofs of the decay lemmas in \S\S\ref{sec_decay_sg}-\ref{sec_decay_ssp} that by the above embedding trick, the computation in the split case will also prove the decay lemmas in all other cases. We use $\{e'_i,f'_i\}_{i=1}^m$ to denote a $\bZ_p$-basis of $\cL_1$ in the split case such that the Gram matrix with respect to this basis is $\left[
\begin{array}{c|c}
0&I\\
\hline
I&0
\end{array}
\right]$.
\end{para}

\subsection*{Description of $\bbL_{\cris}$ at the formal neighborhood}
\begin{para}\label{Frob_vw}
Following \cite[\S\S 1.4-1.5]{Kisin}, we will describe the formal neighborhood of the Shimura variety at the supersingular point $P$, and also compute the $F$-crystal $\bbL_\cris$ over this formal neighborhood (see also \cite[\S 3.1.5, \S 3.2.1]{MST}). 
We first summarize Kisin's description in abstract terms, before providing an explicit description of the $F$-crystal in terms of the coordinates provided earlier in this section. 

Recall from \S\ref{par_decomp_cL} that the quadratic form $\bbQ $ on $\bL$ is compatible with the Frobenius $\varphi$ on $\bL$; moreover, $\bbL_\dR$ defined in \S\ref{def_bbL} admits the Hodge filtration and by the canonical de Rham-crystalline comparison, $\bL\otimes_W k$ is also equipped with a filtration and we call it the $\bmod\, p$ Hodge filtration. Let $\mu: \bG_{m,W} \rightarrow \SO(\bL,\bbQ)$ denote any co-character (which we shall refer to as ``the Hodge co-character'') whose mod $p$ reduction induces the above filtration. Let $U$ denote the opposite unipotent in $\SO(\bL,\bbQ)$ with respect to $\mu$, and let $\Spf R = \widehat{U} $ denote the completion of $U$ at the identity section. Pick $\sigma: R\rightarrow R$ to be a lift of the Frobenius endomorphism on $R \bmod p$.  Let $u$ be the tautological $R$-point of $U$. 

Then, by \cite[\S\S 1.4, 1.5]{Kisin} the complete local ring of the Shimura variety at $P$ is isomorphic to $\Spf R$. The $F$-crystal $\bbL_{\cris}(R)$ equals $\bL \otimes_W R$ as an $R$-module, and the Frobenius action on $\bbL_{\cris}(R)$, denoted by $\Frob$, is given by $\Frob = u\circ (\varphi \otimes \sigma)$.

We will now provide an explicit description in terms of coordinates of the above objects. By \Cref{cyclicfrob}(3) and Mazur's theorem on determining $\bmod\, p$ Hodge filtration using $\varphi$ (see for instance \cite[p.~411]{Ogus}), the $\bmod \, p$ Hodge filtration on $\bL\otimes_W k$ is given by \[\Fil^1\bL\otimes_W k=\Span_k\{\bar{w}_n\}, \Fil^0\bL\otimes_W k=\Span_k\{\bar{v}_i,\bar{w}_j,\bar{e}'_l,\bar{f}'_l\}_{i=1,\dots,n-1,j=1,\dots,n,l=1,\dots, m}, \Fil^{-1}\bL\otimes_W k=\bL\otimes_W k,\]
where $\bar{v}_i,\bar{w}_j,\bar{e}'_l, \bar{f}'_l$ denote the reduction of $v_i, w_j, e'_l, f'_l \bmod p$. Thus, with respect to the basis $\{v_i,w_i,e'_j,f'_j\}_{i=1,\dots,n, j=1,\dots,m}$, we choose the Hodge cocharacter $\mu: \bG_{m,W}\rightarrow \SO(\bL,\bbQ)$ in the local Shimura datum to be

\[ 
\mu(t) = \left[
\begin{array}{cccc|cccc|ccc}
1&&&&&&&&&&\\
&\ddots&&&&&&&&&\\
&&1&&&&&&&&\\
&&&t^{-1}&&&&&&&\\
\hline
&&&&1&&&&&&\\
&&&&&\ddots&&&&&\\
&&&&&&1&&&&\\
&&&&&&&t&&&\\
\hline
&&&&&&&&1&&\\
&&&&&&&&&\ddots&\\
&&&&&&&&&&1
\end{array}
\right].
\]

Moreover, there exist local coordinates $\{x_i,y_i,x'_j,y'_j\}_{i=1,\dots,n-1, j=1,\dots,m}$ such that the complete local ring $\widehat{\cO}_{\Sh,P}$ of $\Sh$ at $P$ is isomorphic to $\Spf R$, where $R= W[[x_i,y_i,x'_j, y'_j]]_{i=1,\dots,n-1, j=1,\dots,m}$. We define $\sigma: R\rightarrow R$, the operator that restricts to the usual Frobenius element on $W$ and which lifts the Frobenius endomorphism on $R\mod p$, to be  $\sigma(x_i)=x_i^p, \sigma(y_i)=y_i^p, \sigma(x'_j)=(x'_j)^p, \sigma(y'_j)=(y'_j)^p$.
The tautological point of the opposite unipotent in $\SO(\bL,\bbQ)$ with respect to $\mu$ has the following description in terms of our coordinates:

\[u=I+\left[
\begin{array}{cccc|cccc|cccccc}\label{unip}
&&&&&&&-y_1&&&&&&\\
&&&&&&&\vdots&&&&&&\\
&&&&&&&-y_{n-1}&&&&&&\\
x_1&\hdots&x_{n-1}&0&y_1&\hdots&y_{n-1}&Q&x'_1&\hdots&x'_{m}&y'_1&\hdots&y'_m\\
\hline
&&&&&&&-x_1&&&&&&\\
&&&&&&&\vdots&&&&&&\\
&&&&&&&-x_{n-1}&&&&&&\\
&&&&&&&0&&&&&&\\
\hline
&&&&&&&-y'_1&&&&&&\\
&&&&&&&\vdots&&&&&&\\
&&&&&&&-y'_m&&&&&&\\
&&&&&&&-x'_1&&&&&&\\
&&&&&&&\vdots&&&&&&\\
&&&&&&&-x'_m&&&&&&
\end{array}
\right],
\]
where $\displaystyle Q = -\sum_{i=1}^{n-1} x_i y_i - \sum_{j=1}^m x'_jy'_j$. 

The Frobenius action $\Frob$ on $\bbL_\cris(R)=\bL\otimes_W R$ is given by $\Frob = u\circ(\varphi \otimes \sigma)$. Thus, with respect to the $R$-basis $\{v_i\otimes 1, w_i\otimes 1, e'_j \otimes 1, f'_j\otimes 1\}_{i=1,\dots, n, j=1,\dots, m}$, we have that $\Frob=(uB)\circ \sigma$, where $B=\left[
\begin{array}{c|c}
B_0&0\\
\hline
0&I
\end{array}
\right]$, and $B_0$ is given in Lemma \ref{cyclicfrob}.
\end{para}

\subsection*{Equation of 
the non-ordinary locus}
We now compute the local equation of the non-ordinary locus in a formal neighborhood of a supersingular point. Recall that we have the Hodge cocharacter $\mu$, whose mod $p$ reduction induces the mod $p$ Hodge filtration on $\bL_{\cris,P}(k) = \bL\otimes_W k$. This induces the Hodge filtration on $\bL_{\cris}(R\otimes_W k) = \bL\otimes_W (R\otimes_W k)$, given by \[\Fil^i(\bL_{\cris}(R\otimes_W k)) = \Fil^i(\bL_{\cris}(k))\otimes_k (R\otimes_W k), i=-1,0,1.\] As in \cite[\S 3.4]{MST}, we note that $p\Frob$ induces a map $\gr_{-1}\bL_{\cris}(R\otimes_W k) \rightarrow \gr_{-1}\bL_{\cris}(R\otimes_W k)$. Ogus proved the following result.

\begin{lemma}[Ogus]
For a supersingular point $P$,
the non-ordinary locus (over $k$) in the formal neighborhood of $P$ is given by the equation \[p\Frob|_{\gr_{-1}\bL_{\cris}(R\otimes_W k)}=0.\]
\end{lemma}
See \cite{Ogus01}[Prop. 11 and p 333-334] (or \cite{MST}[Lemma 3.4.1] which elaborates on \cite{Ogus01}). 

\begin{corollary}\label{eq-nonord}
For a supersingular point $P$, the non-ordinary locus (over $k$) in the formal neighborhood of $P$ is given by the equation $Q=0$ if $P$ is superspecial; otherwise, the equation is given by $y_1=0$.
\end{corollary}
\begin{proof}
In what follows, the number $n$ is as in Lemma \ref{cyclicfrob}, i.e., $2n = t_P = \dim_W \bL_0$ and we follow the notation in \S\ref{Frob_vw}. The space $\gr_{-1}\bL_{\cris}(R\otimes_W k)$ is spanned by $\bar{v}_n$. We use description of $\Frob = (uB)\circ\sigma$ (from Lemma \ref{cyclicfrob} and the explicit description of $u$ in \S\ref{Frob_vw}) to see that the map $p\Frob: \gr_{-1}\bL_{\cris}(R\otimes_W k)\rightarrow \gr_{-1}\bL_{\cris}(R\otimes_W k)$ has the explicit description
 \[p\Frob(\bar{v}_n) = Q\bar{v}_n  \text{ if }n=1;\,
    p\Frob(\bar{v}_n) = y_1 \bar{v}_n \text{ if }n>1.\]
The result now follows from the fact that $P$ is superspecial if and only if $t_P = 2$ if and only if $n = 1$. 
\end{proof}

\begin{para}\label{Frob_ef}
In order to compute the powers of $\Frob$ in the proofs of the decay lemmas later, we describe $\Frob$ with respect to the $K$-basis $\{e_i,f_i,e'_j,f'_j\}_{i=1,\dots,n,j=1,\dots,m}$ of $\bL\otimes_W K$. Let $S=\left[
\begin{array}{c|c}
S_0&0\\
\hline
0&I
\end{array}
\right], S'=\left[
\begin{array}{c|c}
S'_0&0\\
\hline
0&I
\end{array}
\right]$, where $S_0,S_0'$ are defined in \S\ref{change-basis} and thus $S'\in \GL_{2n+2m}(W)$. Then by definition, $B=S^{-1}\sigma(S)$.

We view $\{e_i,f_i,e'_j,f'_j\}_{i=1,\dots,n,j=1,\dots,m}$ as an $R[1/p]$-basis of $\bbL_\cris(R)\otimes_R R[1/p]$ and then $\Frob$ is given by $S(uB)\sigma(S^{-1})\circ \sigma =SuS^{-1}\circ \sigma=S'u'(S')^{-1} \circ \sigma$, where 
\[u'=I+\left[
\begin{array}{cccc|cccc|cccccc}\label{unip}
&&&&&&&-y_1/p&&&&&&\\
&&&&&&&\vdots&&&&&&\\
&&&&&&&-y_{n-1}/p&&&&&&\\
x_1&\hdots&x_{n-1}&0&y_1/p&\hdots&y_{n-1}/p&Q/p&x'_1/p&\hdots&x'_{m}/p&y'_1/p&\hdots&y'_m/p\\
\hline
&&&&&&&-x_1&&&&&&\\
&&&&&&&\vdots&&&&&&\\
&&&&&&&-x_{n-1}&&&&&&\\
&&&&&&&0&&&&&&\\
\hline
&&&&&&&-y'_1&&&&&&\\
&&&&&&&\vdots&&&&&&\\
&&&&&&&-y'_m&&&&&&\\
&&&&&&&-x'_1&&&&&&\\
&&&&&&&\vdots&&&&&&\\
&&&&&&&-x'_m&&&&&&
\end{array}
\right].
\]
\end{para}

\section{Decay for non-superspecial supersingular points}\label{sec_decay_sg}
The goal of this and the next section is to prove that, at supersingular points, special endomorphisms ``decay rapidly'' in the sense of \cite[Def.~5.1.1]{MST}, which we will recall below.  

Throughout these sections, $k=\bar{\bF}_p$, $W=W(k)$, $K=W[1/p]$. We focus on the behavior of the curve $C$ in \Cref{thm_int} in a formal neighborhood of a supersingular point $P$, so we may let $C = \Spf k[[t]]$ denote a generically ordinary formal curve in $\Sh_k$ which specializes to $P$. In this section, we will focus on the case when $P$ is non-superspecial and we treat the superspecial case in \S\ref{sec_decay_ssp}.

Let $\cA/k[[t]]$ denote the pullback of the universal abelian scheme $\cA^\univ$ over the integral model $\Sh$ of the GSpin Shimura variety via $\Spf k[[t]] \rightarrow \Sh_k$ and let $A$ denote $\cA \bmod t$, and we will consider the $p$-divisible groups $\cA[p^{\infty}], A[p^\infty]$ associated to $\cA,A$. Let $h$ denote the $t$-adic valuation of the local equation defining the non-ordinary locus given in \Cref{eq-nonord}. Recall from \S\ref{par_decomp_cL}, $\cL$ is the lattice of special endomorphisms of $A[p^\infty]$.

\begin{defn}[{\cite[Def.~5.1.1]{MST}}]\label{decaydef}
We say that $w\in \cL$ \emph{decays rapidly} if for every $r\in \bZ_{\geq 0}$, the special endomorphism $p^r w$ does not lift to an endomorphism of $\cA[p^\infty]$ modulo $t^{h_r+1}$, where $h_r:=[h(p^r+\cdots+1+1/p)]$. We say that a $\bZ_p$-submodule of $\cL$ decays rapidly if every primitive vector in this submodule decays rapidly.
\end{defn}

The main theorem of this section is the following:
\begin{theorem}[The Decay Lemma]\label{genericdecay}
There exists a rank $2$ saturated $\bZ_p$-submodule of $\cL$ which decays rapidly.
\end{theorem}

Theorem \ref{cor_decay_sg} follows directly from the Decay Lemma:
\begin{proof}[Proof of Theorem \ref{cor_decay_sg}]
We first note that $\cL_n$, the lattice of special endomorphisms of $\cA[p^{\infty}] \bmod t^n $, is precisely $L_n \otimes \bZ_p$. Upon choosing a basis of $\cL$ that extends a basis of the submodule that decays rapidly (which we may do, as this submodule is saturated in $\cL$), we see that the index $|\cL/\cL_n|$ of $\cL_n$ in $\cL$ is at least $p^{2r+2}$ if $n\geq h_r+1$. The corresponding statements for $L_n$ now follow directly.  
\end{proof}

\begin{para}
We first give an indication as to why such the reader should expect a statement along these lines to hold. Note that in the mixed characteristic setting, namely while deforming from $k$ to $W(k)$, applying Grothendieck--Messing theory yields that if a special endomorphism $\alpha$ lifts mod $p^n$ but not mod $p^{n+1}$, then the special endomorphism $p\alpha$ would lift mod $p^{n+1}$ but \emph{not} $p^{n+2}$ (see \cite[Lemma 4.1.2]{ST20}). However, Grothendieck--Messing theory is inherently limited in the equicharacteristic $p$ setting, and the bounds it yields are worse than the bounds it yields in the mixed characteristic setting. 

We illustrate this with the following example. Let $\pdiv $ denote a $p$-divisible group over $\Spec k[t]/(t^a)$, suppose that $\alpha$ is any endomorphism of $\pdiv$, and suppose that $\pdiv'$ over $\Spec k[t]/(t^{pa})$ is a $p$-divisible group that deforms $\pdiv$. We claim that the endomorphism $p\alpha$ deforms to $\pdiv'$ regardless of how $\alpha$ behaves. Indeed, let $\bD$ denote the Dieudonn\'e crystal of $\pdiv/ \Spec k[t]/(t^a)$. As the map $\Spec k[t]/(t^{a}) \rightarrow \Spec k[t]/(t^{pa})$ is naturally equipped with a divided powers structure (and this is the key point in this observation), we may evaluate the Dieudonn\'e crystal $\bD$ at $\Spec k[t]/(t^{pa})$, and Grothendieck--Messing theory implies that the choice of deformation $\pdiv'$ is equivalent to the choice of a filtration of $\bD(\Spec k[t]/(t^{pa}))$ which is compatible with the filtration on $\bD(\Spec k[t]/(t^a))$ given by $\pdiv$. This corresponds to $\Fil \subset \bD(\Spec k[t]/(t^{pa}))$, which is a free $k[t]/(t^{pa})$ sub-module of $\bD(\Spec k[t]/(t^{pa}))$, which itself is a free $k[t]/(t^{pa})$-module. Moreover, any endomorphism $\beta$ of $\pdiv$ induces an endomorphism of the crystal $\bD$, and therefore induces an endomorphism of $\bD(\Spec k[t]/(t^{pa}))$. Finally, $\beta$ deforms to an endomorphism of $\pdiv'$ if and only if $\beta(\Fil) \subset \Fil$. Given that $\alpha$ induces an endomorphism of $\bD(\Spec k[t]/(t^{pa}))$ (which need not preserve $\Fil$), it follows that $p\alpha$ induces the zero map on $\bD(\Spec k[t]/(t^{pa}))$, and thus tautologically preserves $\Fil$ whether or not $\alpha$ does. Therefore, it follows that if $\alpha$ is an endomorphism of $\pdiv$ over $k[t]/(t^a)$, then $p\alpha$ lifts to any deformation of $\pdiv$ to $\Spec k[t]/(t^{pa})$, which suggests that it is not possible to expect a much faster rate of decay than defined in Definition \ref{decaydef}. 

We now work in the setting of a $p$-divisible group $\cA[p^{\infty}]/ k[[t]]$. Let $\alpha$ denote an endomorphism of $\cA[p^{\infty}] \bmod t$, that extends to an endomorphism of $\cA[p^{\infty}] \bmod t^{a}$, but not $t^{a+1}$. The example considered above implies that $p\alpha$ extends to an endomorphism mod $t^{pa}$. However, Grothendieck--Messing theory cannot be naively applied to find an effective integer $b$ (in terms of $a$) which has the property that $p\alpha$ does not extend to an endomorphism of $\cA[p^\infty] \bmod t^b$. 
Therefore, in order to prove Theorem \ref{genericdecay}, we use Kisin's description of the $F$-crystal $\bbL_\cris$, which controls the $t$-adic deformation of the special endomorphisms of $\cA[p^\infty] \bmod t$ -- see the next paragraph for a sketch of how we proceed. 
\end{para}

\begin{para}
We give a rough idea of the proof of \Cref{genericdecay} here (see \S\S\ref{Yn}-\ref{dJ+Kisin} for details and references); we also provide a toy example of the explicit computation in this section. The reader should feel free to read this description and skip the details of our proof on a first reading. For a given special endomorphism $w\in \cL$, if it extends to an endomorphism of $\cA[p^\infty] \bmod t^r$, then this extension must be the restriction to $k[t]/t^r$ of the horizontal section $\widetilde{w}$ with respect to the natural connection on the $F$-crystal $\bbL_\cris$. On the other hand, given such a $w$, this horizontal section $\widetilde{w}$ is given by $\lim_{n\rightarrow \infty} \Frob^n(w)$, where as in \S\ref{Frob_vw}, $\Frob$ is the $\sigma$-linear Frobenius action on $\bbL_\cris$. Note that the expression $\lim_{n\rightarrow \infty} \Frob^n(w)$ is in general not an element in $\bbL_\cris(R)=\bL\otimes_W R$; indeed, as a $\bL$-valued power series in local coordinates of the formal neighborhood, its coefficient lies in $\bL\otimes_W K$ and the $p$-adic valuation of these coefficients may go to $+\infty$. Thus the proof of \Cref{genericdecay} boils down to study the $p$-adic valuation of the coefficients of $\widetilde{w}$, which requires an explicit computation of $\Frob^n$. Here is a toy model of such computation (we refer the reader to \S\ref{Frob_ef_ssp} for why the following is a toy model). Consider $\Frob=(I+F)\circ \sigma$ with respect to a $\varphi$-invariant basis of $\bL$, where $F=\begin{bmatrix} xy/p & x/p \\ y & 0\end{bmatrix}$, $R=W[[x,y]]$, $\sigma(x)=x^p, \sigma(y)=y^p$ and when we restrict ourselves to $C$, we plug in $x,y$ by certain power series $x(t), y(t)\in W[[t]]$, which are chosen based on $C\rightarrow \Spf k[[x,y]]$. Thus $\prod_{n=1}^\infty (I+F)\circ \sigma$ is an infinite summation of products of $F^{(i)}:=\sigma^{i}(F)$. Consider $w=[1,0]^T$ (with respect to the chosen basis of $\bL$), then a direct computation of matrix products implies that for the first coordinate of $\widetilde{w}$, among all the terms with $p$-adic valuation $-(r+1)$, the unique term with minimal $t$-adic valuation is $\prod_{i=1}^{r+1} \sigma^{i-1}(xy/p)=(xy)^{1+p+\cdots+p^{r}}/p^{r+1}$. This observation allows us to prove the Decay Lemma.
\end{para}

\subsection*{The setup}
The setup and the first reduction steps in the proof of \Cref{genericdecay} is the same as that in the proof of \cite[Thm.~5.1.2]{MST} in \cite[\S 5.1]{MST}. We briefly introduce the notation for the proof of \Cref{genericdecay} here and the reader may see \cite{MST} for more details.

\begin{para}\label{Yn}
Recall from \S\ref{Frob_vw} that $\widehat{\cO}_{\Sh,P}=\Spf W[[x_1,\dots, x_{n-1}, y_1,\dots, y_{n-1}, x'_1,\dots, x'_m,y'_1,\dots,y'_m]]$ when $\cL_1$ is split. Since $P$ is non-superspecial, we have $n\geq 2$ through out this section. 

The formal curve $C$ gives rise to the tautological ring homomorphism 
\[
W[[x_1,\dots, x_{n-1}, y_1,\dots, y_{n-1}, x'_1,\dots, x'_m,y'_1,\dots,y'_m]] \rightarrow k[[t]],
\]
and we denote by $x_i(t)$ (respectively $x'_i(t), y_i(t), y'_i(t)$) the images of $x_i$ (respectively $x'_i,y_i,y'_i$) in $k[[t]]$. For each $x_i(t)$ (respectively $x'_i(t), y_i(t), y'_i(t)$), let $X_i(t)$ (respectively $X'_i(t), Y_i(t), Y'_i(t)$) denote power series in $W[[t]]$ whose coefficients are the Teichmuller lifts of the coefficients of $x_i(t)$ (respectively $x'_i(t),y_i(t),y'_i(t)$). We define $\displaystyle Y_n(t) = -\sum_{i=1}^{n-1} X_i(t)Y_i(t) - \sum_{j=1}^m X'_j(t)Y'_j(t)$, and define $\displaystyle Y_{n+1} (t) =- \sum_{i=1}^m(X'_i(t)(Y'_i(t))^p + (X'_i(t))^pY'_i(t))$. 
Let $a_i = v_t(Y_i(t))$ for $1\leq i\leq n+1$, where $v_t$ denotes the function of taking $t$-adic valuation. By \Cref{eq-nonord}, since $P$ is a non-superspecial supersingular point, the local equation of the non-ordinary locus is given by $y_1(t) = 0$  and hence $h =v_t(y_1(t))=v_t(Y_1(t))= a_1$.
\end{para}

\begin{para}\label{dJ+Kisin}
We now relate the lift of special endomorphisms to explicit computations of powers of the Frobenius matrix given in \S\ref{Frob_ef}. For details, see \cite[Proof of Thm.~5.1.2 assuming Prop.~5.1.3]{MST}. For $s\in \bZ_{\geq 0}$, let $D_s$ denote the $p$-adic completion of the PD enveloping algebra of the ideal $(t^s, p)$ in $W[[t]]$. By de Jong's theory \cite[\S 2.3]{dJ95}, if $w\in \cL$ lifts to a special endomorphism of $\cA \bmod t^s$, then it gives rise to a horizontal section in the Dieudonn\'e module $\bbL_\cris(D_s)$. Thus, in order to find the largest $s$ such that $w$ lifts to $k[t]/(t^s)$, we first compute the horizontal section $\tilde{w}$ passing through $w$ and then study the $p$-adic integrality of $\tilde{w}$.

Here we recall the construction of $\tilde{w}$ following \cite[\S 1.5.5]{Kisin} and the rest of this section is devoted to the study of the $p$-adic integrality. Let $F$ denote the matrix $S'(u'-I)(S')^{-1}$ in \S\ref{Frob_ef} with $X_i(t)$ (resp.~$X'_j(t), Y_i(t), Y'_j(t)$) substituted in place of $x_i$ (resp.~$x'_j,y_i,y'_j$); with respect to the basis $\{e_i,f_i,e'_j,f'_j\}$, $\Frob=(I+F)\circ \sigma$, where $\sigma$ is defined in \S\ref{Frob_vw}. Let $F^{(i)}$ denote the $i$-th $\sigma$-twist of $F$; more precisely, $F^{(i)}$ is given by $\sigma^{i}(S'(u'-I)(S')^{-1})$ with $X_i(t)$ (resp.~$X'_j(t), Y_i(t), Y'_j(t)$) substituted in place of $x_i$ (resp.~$x'_j,y_i,y'_j$) (here we first do the $\sigma$-twist and then substitute). Let $\displaystyle \F = \prod_{i=0}^{\infty}(I + F^{(i)})\in M_{2n+2m}(K[[t]])$. This product is well-defined and the $\bQ_p$-span of the columns of $F_\infty$ are vectors in $\bbL_\cris(R)\otimes_R K[[x_i,y_i,x'_j,y'_j]]$ which are $\Frob$-invariant and horizontal. Thus the horizontal section $\tilde{w}$ over $W[[t]]$ is given by $F_\infty w$, where we write $w$ as a column vector with respect to the basis $\{e_i,f_i, e'_j, f'_j\}$. The horizontal section over $D_s$ is given by natural pullback of the one over $W[[t]]$. In order to show that $w$ decays rapidly, it suffices to show that for every $r$, we have that $p^r\tilde{w}$ does not lie in $\bbL_\cris (D_{h_r+1})$. Thus in what follows, we will find the term with minimal $t$-adic valuation in the expansion of $\tilde{w}=F_\infty w$ among all terms with $p$-adic valuation $-r$. 
\end{para}

\subsection*{The terms in $\F$ with minimal $t$-adic valuation among ones with fixed $p$-adic valuation}
In order to prove \Cref{genericdecay}, it suffices to work with $F_\infty(1)$, the top-left $2n\times 2n$ block of $F_\infty$ (see the first paragraph of the proof of Theorem \ref{genericdecay} right after \Cref{lastgenericdecay} for details) and in what follows, we compute $F_\infty(1)$ explicitly. 

\begin{para}\label{Ki}
Let $A$ (resp. $C$, $D$) denote the top-left $2n\times 2n$ (resp. top-right $2n\times m$ and bottom-left $m\times 2n$) block of $u'-I$ in \S\ref{Frob_ef} with $X_i(t)$ (resp.~$X'_j(t), Y_i(t), Y'_j(t)$) substituted in place of $x_i$ (resp.~$x'_j,y_i,y'_j$). For $i<n$, we let $A_i$ denote the $2n\times 2n$ matrix with the $n,n+i$ entry equal to $p^{-1}Y_i(t)$, the $i,2n$ entry equalling $-p^{-1}Y_i(t)$ and all other entries equal to 0. We let $A_n$ denote the matrix with zeros everywhere except for the $n,2n$ entry, which equals $p^{-1}Y_n(t)$. Let $K_i$ equal $S'_0A_i(S'_0)^{-1}$. Let $A_{n+1} = CD^{(1)}$ and $K_{n+1} = S'_0 A_{n+1}((S'_0)^{-1})^{(1)}$. Note that $A_{n+1}$ is a $2n\times 2n$ matrix with zeros everywhere except for the $n,2n$-entry which equals $p^{-1}Y_{n+1}(t)$.

For brevity, let $F(1)$ denote the top-left $2n\times 2n$ block of $F$. We have $F(1) = \displaystyle \sum_{i=1}^{n}  K_i + B(x)$, where $B(x)$ involves only the $X_i$ and has $p$-adically integral coefficients since $S'_0\in \GL_{2n}(W)$. Moreover, we observe that $\F(1)$ is made up of sums of finite products of $\sigma$-twists of $B(x)$ and $K_i,i=1,\dots, n+1$.
\end{para}

The following two lemmas identify the nonzero products of $\sigma$-twists of $K_i,i=1,\dots, n+1$.
For brevity, let $R_1\hdots R_{2n}$ denote the rows of the matrix $(S'_0)^{-1}$. Since $(S'_0)^{-1}\in \GL_{2n}(W)$, we have that $\{\overline{R}_i\}_{i=1}^{2n}$ is a basis of $k^{2n}$, where we use $\overline{R}_i$ to denote $R_i \bmod p$.

\begin{lemma}\label{frobactionrows}
We have $\sigma(R_i) = R_{i+1}$ for $n+1\leq i\leq 2n-1$; $\sigma(R_{2n}) = R_1$, $\sigma(R_i) = R_{i+1} - b_iR_1$ for $1\leq i \leq n-1$, and $\sigma(R_n) = R_{n+1} + \sum_{i=1}^{n-1}b_i R_{n+i+1}$.
\end{lemma}
\begin{proof}
Recall from \S\ref{Frob_ef} that $B_0=S_0^{-1}\sigma(S_0)$ and thus by \S\ref{change-basis}, we have $\sigma((S'_0)^{-1})=(B_0')^{-1} (S'_0)^{-1}$, where $B'_0=\left[
\begin{array}{c|c}
p^{-1}I&0\\
\hline
0&I
\end{array}
\right] B_0 \left[
\begin{array}{c|c}
pI&0\\
\hline
0&I
\end{array}
\right]$.
We then obtain the assertions by a direct computation using \Cref{cyclicfrob}(3). 
\end{proof}

\begin{lemma}\label{lintopleft}
Notation as in \S\S\ref{Yn},\ref{Ki}.
\begin{enumerate}
    \item 
    For $i,j,l\in \bZ_{>0}$ such that $i,j,l\leq n + 1$ and $l\leq i$, 
    the matrix $K_iK_j^{(l)}=0$ unless $i=l$. Moreover, the image of $K_iK_j^{(i)}$ is $\Span_K\{v_n\}$ and $K_iK_j^{(i)}=S'_0M$ where $M\in M_{2n}(K)$ is the matrix with its $n\thh$ row being $ Y_iY_j^{(i)}p^{-2}\sigma^{i+j-1}(R_{n+1})$ and all other rows being $0$.
    
    \item For $i_1, i_2,\dots, i_l\in \bZ_{\geq 1}$ such that $i_1,i_2,\dots,i_l\leq n+1$, the image of the matrix $\prod_{j=1}^l K_{i_j}^{(i_1+\cdots+i_{j-1})}$ is $\Span_K\{v_n\}$. Moreover, $\prod_{j=1}^l K_{i_j}^{(i_1+\cdots+i_{j-1})}=S'_0M$ where $M\in M_{2n}(K)$ is the matrix with its $n\thh$ row being $(\prod_{j=1}^l Y_{i_1}^{(i_1+\cdots i_{j-1})})p^{-l}\sigma^{i_1 + i_2 + \cdots +i_{l}-1}(R_{n+1})$ and all other rows being $0$.
\end{enumerate}
\end{lemma}
\begin{proof}
(1) By \S\ref{Ki}, $\ker K_i=\Span_K\{v_1,\dots, v_n, w_1, \dots, w_{i-1}, w_{i+1}, \cdots, w_{n-1}\}$ for $1\leq i<n$, and $\ker K_n=\Span_K\{v_1,\dots, v_n, w_1,\dots, w_{n-1}\}$, $\ker K_{n+1}=\Span_K\{\varphi(v_1),\dots, \varphi(v_n), \varphi(w_1)\hdots \varphi(w_{n-1})\}$. Note that if we view $v_i,w_i$ as vectors in $K^{2n}$ by using the basis $\{e_i,f_i\}_{i=1}^n$, then $\varphi(v_i),\varphi(w_i)$ are just applying $\sigma$ to all coordinates.
    On the other hand, $\im K_j=\Span_K\{v_j,v_n\}$ for $1\leq j\leq n$ and $\im K_{n+1}= \Span_K\{v_n\}$; hence by \Cref{cyclicfrob}(3), for $1\leq l \leq n$, \[\im K_j^{(l)}=\varphi^l(\im K_j)\subset \Span_K\{v_1,\dots, v_n, w_1,\dots,w_l\},\] and $\im K_j^{(n+1)}\subset \Span_K\{\varphi(v_1),\dots,\varphi(v_n),\varphi(w_1),\dots,\varphi(w_n)\}$. 
    Therefore, $K_iK_j^{(l)}=0$ if $l<i$. 
    
    Suppose now that $l=i$. Then $\im K_i K_j^{(i)}=\Span_K\{K_i w_i\}=\Span_K\{v_n\}$ for $i\leq n$ and $\im K_{n+1} K_j^{(n+1)}=\Span_K\{K_{n+1}\varphi(w_n)\}=\Span_K\{v_n\}$. Thus the matrix $M$ has only its $n\thh$ row non-zero.
We now compute the $n\thh$ row of $M$. For $i,j \leq n$, note that $M= A_i (S'_0)^{-1} (S'_0)^{(i)} A_j^{(i)} ((S'_0)^{-1})^{(i)}$; if $i=n+1$ or $j=n+1$, the matrix $M$ is given by the same formula once we replace the $(S'_0)^{-1}$ after $A_i$ or $A_j$ by $((S'_0)^{-1})^{(1)}$. For $j\leq n$, the product $A_j^{(i)} ((S'_0)^{-1})^{(i)}$ has only its $j\thh$ and $n\thh$ rows non-zero (if $j=n$, then only the $n\thh$ row is non-zero), and these rows equal $-Y_j^{(i)}p^{-1}\sigma^{i}(R_{2n})$ and $Y_j^{(i)}p^{-1}\sigma^{i}(R_{n+j})$ respectively (if $j=n$, the row $Y_n^{(i)}p^{-1}\sigma^{i}(R_{2n})$); and $A_{n+1}^{(i)} ((S'_0)^{-1})^{(i+1)}$ only has its $n\thh$ row non-zero and its $n\thh$ row is $Y_{n+1}^{(i)}p^{-1}\sigma^{i+1}(R_{2n})$. Similarly, for $i\leq n$, the $n\thh$ row in the product $A_i(S'_0)^{-1}$ equals $Y_ip^{-1}R_{n+i}$; the $n\thh$ row in $A_{n+1}((S'_0)^{-1})^{(1)}$) equals $Y_{n+1}p^{-1}\sigma(R_{2n})=Y_{n+1}p^{-1}R_1$ by \Cref{frobactionrows}.
For $j<n$, by the above computation, we write $A_j^{(i)} ((S'_0)^{-1})^{(i)}$ as $(-Y_j^{(i)}p^{-1})e_j\sigma^{i}(R_{2n})+(Y_j^{(i)}p^{-1})e_n\sigma^{i}(R_{n+j})$, where $e_j$ (resp. $e_n$) is the column vector with all coordinates $0$ expect the $j\thh$ (resp. $n\thh$) coordinates being $1$. Then \[(S'_0)^{(i)}A_j^{(i)} ((S'_0)^{-1})^{(i)}=(-Y_j^{(i)}p^{-1})\varphi^i(pv_j)\sigma^{i}(R_{2n})+(Y_j^{(i)}p^{-1})\varphi^i(pv_n)\sigma^{i}(R_{n+j}).\] By definition of $R_{n+i}, R_1$, we have 
$R_{n+i}v_j=0=R_1 w_j$, $R_{n+i} w_j=0$ for $j\neq i$, $R_{n+i} w_i=1$, $R_1 v_j=0$ for $j>1$, and $R_{1}( pv_1)=1$.
For $i\leq n$, the coefficient of $w_i$ in $\varphi^i(pv_j)$ (resp. $\varphi^i(pv_n)$) is $0$ (resp. $1$) by \Cref{cyclicfrob}(3) and thus the $n^{\text th}$ row of $M$ is $ Y_iY_j^{(i)}p^{-2}\sigma^i(R_{n+j})=Y_iY_j^{(i)}p^{-2} \sigma^{i+j-1}(R_{n+1})$ by \Cref{frobactionrows}. The cases when $i=n+1$ or $j=n,n+1$ also follow from \Cref{cyclicfrob} and \Cref{frobactionrows} by direct computations as above.

(2) We prove by induction. Indeed, we only need to verify the expression of $M$. The base case is just (1). We assume that $i_1\leq n$ and the case $i=n+1$ follows by a similar computation. Note that \[M=A_{i_1}(S'_0)^{-1}(\prod_{j=2}^l K_{i_j}^{(i_2+\cdots+i_{j-1})})^{(i_1)}=A_{i_1}(S'_0)^{-1} (S'_0)^{(i)} (Y_{i_2}^{(i_1)}\cdots Y_{i_l}^{(i_1+\cdots +i_{l-1})}p^{l-1})e_n\sigma^{i_1+\cdots+i_l-1}(R_{n+1})\] by the induction hypothesis. By the computation in (1), we have $A_{i_1}(S'_0)^{-1} (S'_0)^{(i)}e_n=(Y_{i_1}p^{-1})e_n$ and thus the assertion follows.
\end{proof}

The following lemma pick out the terms with minimal $t$-adic valuation among those with a fixed $p$-adic valuation.

\begin{para}\label{Notation_Imin}
We introduce some notation for the lemmas. For $r\in \bZ_{>0}$, define $\bI_r= \{1,2,\dots, n+1 \}^r$. For $I = (i_1,\dots, i_r) \in \bI_r$, define $P_I=K_{i_1}\cdot K_{i_2}^{(i_1)}\cdot K_{i_3}^{(i_1 + i_2)}\cdots K_{i_r}^{(i_1+i_2+\cdots+ i_{r-1})}$ and define the \emph{weight} of $I$, denoted by $\mu_I$, to be $\sum_{j=1}^r i_j$.
By \Cref{lintopleft}, we write $P_I=S'_0 M_I$ and note that all nonzero entries in $P_I$ have the same $t$-adic valuation $\sum_{j=1}^r p^{i_1+\cdots+ i_{j-1}}a_{i_j}=:\nu_I$ (recall that $a_i=v_t(Y_i)$).

Note that in the product expansion of $\F(1)$, among all terms with a $p$-adic valuation $-r$, the ones with the minimal $t$-adic valuation have to be of the form $P_I$ with $I\in \bI_r$. Let $\nu_r$ denote this minimal $t$-adic valuation. Define $\bI^{\min}_r = \{I \in \bI_r: v_t(P_I) = \nu_r \}$. In other words, among all terms with $p$-adic valuation $-r$ in the product expansion of $\F(1)$, the ones with minimal $t$-adic valuations are $P_I, I\in \Imin_r$. The following lemma provides some information of the set $\Imin_r$.
\end{para}

\begin{lemma}\label{minval}
Notation as in \S\ref{Notation_Imin} and let $I = (i_1,\dots, i_r) \in \bI^{\min}_r$. Then: \begin{enumerate}
    \item $(i_2,\dots, i_r) \in \bI^{\min}_{r-1}$. Conversely, if $(j_2,\dots j_r) \in \bI^{\min}_{r-1}$, then there exists $j_1\in \{1,\dots, n+1 \}$ such that $(j_1,j_2,\dots, j_r) \in \bI^{\min}_{r-1}$.
    
    \item $i_1 \leq i_2\leq  \cdots \leq i_r$ and $a_{i_1} \geq a_{i_2} \geq \cdots \geq a_{i_r}$.
    
    \item Let $J = (j_1,\dots j_r)\in \Imin_r$. Then $(l_1,\dots, l_r)\in \Imin_r$, where each $l_\alpha$ is either $i_\alpha$ or $j_\alpha$ for $1\leq \alpha \leq r$. 
    
    \item Let $J$ be as in (3). Then $|\mu_I - \mu_J | < n+1$.
    
    \item Suppose that $|\Imin_r| > 1$. Then there exist two elements in $\Imin_r$ with different weights. Further, there is a unique $I \in \Imin_r$ with maximal weight, and a unique $J \in \Imin_r$ with minimal weight. 
\end{enumerate}
\end{lemma}

\begin{proof}
\begin{enumerate}
    \item By \S\ref{Notation_Imin}, we have $\nu_I = a_{i_1} + p^{i_1}\nu_{(i_2,\dots,i_r)}$, and thus $\nu_{(i_2,\dots,i_r)}$ has to be minimized in order for $\nu_I$ to be minimal. On the other hand, take $j_1 = i_1$ and then we have $\nu_{(j_1,\dots, j_r)}=a_{i_1}+p^{i_1}\nu_{(j_2,\dots,j_r)}=a_{i_1}+p^{i_1} \nu_{(i_2,\dots,i_r)}= \nu_r$.
     
    \item We prove the assertion by induction on $r$. By the inductive hypothesis and (1), we may assume that $i_2\leq i_3 \leq \cdots i_r$ and $a_{i_2}\geq a_{i_3} \geq \cdots \geq a_{i_r}$. 
    
    Assume for contradiction that $i_1 > i_2$. If $a_{i_2}\leq a_{i_1}$, then $\nu_{(i_2,i_2,i_3, \hdots i_r)} < \nu_I$, which contradicts with $I\in \Imin_r$. Therefore $a_{i_1} < a_{i_2}$. Let $I' = (i_2,i_1,i_3,\dots, i_r)$. We have $\nu_{I'} - \nu_I = \nu_{(i_2,i_1)} - \nu_{(i_1,i_2)} = a_{i_1}(p^{i_2} - 1) - a_{i_2}(p^{i_1}-1)<0$. Thus we must have $i_1\leq i_2$. 
    
    Now assume for contradiction that $a_{i_1} <a_{i_2}$. Then  $\nu_{(i_1,i_1,i_3,\hdots i_r)} < \nu_I$. Thus $a_{i_1}\geq a_{i_2}$ as required. 
    
    \item Suppose that $J = (j_1,\dots, j_r)$. By (1), it follows that $\nu_{(j_2,\dots, j_r)}= \nu_{(i_2,\dots, i_r)} = \nu_{r-1}$, whence $a_{i_1} + p^{i_1}\nu_{r-1} = a_{j_1} + p^{j_1}\nu_{r-1}$. It follows that $(i_1,j_2,j_3,\dots, j_r), (j_1,i_2,\dots, i_r) \in \Imin_r$. (3) now follows by induction on $r$. 
    
    \item $\mu_I - \mu_J = i_1 - (\sum_{\alpha=1}^{r-1} (j_\alpha - i_{\alpha+1})) - j_r$. By (2) and (3), $j_\alpha \leq i_{\alpha+1} $; since $j_r > 0$, then $\mu_I - \mu_J < i_1 \leq n+1$. Similarly, $\mu_J - \mu_I < n+1$; thus the result follows. 
    
    \item For $1\leq \alpha \leq r$, set $M_\alpha = \max_{J \in \Imin_r} \{j_\alpha \}$ and $m_\alpha = \min_{J \in \Imin_r}\{j_\alpha\}$. By applying (3) repeatedly the set of all $J \in \Imin_r$, it follows that $(M_1,\dots, M_r), (m_1,\dots, m_r)\in \Imin_r$. By definition, $(M_1, \dots, M_r)$ is the unique element of $\Imin_r$ with maximal weight, and $(m_1,\dots, m_r)\in \Imin_r$ is the unique element with minimal weight. \qedhere
\end{enumerate}
\end{proof}

\subsection*{Other preparation lemmas} Recall that $\overline{R}_i$ denote $R_i \bmod p$.

\begin{lemma}\label{rowdotFp}
For any $0\neq v\in \bF_p^{2n}$, we have $\overline{R}_{n+1}  v \neq 0$. Consequently, if $\{z_1,\dots, z_{2n}\}$ is a basis of $\bF_p^{2n}$, then $\overline{R}_{n+1} z_1, \dots, \overline{R}_{n+1}  z_{2n} \in k$ are linearly independent over $\bF_p$. 
\end{lemma}
\begin{proof}
If $\overline{R}_{n+1} v=0$, then $\sigma^i(\overline{R}_{n+1}) v=0$ for all $i\geq 0$. By \Cref{frobactionrows}, $\Span_W\{\sigma^i(R_{n+1})\}_{i=0}^{2n-1}=W^{2n}$; thus $\overline{R}_{n+1},\sigma(S_{n+1}),\dots, \sigma^{2n-1}(\overline{R}_{n+1})$  form a basis of $k^{2n}$. Therefore, if $\overline{R}_{n+1} v=0$, then $v=0$; this proves the first assertion.

For the second assertion, suppose there exists a non-trivial linear relation $\sum_{i=1}^{2n} a_i (\overline{R}_{n+1} z_i) = 0$ with $a_i\in \bF_p$. Then, $\overline{R}_{n+1} (\sum_{i=1}^{2n} a_i z_i) = 0$, which contradicts the first assertion. 
\end{proof}

\begin{lemma}\label{lastgenericdecay}
Let $\alpha_0, \dots, \alpha_n \in k$ such that $(\alpha_0,\dots,\alpha_n)\neq(0,\dots,0)$. Consider the linear combination $\displaystyle \overline{R} = \sum_{i=0}^n \alpha_i \sigma^{i}(\overline{R}_{n+1})$. Then $\dim_{\bF_p}\{v\in \bF_p^{2n}\mid \overline{R} v = 0\}\leq n$.
\end{lemma}
\begin{proof}
For any $z\in \bF_p^{2n}$, note that $\overline{R}  z = \vec{\alpha} \vec{\beta}(z)$, where
\[\vec{\alpha}=(\alpha_0,\dots, \alpha_n),\,\vec{\beta}(z) = \left[
\begin{array}{c}
\overline{R}_{n+1} z\\
(\overline{R}_{n+1} z)^{(1)}\\
\vdots\\
(\overline{R}_{n+1} z)^{(n)}
\end{array}
\right]. \]

Now assume for contradiction that there exist linearly independent vectors $z_1, z_2,\dots, z_{n+1}\in \bF_p^{2n}$ such that $\overline{R}z_j = 0$. This implies that $\vec{\alpha} \vec{\beta}(z_j) = 0$ for every $1\leq j\leq n+1$. In particular, this implies that the row vector $\vec{\alpha}$ is in the (left) kernel of the $(n+1)\times (n+1)$ matrix $M(z)$ whose $j^{th}$ column is $\vec{\beta}(z_j)$. This contradicts the assumption $\vec{\alpha}\neq 0$ once we show that $M(z)$ is invertible. Indeed, note that the $(i+1)\thh$ row of $M(z)$ is the Frobenius twist of the $i^{th}$ row, and hence $M(z)$ is a Moore matrix. The determinant of a Moore-matrix vanishes if and only if the entries of the first row are linearly independent over $\bF_p$. However, the first row of $M(z)$ consists of the elements $\{\overline{R}_{n+1} z_i\}_{i=1}^{n+1}$, and by \Cref{rowdotFp} these elements are linearly independent over $\bF_p$. The lemma follows. 
\end{proof}

\subsection*{Decay in the non-superspecial case}

\begin{proof}[Proof of Theorem \ref{genericdecay}]
We follow the argument in \S\ref{dJ+Kisin}.
Let $w$ be a primitive vector (that is, not a multiple of $p$) in $\Span_{\bZ_p}\{e_1,\dots, e_n, f_1,\dots, f_n\}\subset \cL$. With respect to the basis $\{e_i,f_i, e'_j, f'_j\}$, we view $w$ as a vector in $W^{2n+2m}$, which has the last $2m$ coordinates being $0$. Let $w_0$ denote the vector in $W^{2n}$ whose coordinates are the first $2n$ coordinates of $w$ (indeed, as vectors in $\cL$, $w=w_0$). Then for any $r,s\in \bZ_{\geq 0}$, if $p^{r}\F(1)w_0$ is not integral in $\bbL_\cris(D_s)$, then neither is $p^{r} \F w$. Thus to prove the Decay Lemma, it suffices to work with $\F(1)$. Thus for a general $\bL_1$, we may embed it into the split case as described in \S\ref{embedding} to reduce the proof to the split case.

Notation as in \S\ref{Notation_Imin}; let $M_{r+1}$ denote the matrix in $M_{2n}(K)$ such that \[\sum_{I \in \Imin_{r+1}} P_I = S'_0\sum_{I \in \Imin_{r+1}}M_I=t^{\nu_{r+1}}M_{r+1} + t^{\nu_{r+1}+1}M_{2n}(K[[t]]).\] By definition, $p^{r+1}M_{r+1}\in M_{2n}(W)$.
First we follow the reduction step as in the last paragraph of the proof of Thm.~5.1.2 assuming Prop.~5.1.3 in \cite{MST}.
Note that if $w_0 \bmod p\notin (p^{r+2} M_{r+2})\bmod p$, then the coefficient of $t^{\nu_{r+2}}$ in $p^rw_0$, with respect to the basis $\{e_i,f_i\}$ does not lie in $p^{-1}W$. Since $p\bL\subset \cL\otimes_{\bZ_p} W$, then the coefficient of $t^{\nu_{r+2}}$ in $p^rw_0$, with respect to the basis $\{v_i,w_i\}$ does not lie in $W$. Thus, $p^r w_0\notin \bbL_\cris(D_s)$ for $s>\nu_{r+2}/p$. Thus the following claim implies the Decay lemma.

\begin{claim}
\begin{enumerate}
    \item  There exists a saturated $\bZ_p$-submodule $\Lambda\subset \cL$ of rank at least $n$ such that if $v\in \Lambda$ is a primitive vector, then $v\bmod p\notin \ker (p^{r+2}M_{r+2} \bmod p)$ for all $r\in \bZ_{\geq 0}$.
    \item For all  $r\in \bZ_{\geq 0}$, we have $\nu_{r+2}/p<h_r+1$, where $h_r$ is defined in \Cref{decaydef}.
\end{enumerate}
\end{claim}

For (1), by \Cref{lintopleft}(2), $p^{r+2}M_{r+2}$ only has its $n\thh$ row non-zero and its $n\thh$ row is a $W$-linear combination of $\{\sigma^{\mu_I-1}(R_{n+1})\}_{I\in \Imin_{r+2}}$. By \Cref{minval}, let $J,J' \in \Imin_{r+2}$ denote the unique elements such that $\mu_J = \max_{I\in \Imin_{r+2}}\mu_I$ and $\mu_{J'} = \min_{I\in \Imin_{r+2}}\mu_I$. Thus the $n\thh$ row of $p^{r+2}M_{r+2}\bmod p$ is a $k$-linear combination of $\{\sigma^{\mu-1}(\overline{R}_{n+1})\}_{\mu_{J'}\leq \mu \leq \mu_J}$ and by \Cref{minval}(4), this set consists at most $n+1$ elements. Note that since $J,J'$ are unique, the the coefficients of $\sigma^{\mu_{J'}-1}(\overline{R}_{n+1})$ and $\sigma^{\mu_J-1}(\overline{R}_{n+1})$ are non-zero in $k$. Then by \Cref{lastgenericdecay}, $\dim_{\bF_p}S_{r+1}\leq n$, where $S_{r+1}:=\{\bar{v}\in \bF_p^{2n}\mid \bar{v}\in \ker(p^{r+1}M_{r+1}\bmod p)\}$. By \Cref{minval}(1)(3), $\Imin_{r+2}=\bI'\times \Imin_{r+1}$, where \[\bI'=\{i\mid 1\leq i \leq n+1, \exists J(i)\in \Imin_{r+1}\text{ such that } (i,J(i))\in \Imin_{r+2}\}.\]
Then $p^{r+2}M_{r+2}=\sum_{i\in \bI'} A_i(p^{r+1}M_{r+1})^{(i)}$, where $A_i\in M_{2n}(W)$. Therefore, $S_{r+1}\subset S_{r+2}$. Thus $S_\infty:=\bigcup_{j=1}^\infty S_j$ is a subspace of $\bF_p^{2n}$ of dimension at most $n$. Thus there exists a saturated $\bZ_p$-submodule $\Lambda\subset \cL$ or rank at least $n$ such that $\Lambda\bmod p\cap S_\infty=\{0\}$ and any primitive vector $v\in \Lambda$ satisfies the desired condition.

For (2), note that $\nu_{r+2}\leq \nu_{(1,\dots,1)}=h(1 + \cdots + p^{r+1})$ since $h = a_1$; thus $h_r+1>h(p^r+\cdots+p^{-1})\geq \nu_{r+2}/p$.
\end{proof}

\begin{para}
Ogus \cite[Lem.~2, Prop.~11]{Ogus01} gives explicit description of the local equation of Newton strata of $\Sh_k$ and by using the explicit coordinates in \S\ref{Yn}, in the formal neighborhood of a supersingular point $P$ of type $2n$, the Newton stratum of codimension $s$ is cut out by the single equation $y_{s}=0$ in the Newton stratum of codimension $s-1$ for $s\leq n$. Let $C\rightarrow \Sh_k$ be a formal curve which specializes to $P$. Assume that the generic point of $C$ lies in the open Newton stratum of codimension $s-1\leq n-1$. We say a special endomorphism $w$ of $A[p^\infty]$ \emph{decays rapidly} if it satisfies the condition in \Cref{decaydef} with $h=v_t(y_s(t))=a_s$. The only place in the proof of \Cref{genericdecay} where we used the generic ordinary assumption of $C$ is Claim (2). Once we replace the computation there by $\nu_{r+2}\leq \nu_{(s,\dots,s)}=a_s(1+\cdots+p^{r+1})$, we obtain the following general version of the Decay Lemma for non-superspecial supersingular points.
\end{para}

\begin{theorem}[Generalized decay lemma in the generic case]\label{nonorddecay}
Suppose that $C\rightarrow \Sh_k$ is a formal curve which specializes to a non-superspecial supersingular point of type $2n$ (i.e., Artin invariant $n$) and is generically in an open Newton stratum of codimension $\leq n-1$. Then there exists a saturated rank $n$ submodule of special endomorphisms which decays rapidly. 
\end{theorem}

\section{Decay for superspecial points}\label{sec_decay_ssp}
The goal of this section is to prove a Decay Lemma for superspecial points (\Cref{decayspecial}). 
The computations in the proof of \Cref{decayspecial} go along very similar lines to the calculations carried out in \cite[\S 5]{MST}. We will therefore be brief and will refer to \cite{MST} whenever appropriate. 

Throughout this section, we work in the setting of a formal curve $C = \Spf k[[t]]\rightarrow \Sh_k$ which is generically ordinary, and specializes to a superspecial point $P$. Recall that $\cA/k[[t]]$ denotes the pullback of $\cA^\univ$, $h$ denotes the $t$-adic valuation of the local equation of the non-ordinary locus given in \Cref{eq-nonord}, and $\cL$ is the lattice of special endomorphisms of the $p$-divisible group at $P$.

In order to obtain sufficiently strong bounds to prove \Cref{thm_int}, we require a Decay Lemma which is slightly stronger than \Cref{genericdecay}. In order to do this, we introduce the notion of very rapid decay; also the following definition for rapid decay in the superspecial case is stronger than \Cref{decaydef}. 
\begin{defn} For a superspecial point $P$,
\begin{enumerate}
    \item We say that $w\in \cL$ \emph{decays rapidly} (resp. \emph{very} rapidly) if for every $r\in \bZ_{\geq 0}$, the special endomorphism $p^r w$ does not lift to an endomorphism of $\cA[p^\infty]$ modulo $t^{h'_r+1}$ (resp.$t^{h'_{r-1}+ap^r+1}$),for some $a\leq \frac{h}{2}$ independent of $r$; here $h'_r:=[h(p^r+\cdots+1)+a/p)]$ and $h'_{-1}:=[a/p]$.

    \item We say that $w\in \cL$ \emph{ decays rapidly (resp.~very rapidly) to first order} if $w$ does not extend to an endomorphism modulo $t^{[h+a/p]+1}$ (resp.~$t^{[a+a/p]+1}$) for some $a \leq \frac{h}{2}$.

    \item 
    A $\bZ_p$-submodule of $\cL$ \emph{decays rapidly} if every primitive vector in this submodule decays rapidly.
    Given a submodule $\Lambda\subset \cL$ which decays rapidly and a vector $w\in \cL$ such that $w\notin \Lambda$, we say that  \emph{the pair $(L,w)$ decays very rapidly to first order} if $w$ decays very rapidly to first order, and every primitive vector in $\Span_{\bZ_p}\{L,w\}$ decays rapidly to first order. 
\end{enumerate}
\end{defn}
The main theorem of this section is the following: 

\begin{theorem}[Decay Lemma in the superspecial case]\label{decayspecial}
There exists a saturated rank $2$ $\bZ_p$-submodule $\Lambda \subset \cL$ which decays rapidly. Moreover, at least one of the following statements holds: 
\begin{enumerate}
    \item there exists a primitive $w \in \Lambda$ which decays very rapidly; 
    
    \item there exists a primitive vector $w\in \cL$ such that $w \notin \Lambda$ and the pair $(L,w)$ decays very rapidly to first order. 
\end{enumerate}
\end{theorem}
We expect that an analogous statement of \cite[Thm.~5.1.2]{MST} holds; more precisely, we expect that there is a rank $3$ submodule of $\cL$ which decays rapidly, and moreover, there exists a vector in this rank $3$ submodule which decays very rapidly. In order to prove \Cref{thm_int}, the weaker statement \Cref{decayspecial} suffices.

Theorem \ref{cor_decay_ssp} follows directly from Theorem \ref{decayspecial}:
\begin{proof}[Proof of Theorem \ref{cor_decay_ssp}]
The argument used to deduce Theorem \ref{cor_decay_sg} from Theorem \ref{genericdecay} works in this setting, with Theorem \ref{cor_decay_ssp}(1) following from Theorem \ref{decayspecial}(1), and Theorem \ref{cor_decay_ssp}(2) following from Theorem \ref{decayspecial}(2).
\end{proof}

\subsection*{The setup} Here we carry out all the computation for the split case described in \S\ref{embedding} and we will explain in the proof how to deduce the general case from the split case.

\begin{para}
Recall from \S\ref{Frob_vw} that  $\widehat{\cO}_{\Sh,P}=\Spf W[[ x_1,\dots, x_m,y_1,\dots,y_m]]$ (note that $x_j,y_j$ here were denoted by $x'_j,y'_j$ in \S\ref{Frob_vw}); the formal curve $C$ gives rise to the tautological map of local rings
\[
W[[x_1,\dots, x_m, y_1, \dots, y_m]] \rightarrow k[[t]]    
\]
and we let $x_i(t)$ (respectively $y_i(t)$) denote the images of the $x_i$ (respectively $y_i)$) in $k[[t]]$. For each of the $x_i(t)$ (respectively $y_i(t)$), let $X_i(t) \in W[[t]]$ (respectively $Y_i(t)\in W[[t]]$) denote the power series whose coefficients are the Teichmuller lifts of those of $x_i(t)$ (respectively $y_i(t)$). Let $\displaystyle Q(t) = -\sum_{i=1}^m X_i(t)Y_i(t)$ (compare to \S\ref{Frob_vw}, here we use $Q$ to denote the lift of itself), and let $\displaystyle R(t) = -\sum_{i=1}^m \Big(X_i(t)(Y_i(t))^{p} + (X_i(t))^{p}Y_i(t)\Big)$. By \Cref{eq-nonord}, $Q(t)\bmod p=0$ is the local equation for the non-ordinary locus, so  $h = v_t(Q(t))$. Let $h' = v_t(R(t))$. Without loss of generality, we may assume that $v_t(X_1(t))\leq v_t(X_i(t))$ and $v_t(X_1(t))\leq v_t(Y_i(t))$ as everything is symmetric in the $x_i,y_i$, and let $a$ denote $v_t(X_1(t))$. By definition, we have that $2a \leq h$ and $(p+1)a \leq h'$. 
\end{para}

\begin{para}\label{Frob_ef_ssp}
Following the notation of \Cref{L''_p0} for $n=1$, the vector $v_1$ in \Cref{cyclicfrob} must be $\frac{1}{2p}(e_1+f_1/\lambda)$ and $w_1=\frac{1}{2}(e_1-f_1/\lambda)$.\footnote{There is another possible choice with $\lambda$ replaced by $-\lambda$; given the computation is the same for both cases, we will just work with the first case.} 
By \S\S\ref{Frob_vw},\ref{Frob_ef}, we have that with respect to the basis $\{e_1,f_1,e'_i,f'_i\}_{i=1}^m$, the Frobenius on $\bbL_\cris(W[[x_i,y_i]])$ is given by 
\[
\Frob =(I +F)\circ \sigma, \text{ where } F=
\left[
\begin{array}{cc|cccccc}
\frac{Q}{2p}&\frac{-\lambda Q}{2p}&\frac{x_1}{2p}&\hdots&\frac{x_m}{2p}&\frac{y_1}{2p}&\hdots&\frac{y_m}{2p}\\
\frac{Q}{2p\lambda}&\frac{-Q}{2p}&\frac{x_1}{2p\lambda}&\hdots&\frac{x_m}{2p\lambda}&\frac{y_1}{2p\lambda}&\hdots&\frac{y_m}{2p\lambda}\\
\hline
-y_1&\lambda y_1&&&&&&\\
\vdots&\vdots&&&&&&\\
-y_m&\lambda y_m&&&&&&\\
-x_1&\lambda x_1&&&&&&\\
\vdots&\vdots&&&&&&\\
-x_m&\lambda x_m&&&&&&\\
\end{array}
\right].
\]
Let $F_t, F_r$ and $F_l$ denote the top-left $2\times2$ block, the top-right $2\times 2m$ block and the bottom-left $2m\times2$ block of $F$ respectively. Let $F_{r,i},F_{l,i}$ denote the $i\thh$ column of $F_r$ and $i\thh$ row of $F_l$ respectively.

As in \S\ref{dJ+Kisin}, in order to prove the Decay Lemma, we study the expansion of $\F=\prod_{i=0}^\infty(I+F^{(i)})$.
Let $\F(1)$, $\F(2)$ and $\F(1,3)$ denote the top-left $2 \times 2$ block, the top-right $2\times 2m$ block, and the $(m+1)\thh$ row of the bottom-left $2m\times2$ block of $\F$ respectively. We denote by $\F(2)_i$ the $i\thh$ column of $\F(2)$ and $\F(2,3)_i$ the $i\thh$ entry of the $(m+1)\thh$ row of the bottom-right $2m\times 2m$ block of $\F$.

Let
\[
M = \frac{1}{2}
\left[
\begin{array}{cc}
    1 & -\lambda \\
    \frac{1}{\lambda} &-1 
\end{array}
\right], 
N = \frac{1}{2}
\left[
\begin{array}{cc}
    1 & \lambda \\
    \frac{1}{\lambda} &1 
\end{array}
\right].
\]
\end{para}

\subsection*{Preliminary lemmas}
As in \cite[\S 5.2]{MST}, $\F$ is evaluated by considering sums of finite products of $\sigma$-twists of $F_t,F_r$ and $F_l$. The following lemma follows directly from the shape of $F$ and an elementary analysis on $t$-adic valuations.
\begin{lemma}\label{Finfspecial} Fix $r\in \bZ_{\geq 0}$
\begin{enumerate}
    \item Among all terms in the product expansion of $\F(1,3) $ with $p$-adic valuation $-(r+1)$, the terms which have the smallest $t$-adic valuation are contained in the set \[\bS_{1,r+1}=\{F_{l,m+1}\prod_{i=1}^\alpha F_t^{(i)} \prod_{j=1}^\beta F_r^{(\alpha+2j-1)}F_l^{(\alpha + 2j)}\mid \alpha,\beta\in \bZ_{\geq 0}, \alpha+\beta=r+1  \}.\]
    
    \item Among all terms in the product expansion of $\F(2,3)_s$ with $p$-adic valuation $-(r+1)$, the terms with the smallest $t$-adic valuation are contained in the set \[\bS_{2,s,r+1}=\{F_{l,m+1}\prod_{i=1}^\alpha F_t^{(i)} \prod_{j=1}^\beta F_r^{(\alpha+2j-1)}F_l^{(\alpha + 2j)} F_{r,s}^{(\alpha+2\beta +1)}\mid \alpha,\beta\in \bZ_{\geq 0}, \alpha+\beta=r\}.\]
\end{enumerate}
\end{lemma}
The following lemma follows from a direct computation similar to \cite[Lem.~5.2.1, Lem.~5.2.3]{MST}.
\begin{lemma}\label{evalproduct}
Consider the product $\displaystyle P_{\alpha,\beta} = \prod_{i=1}^\alpha F_t^{(i)} \cdot \prod_{j=1}^{\beta}F_r^{(\alpha + 2j-1)}F_l^{(\alpha + 2j)}$. 
\begin{enumerate}
    \item If $\alpha$ is odd, the product equals $p^{-(\alpha+\beta)}\prod_{i=1}^\alpha Q^{(i)}\cdot \prod_{j=1}^\beta R^{(\alpha+2j-1)} M^{(1)}$.
    
    \item If $\alpha$ is even, the product equals $p^{-(\alpha+\beta)}\prod_{i=1}^\alpha Q^{(i)}\cdot \prod_{j=1}^\beta R^{(\alpha+2j-1)} N^{(1)}$.
\end{enumerate}
In either case, the kernel of $F_{l,m+1}p^{\alpha+\beta}P_{\alpha+\beta}\bmod p$ does not contain 
 any non-zero $\bF_p$-rational vectors or any $k^\times$-multiple of
     $([1,\lambda^{-1}]^{(\alpha+1)})^T$.
\end{lemma}

\subsection*{Decay in the superspecial case}

\begin{proof}[Proof of Theorem \ref{decayspecial}]We first prove the theorem in the split case. We continue the argument in \S\ref{dJ+Kisin}.
\par{\noindent \bf Case 1: $h < h'$.} 
It follows from Lemmas \ref{Finfspecial} and \ref{evalproduct} that there is a unique element of $\bS_{1,r+1}$ (respectively $\bS_{2,1,r+1}$) with minimal $t$-adic valuation. This term is $F_{l,m+1}P_{r+1,0}$ (respectively $F_{l,m+1}P_{k,0} F_{r,1}^{(r+1)}$), and has $t$-adic valuation $a+h(p+\cdots+ p^{r+1})$ (respectively $a+h(p+\cdots + p^{r}) + ap^{r+1}$). 
By the last assertion of \Cref{evalproduct} and the fact that $\gamma_0 e_1+\delta_0 f_1+\sum_{i=1}^m \gamma_i e'_i +f'_1+\sum_{i=2}^m \delta_i f'_i $ is primitive in $\bL=\bbL_{\cris,P}(W)$ for any $\gamma_i, \delta_i \in W$, we conclude that for any primitive vector $w\in \Span_{\bZ_p}\{e_1,f_1\}$, the horizontal section $p^r\F w$ is not integral in $\bbL_\cris(D_s)$ if $s>(a+h(p+\cdots+ p^{r+1}))/p$; thus $\Span_W\{e_1,f_1\}$ decays rapidly.
Similarly by \Cref{evalproduct}, the special endomorphism $e'_1$ decays very rapidly. The fact that $\Span_{\bZ_p}\{e_1,f_1,e'_1\}$ decays rapidly follows from an argument identical to that outlined in the last paragraph of \cite[Proof of Prop.~5.1.3 Case 1 in \S 5.2]{MST}. For the convenience of the reader, we give a brief sketch of this argument. Let $w$ be a primitive vector in $\Span_{\bZ_p}\{e_1,f_1\}$. As in \emph{loc. cit.}, it suffices to prove that the $t$-adic valuation of the term of $\F w$ with denominator $p^{r_1}$ is different from the $t$-adic valuation of the term of $\F e'_1$ with denominator $p^{r_2}$ for any $r_1,r_2 \in \bZ_{>0}$. The former quantity equals $a+h(p + \cdots+ p^{r_1})$, and the latter quantity equals $a+h(p + \hdots p^{r_2-1}) + ap^{r_2}$. As $1\leq a \leq h/2$, it follows that these quantities can never be the same and the result follows.

Note that in this case, we have proved that a rank $3$ submodule of $\cL$ must decay.\\

\par{\noindent \bf Case 2: $h'(1+p^{2e-1}) < h(1+p) < h'(1+p^{2e+1})$, for some $e\in \bZ_{\geq 1}$.}
As in Case 1, by \Cref{Finfspecial}, $F_{l,m+1}P_{r-e+1,e}$ is the unique element of $\bS_{1,r+1}$ with minimal $t$-adic valuation (the argument is similar to that of \cite[Lem.~5.2.6]{MST}); moreover $v_t(F_{l,m+1}P_{r-e+1,e})<v_t(F_{l,m+1}P_{r+1,0})<p(h'_r+1)$. Thus, by \Cref{evalproduct} and the same argument as in Case 1, $\Span_{\bZ_p}\{e_1,f_1\}$ decays rapidly.
On the other hand, $|S_{2,1,1}|=1$ and the unique element has $t$-adic valuation $a+pa$ and thus $e'_1$ decays very rapidly to first order. The fact that the pair $(\Span_{\bZ_p}\{e_1,f_1\},e'_1)$ decays very rapidly to first order follows from an argument identical to the one outlined at the end of Case 1. \\

\par{\noindent \bf Case 3: $h'(1+p^{2e-1}) = h(1+p)$ for some $e\in \bZ_{\geq 1}$.}
In this case, by \Cref{Finfspecial} and a computation similar to \cite[Lem.~5.2.7]{MST}, we have that $F_{l,m+1}P_{r-e+1,e-1} \cdot F_{r,1}^{(r+e)}$ is a unique element of $\bS_{2,1,r+1}$ with the smallest $t$-adic valuation and \[v_t(F_{l,m+1}P_{r-e+1,e-1} \cdot F_{r,1}^{(r-e)})\leq v_t(F_{l,m+1}P_{r,0} \cdot F_{r,1}^{(r+1)})<p(h'_{r-1}+ap^r+1).\] Then the last assertion of Lemma \ref{evalproduct} implies that $e'_1$ decays very rapidly. 

\begin{claim}
At least one of $e_1,f_1$ decays rapidly.
\end{claim}
\begin{proof}
When $r<e-1$, there is a unique element of $\bS_{1,r+1}$ with minimal $t$-adic valuation, and thus the argument as in Case 1 shows that for any primitive vector $w\in \Span_{\bZ_p}\{e_1,f_1\}$, we have that $p^rw$ does not lift to an endomorphism $\bmod t^{h'_r+1}$.

When $r\geq e-1$, there are exactly \emph{two} distinct elements of $\bS_{1,r+1}$ with minimal $t$-adic valuation, and they are $P_{r-e+1,e}$ and $P_{r-e+2,e-1}$. We first prove that at least one of $p^{e-1}e_1, p^{e-1}f_1$ does not extend to  an endomorphism modulo $t^{h'_{e-1}+1}$. Indeed, by \Cref{evalproduct}, we have that $F_{l,m+1}(P_{0,e} + P_{1,e-1}))$ equals $p^{-e}(AF_{l,m+1}M^{(1)} + BF_{l,m+1}N^{(1)})$, where $A =\prod_{i=1}^e R^{ (2i-1)} , B =  Q^{(1)} \cdot\prod_{i=1}^{(e-1)} R^{(2i)}$. Let $\gamma , \delta\in W^{\times}$) denote the leading coefficients of $A, B\in W[[t]]$. As $p>2$, we have that at most one of  $\gamma - \delta$ and $\gamma +\delta$ lie in $pW$. Suppose that $\gamma + \delta\notin pW$, then $[1,0]^T$ does not lie in $\ker (\gamma [-1,\lambda]N^{(1)}+\delta [-1, \lambda]M^{(1)})$ and thus the $e'_1$-coordinate of the horizontal section $\tilde{e}_1$, up to a scalar multiple in $W^\times$, is $p^{-e}t^{a+h'(p+p^3+\cdots+p^{2e-1})}+$ other powers of $t$. Therefore, $p^{e-1}e_1$ does not lift to $\bmod t^{h'_{e-1}+1}$ because $h'_{e-1}+1>h(p^{e-1}+p^{e-2}+\cdots 1)+a/p \geq (a+h'(p+p^3+\cdots+p^{2e-1}))/p$. On the other hand, if $\gamma-\delta \notin pW$, then $[0,1]^T$ does not lie in $\ker \gamma [-1,\lambda]M^{(1)}+\delta [-1, \lambda]N^{(1)}$ and the same argument as above implies that $p^{e-1}f_1$ does not lift to $\bmod t^{h'_{e-1}+1}$.

Now we show that if $\gamma+\delta \notin pW$, then $e_1$ decays rapidly. A similar computation as above shows that for $r\geq e$, we have that $F_{l,m+1}(P_{r-e+1,e}+P_{r-e+2,e-1})=p^{-(r+1)}X_1\prod_{i=1}^{r-e+1} Q^{(i)}(A^{(r-e+1)}M^{(1)}+B^{(r-e+1)}N^{(1)})$ when $r-e$ is even and $F_{l,m+1}(P_{r-e+1,e}+P_{r-e+2,e-1})=p^{-(r+1)}X_1\prod_{i=1}^{r-e+1} Q^{(i)}(A^{(r-e+1)}N^{(1)}+B^{(r-e+1)}M^{(1)})$ when $r-e$ is odd. Since $\gamma+\delta \notin pW$, then $\gamma^{(r-e+1)}+\delta^{(r-e+1)}\notin pW$. Thus $[1,0]^T$ does not lie in $\ker (\gamma^{(r-e+1)} [-1,\lambda]N^{(1)}+\delta^{(r-e+1)} [-1, \lambda]M^{(1)})$ and $\ker (\gamma^{(r-e+1)} [-1,\lambda]M^{(1)}+\delta^{(r-e+1)} [-1, \lambda]N^{(1)})$. Therefore, we conclude that $e_1$ decays rapidly by a direct computation of the $t$-adic valuation of $F_{l,m+1}P_{r-e+1,e}$ as the $r=e-1$ case.

If $\gamma-\delta\notin pW$, then an identical argument as above implies that $f_1$ decays rapidly.
\end{proof}
To finish the proof of Case 3, we notice that an argument identical to the one outlined at the end of Case 1 goes through to show that if $e_1$ (resp. $f_1$) decays rapidly, then $\Span_{\bZ_p}\{e_1,e'_1\}$ (resp. $\Span_{\bZ_p}\{f_1,e'_1\}$) decays rapidly. \\

In general, if $\cL_1$ is not split in the sense of \S\ref{embedding}, we may embed it into a split one. More precisely, the even dimensional nonsplit case can be recovered as the subspace $\Span_{\bZ_p}\{e_1,f_1, e'_i+f'_i, e'_j+\lambda^2 f'_j, e'_l,f'_l \}_{1\leq l\leq m, l\neq i,j}$ for some $i\neq j$; the odd dimensional case can be recovered as the spaces $\Span_{\bZ_p}\{e_1,f_1, e'_i+f'_i,  e'_l,f'_l \}_{1\leq l\leq m, l\neq i}$ or $\Span_{\bZ_p}\{e_1,f_1, e'_i+\lambda^2 f'_i,  e'_l,f'_l \}_{1\leq l\leq m, l\neq i}$ for some $i$. The above argument for the split case proves the decay lemma for those subspaces such that $i,j\neq 1$. Moreover, if $i=1$ or $j=1$, the argument above for $\Span_{\bZ_p}\{e_1,f_1\}$ is still valid; on the other hand, once we replace $e'_1$ by $e'_1+f'_1$ or $e'_1+\lambda^2 f'_1$ and work with $\bS_{2,1,r+1}$ and $\bS_{2,m+1,r+1}$, the minimal $t$-adic valuation does not change by the last assertion of \Cref{evalproduct} and thus the above proof for the split case indeed proves the Decay Lemma in general.
\end{proof}


\section{Proof of \Cref{thm_int}}\label{sec_pf}
In this section, we prove \Cref{thm_int}. As sketched in the introduction, our approach is to combine global bounds from Borcherds theory with bounds on the average local intersection multiplicities.  At supersingular points, these are obtained using section \ref{sec_heuristic} (\Cref{cor_decay_sg} and \Cref{cor_decay_ssp}).

Note that \Cref{thm_int} is independent of the choice of level structure of $\Sh$ and is equivalent for different $\Sh$ with the same quadratic space $(L\otimes \bQ,Q)$ over $\bQ$; thus without loss of generality, we may assume that $L\subset V=L\otimes \bQ$ is maximal among all lattices over which $Q$ is $\bZ$-valued. Moreover, since all special divisors $\cZ(m)_{\bQ}$ are GSpin Shimura varieties, we may always work with the smallest GSpin Shimura variety whose reduction contains $C$ and thus we may assume that $C$ is not contained in any special divisors $Z(m)$.

\subsection*{The global intersection number and its decomposition}

\begin{para}\label{setS}
Let $S\subset \bZ_{>0}$ be a set of positive density (i.e., $\lim_{X\rightarrow \infty}\frac{1}{X}|\{m\in S \mid m\leq X\}|$ exists and $>0$) and we also assume that each $m\in S$ is representable by $(L,Q)$ and for any $m\in S$, we have $p\nmid m$. By the theory of quadratic forms, such $S$ exists.\footnote{Indeed, by \cite[Lem.~4.7]{SSTT}, every $m\gg 1$ is representable since $L$ is maximal.}
For $X\in \bZ_{>0}$, we use $S_X$ to denote $\{m\in S\mid X\leq m \leq 2X\}$.
\end{para}

\begin{para}\label{def_Eis}
We use vector-valued modular forms to control the asymptotic of $C.Z(m)$ as $m\rightarrow \infty$. Let $L^\vee$ denote the dual of $L$ in $V$ with respect to the bilinear form $[-,-]$ induced by $Q$ and let $\{\fe_{\mu}\}_{\mu\in L^\vee/L}$ denote the standard basis of $\bC[L^\vee/L]$. Let $\rho_L$ denote the Weil representation on $\bC[L^\vee/L]$ of the metaplectic group $\Mp_2(\bZ)$. As in \cite[\S 4.1.4]{MST}, we consider the Eisenstein series $E_0(\tau), \tau \in \bH$ defined by $\displaystyle E_0(\tau)=\sum_{(g,\sigma)\in \Gamma'_\infty\backslash \Mp_2(\bZ)}\sigma(\tau)^{-(2+n)}(\rho_L(g,\sigma)^{-1}\fe_0)$, where $\Gamma'_\infty\subset \Mp_2(\bZ)$ is the stabilizer of $\infty$.
Note that the constant term of $E_0$ is $\fe_0$ and $E_0(\tau)$ is a weight $1+\frac{b}{2}$ modular form with respect to $\rho_L$.
\end{para}

The following theorem of Bruinier--Kuss \cite{BK01} gives explicit formulae of the Fourier coefficients of $E_0$. As we are using different convention of the signature of $(V,Q)$ as in \cite{BK01}, we refer the reader to the formulae in \cite[Thms.~2.3, 2.4]{Br17}. 

\begin{para}\label{not_lat}
We first introduce some notation for an arbitrary quadratic lattice $(L,Q)$ over $\bZ$. We write $\det(L)$ for the determinant of its Gram matrix. For a rational prime $\ell$, we use $\delta(\ell, L,m)$ to denote the local density of $L$ representing $m$ over $\bZ_\ell$. More precisely, $\delta(\ell, L,m)=\lim_{a\rightarrow \infty} \ell^{a(1-\rk L)}\#\{v\in L/\ell^a L \mid Q(v)\equiv m \bmod \ell^a\}$. If $m$ is representable by $(L\otimes \bZ_\ell, Q)$, then $\delta(\ell,L,m)>0$; moreover, when $\rk L\geq 5$ (this is the case for our application), by for instance \cite[pp.~198-199]{Iwa97}, for a fixed $\ell$, we have that $\delta(\ell, L,m)\asymp 1$ for all $m$ representable by $(L\otimes \bZ_\ell, Q)$.

Given $0\neq D\in \bZ$ such that $D\equiv 0,1 \bmod 4$, we use $\chi_D$ to denote the Dirichlet character 
$\chi_D(a)=\left(\frac{D}{a}\right)$, where $\left(\frac{\cdot}{\cdot}\right)$ is the Kronecker symbol. For a Dirichlet character $\chi$, we set $\sigma_s(m,\chi)=\sum_{d|m}\chi(d)d^s$. 
\end{para}

\begin{theorem}[{\cite[Thm.~11]{BK01}}] \label{Eis-cof_L}
Recall that $(L,Q)$ is a quadratic lattice over $\bZ$ of signature $(b,2)$ with $b\geq 3$. Let $q_L(m)$ denote the coefficient of $q^m\fe_0$ in the $q$-expansion of $E_0$.
\begin{enumerate}
\item For $b$ even, the Fourier coefficient $q_L(m)$ is 
\[-\frac{2^{1+b/2}\pi^{1+b/2}m^{b/2}\sigma_{-b/2}(m,\chi_{(-1)^{1+b/2}4\det L})}{\sqrt{|L^\vee/L|}\Gamma(1+b/2)L(1+b/2,\chi_{(-1)^{1+b/2}4\det L})}\prod_{\ell \mid 2\det(L)}\delta(\ell, L, m).\]

\item For $b$ odd, write $m=m_0f^2$, where $\gcd(f,2\det L)=1$ and $v_\ell(m_0)\in \{0,1\}$ for all $\ell\nmid 2\det L$. Then the Fourier coefficient $q_L(m)$ is
\[-\frac{2^{1+b/2}\pi^{1+b/2}m^{b/2}L((b+1)/2,\chi_{\cD})}{\Gamma(1+b/2)\sqrt{|L^\vee/L|}\zeta(b+1)}\left(\sum_{d\mid f}\mu(d)\chi_{\cD}(d)d^{-(b+1)/2}\sigma_{-b}(f/d)\right)\prod_{\ell\mid 2\det L}\Big(\delta(\ell,L,m)/(1-\ell^{-1-b})\Big),\]
where $\mu$ is the Mobius function and $\cD=(-1)^{(b-1)/2}2m_0\det L$.
\end{enumerate}
In particular, $|q_L(m)|\asymp m^{b/2}$ for all $m$ representable by $(L,Q)$.
\end{theorem}
Here the last assertion is a direct consequence of the above explicit formulae and the fact $\delta(\ell,L,m)\asymp 1$ (see also \cite[\S 4.3.1]{MST}).

Recall that $C\rightarrow \Sh_k$ is a smooth proper curve such that the generic point of $C$ maps into the ordinary locus of $\Sh_k$.
\begin{lemma}\label{asymp_glo}
Let $\omega$ denote the tautological line bundle on $\Sh_k$ corresponding to $\Fil^1 V\subset V$ (i.e., $\omega$ is the line bundle of modular forms on $\Sh_k$ of weight $1$). Then
the intersection number $Z(m).C= |q_L(m)|(\omega . C)+O(m^{(b+2)/4})$. In particular, $\sum_{m\in S_X} Z(m).C\asymp (\omega .C) \sum_{m\in S_X} |q_L(m)|\asymp X^{1+b/2}$ for $S_X$ defined in \S\ref{setS}.
\end{lemma}
\begin{proof}
 By the modularity theorem of Borcherds \cite{Bor99} or its arithmetic version by Howard and Madapusi Pera \cite{HMP}, we have that $-(\omega.C)+\sum_{m=1}^\infty Z(m).C$ is the $\fe_0$-component of a vector-valued modular form with respect to $\rho_L$ of weight $(1+b/2)$ and its Eisenstein part is given by the $\fe_0$-component of $-(\omega.C)E_0$ (see \cite[Thm.~4.1.1, \S 4.1.4]{MST}). The difference of $-(\omega.C)+\sum_{m=1}^\infty Z(m).C$ and the $\fe_0$-component of $-(\omega.C)E_0$ is a cusp form, and thus the first assertion follows from the trivial bound on Fourier coefficients of cusp forms (see \cite[Prop.~1.3.5]{Sar90}). We then obtain the last assertion by \Cref{Eis-cof_L}.
\end{proof}

In order to compare $C.Z(m)$ with the local intersection number $i_P(C.Z(m))$ for a point $P\in (C\cap Z(m))(k)$, we introduce the notion of \emph{global intersection number} $g_P(m)$ as follows.

\begin{defn}[{\cite[Def.~7.1.3]{MST}}]\label{def_gP}
Let $H$ denote the Hasse-invariant on $\Sh_k$ (i.e., $H=0$ cuts out the non-ordinary locus).
Let $t$ be the local coordinate at $P$ (i.e., $\widehat{C}_{P}=\Spf k[[t]]$) and let $h_P=v_t(H)$, the $t$-adic valuation of $H$ restricted to $\widehat{C}_P$. We define $g_P(m)=\frac{h_P}{p-1}|q_L(m)|$. In particular, $g_P(m)=0$ for $P$ ordinary and $\sum_{P\in (C\cap Z(m))(k)}g_P(m)=|q_L(m)|(\omega.C)$ since $H$ is a section of $\omega^{p-1}$.
\end{defn}

\subsection*{Local intersection number: preparation and non-supersingular points}

We first introduce some notation and reformulate the calculation of local intersection number as a lattice counting problem. 

\begin{para}\label{lattice}
Recall that $P\in (C\cap Z(m))(k)$ for some $m$. Let $\cA/k[[t]]$ denote the pullback of the universal abelian scheme $\cA^\univ$ via $\Spf k[[t]]=\widehat{C}_P\rightarrow \Sh_k$. Let $L_n$ denote the $\bZ$-lattice of special endomorphisms of $\cA \bmod t^n$. By definition, $L_{n+1} \subset L_n$ for every $n\geq 1$, and our assumption that $C$ is not contained in any special divisor yields that $\cap_n L_n = \{ 0\}$. By \cite[Rmk.~7.2.2]{MST}, all $L_n$ have the same rank. Moreover, by \cite[Lem.~4.2.4]{HP}, $P$ is supersingular if and only if $\rk_{\bZ}L_1=b+2$. Since the weight of $\varphi$ on $\bbL_{\cris,P}(W)$ is $0$, then the slope non-zero part of $\bbL_{\cris,P}(W)$ cannot have rank $b+1$; thus if $P$ is not supersingular, then $\rk_{\bZ} L_1\leq b$. 

On the other hand, by \Cref{posdef}, we have a positive definite quadratic form $Q$ on $L_n$ given by $v\circ v =[Q(v)]$ for $v\in L_n$.
By the moduli-theoretic description of the special divisors and the fact that $C$ intersects $Z(m)$ properly (due to the assumption that the image of $C$ does not lie in any $Z(m)$), we have

\begin{equation}\label{localexpm}
    i_P(C.Z(m)) = \sum_{n=1}^{\infty} \#\{v \in L_n\mid Q(v) = m  \}.  
\end{equation}
Note that although for a fixed $m$, the set $\{v\in L_n\mid Q(v)=m\}$ is empty for $n\gg 1$, but this bound on $n$ is in general dependent on $m$. In the work of Chai and Oort \cite{CO06}, 
they use the canonical product structure in the setting $\Sh = \cA_1\times\cA_1$ and work with a sequence of divisors for which the local contributions from any one fixed point is \emph{absolutely bounded}, independent of the special divisor.

By \Cref{asymp_glo}, there exists an absolute constant $c_1$ (which depends only on the curve $C$) such that 
\begin{equation}\label{heightbound}
i_P(C.Z(m)) \leq (C .Z(m)) <  c_1 m^{b/2}.
\end{equation}
\end{para}

We now recall the definitions of the successive minima of the $L_n$ from \cite{Esk}.
\begin{defn}
\begin{enumerate}
    \item  For $i\in\{1,\cdots,r=\rk_{\bZ}L_n\}$, the successive minima $\mu_i(n)$ of $L_n$ is defined as $\displaystyle \inf\{y\in \bR_{>0}\mid\exists v_1,\cdots, v_i\in L_n \text{ linearly independent, and } Q(v_j)\leq y^2, 1\leq j\leq i\}$.

    \item For $n\in \bZ_{\geq 1}, 1\leq i \leq r$, define $a_i(n)=\prod_{j=1}^i \mu_j(n)$; define $a_0(n)=1$.
\end{enumerate}
\end{defn}
The determinant of a quadratic lattice (which is approximately the product of all the successive minima) gives first order control on the number of lattice points with bounded norm -- however, the error term does depend on the lattice in question. In our setting, we must count lattice points of bounded norm in an infinite family of lattices, and so considering the determinants alone doesn't allow us sufficient control across this family of lattices. Indeed, in the example of a formal curve constructed in \S\ref{formal_example}, the error terms involved can get very large, even on average. As seen in \cite{Esk}, the data of each individual successive-minima controls the error term in a way that is uniform across all lattices of a fixed rank, and hence we keep track of this more refined data in our setting of a nested family of lattices. 

We have the following result establishing lower bounds for the $a_i(n)$, which is similar to \cite[Lem.~7.6]{SSTT}.
\begin{lemma}\label{firstmin}
We have that $a_i(n) \gg n^{i/b}$ for $1 \leq i \leq \rk_{\bZ} L_n$.
\end{lemma}
\begin{proof}
Let $0\neq v \in L_n$ be a vector that minimizes the quantity $Q(v)$, and thus $a_1(n) = Q(v)^{1/2}$. Note that $v\in L_n$ implies $v \in L_i$ for every $i\leq n$. Take $m = Q(v)$. Eqn.~\eqref{localexpm} yields that $i_P(C\cdot Z(m)) \geq n$, and then by Eqn.~\eqref{heightbound}, $ n < c_1m^{b/2}$. As $a_1(n)^2 = m$, it follows that $ c_1a_1(n)^b > n$, whence $a_1(n) \gg n^{1/b}$. The bounds for the other $a_i(n)$ follow from the observation that $\mu_i(n) \geq a_1(n)$, and hence $a_i(n) \geq a_1(n)^i$.
\end{proof}

\begin{corollary}\label{truncate}
For $S_X$ defined in \S\ref{setS}, there exists a constant $c_2$ depending only on $C$ such that
\[ 
\sum_{m \in S_X} i_P(C. Z(m)) = \sum_{n=1}^{c_2X^{b/2}} \sum_{m\in S_X} \#\{v \in L_n\mid Q(v) = m\}.
\]
\end{corollary}
\begin{proof}
\Cref{firstmin} implies that there exists a constant $c_2$ only depending on $C$ such that for $n > c_2 X^{b/2}$, we have $a_1(n) > (2X)^{1/2} $. In other words, $\min_{0\neq v\in L_n}Q(v)> 2X$. Then the corollary follows from Eqn.~\eqref{localexpm}.
\end{proof}

We are now ready to bound the local intersection number $i_P(C .Z(m))$ on average over $m$ for $P$ not supersingular, which is the analogue of \cite[Prop.~7.7]{SSTT}.

\begin{proposition}\label{nonssbound}
For $P$ not supersingular, we have that
\[
\sum_{m=1}^{2X} i_P(C. Z(m)) = O(X^{b/2} \log X),
\]
where the implicit constant only depends on $C$.
In particular, $\displaystyle\sum_{m\in S_X} i_P(C.Z(m))=O(X^{b/2} \log X)$.
\end{proposition}
\begin{proof}
By \S\ref{lattice}, we have that $r:=\rk_{\bZ}L_n \leq b$.
 By \cite[Lem.~2.4, Eqns (5)(6)]{Esk} and \Cref{firstmin}, we have 
$$ \sum_{n=1}^{c_2X^{b/2}} \sum_{m = 1}^{2X} \#\{v \in L_n\mid Q(v) = m\} \ll  \sum_{n=1}^{c_2 X^{b/2}}\sum_{i=0}^r \frac{(2X)^{i/2}}{a_i(n)} \ll \sum_{n=1}^{c_2 X^{b/2}} \sum_{i=1}^r \frac{(2X)^{i/2}}{n^{i/b}},$$
where the implicit constant in the first inequality is absolute and the implicit constant in the second inequality only depends on $C$.
For any $1\leq i < b$, we see that 
$$\sum_{n=1}^{c_2 X^{b/2}}  \frac{(2X)^{i/2}}{n^{i/b}} = (2X)^{i/2} \sum_{n=1}^{c_2 X^{b/2}}\frac{1}{n^{i/b}} = O(X^{b/2}),$$
as required.
If $i = b$, the identical calculation yields a bound of $O(X^{b/2}\log X)$. The result then follows directly by \Cref{truncate}.
\end{proof}

\subsection*{Local intersection number at supersingular points}
\begin{para}
For a supersingular point $P\in C$, we break the local intersection number into two parts for a fixed $T\in \bZ_{>0}$ to be chosen later as follows: $\displaystyle \sum_{m \in S_X} i_P(C.Z(m))=i_P(X,T)_{\err}+i_P(X,T)_{\mt}$, where  \[i_P(X,T)_{\err}=\sum_{n=T}^{c_2X^{b/2}} \sum_{m\in S_X} \#\{v \in L_n\mid  Q(v) = m\},\quad i_P(X,T)_{\mt}=\sum_{n=1}^{T-1} \sum_{m\in S_X} \#\{v \in L_n\mid  Q(v) = m\}\] 
and the equality holds due to \Cref{truncate}.
\end{para}

We first bound the error term $i_P(X,T)_{\err}$.

\begin{proposition}\label{sserrorbound}
There exists an absolute constant $c_3>0$ (independent of $X,T$) such that \[i_P(X,T)_{\err} \leq \frac{c_3}{{T^{2/b}}}X^{\frac{b+2}2} + O(X^{(b+1)/2}).\]
\end{proposition}
\begin{proof}
As in the proof of  \Cref{nonssbound}, we have 
$$i_P(X,T)_{\err}\leq \sum_{n=T}^{c_2X^{b/2}} \sum_{m = 1}^{2X} \#\{v \in L_n\mid Q(v) = m\} =\sum_{n=T}^{c_2 X^{b/2}} \sum_{i=0}^{b+2} \frac{(2X)^{i/2}}{n^{i/b}}.$$
As in the proof of \Cref{nonssbound}, by \Cref{firstmin}, we have that $\displaystyle \sum_{n=T}^{c_2X^{b/2}}\frac{(2X)^{i/2}}{n^{i/b}}=O(X^{\frac{b+1}2})$ for all
$1\leq i\leq b+1$. For $i=b+2$, we have $\displaystyle \sum_{n=T}^{c_2 X^{b/2}}\frac{(2X)^{\frac{b+2}{2}}}{n^{(b+2)/b}}< \sum_{n=T}^{\infty}\frac{(2X)^{\frac{b+2}{2}}}{n^{(b+2)/b}}\leq \frac{c_3}{T^{2/(b+2)}}X^{\frac{b+2}{2}}$ for some absolute constant $c_3>0$ by a direct computation.
\end{proof}

In order to bound $i_P(X,T)_{\mt}$, we study the theta series attached to (certain lattices containing) $L'_n$.

\begin{para}\label{sec_lat}
Let $L'_n\subset L_n\otimes \bQ$ be a $\bZ$-lattice such that $L'_n\supset L_n$, $L'_n$ is maximal at all primes $\ell\neq p$, and $L'_n\otimes \bZ_p=L_n\otimes \bZ_p$; we may choose $L'_n\subset L'_1$ and we will assume this for the rest of this section; the quadratic form $Q$ also endows a positive definite quadratic form on $L'_n$. Let $\theta_n$ denote the theta series attached to $L'_n$ and we write its $q$-expansion as $\theta_n(q)=\sum_{m=0}^\infty r_n(m)q^m$. By definition, $r_n(m)\geq \#\{v\in L_n \mid Q(v)=m\}$ and hence $i_P(X,T)_{\mt}\leq \sum_{n=1}^T\sum_{m\in S_X} r_n(m)$.

The theta series $\theta_n$ is a weight $1+b/2$ modular form and we decompose $\theta_n(q)=E_{L'_n}(q)+G_n(q)$, where $E_{L'_n}$ is an Eisenstein series and $G_n$ is a cusp form. Let $q_{L'_n}(m)$ and $g(m)$ denote the $m$-th Fourier coefficients of $E_{L'_n}$ and $\sum_{n=1}^T G_n$ respectively. By \cite[Prop.~1.3.5]{Sar90}, we have $g(m)=O_T(m^{(b+2)/4})$ and thus 
\[i_P(X,T)_{\mt}\leq \sum_{n=1}^T\sum_{m\in S_X} q_{L'_n}(m)+\sum_{m\in S_X} g(m)=\sum_{n=1}^T\sum_{m\in S_X} q_{L'_n}(m)+O_T(X^{1+(b+2)/4}).\]
\end{para}

The following theorem gives explicit formulae of $q_{L'_n}(m)$.

\begin{theorem}[Siegel mass formula]\label{Ecoeff}
Given any $L'\subset L'_1$ sublattice such that $L'\otimes \bZ_\ell=L'_1\otimes \bZ_\ell$ for all $\ell\neq p$, let $q_{L'}(m)$ be the $m$-th Fourier coefficient of the Eisenstein part of the theta series attached to $L'$. 
\begin{enumerate}

\item For $b$ even, 
\[q_{L'}(m)=\frac{2^{1+b/2}\pi^{1+b/2}m^{b/2}\sigma_{-b/2}(m,\chi_{(-1)^{1+b/2}4\det L'_1})}{\sqrt{|L'^\vee/L'|}\Gamma(1+b/2)L(1+b/2,\chi_{(-1)^{1+b/2}4\det L'_1})}\prod_{\ell \mid 2 det L'_1}\delta(\ell, L',m).\]

\item For $b$ odd, $q_{L'}(m)$ equals
\[\frac{2^{1+b/2}\pi^{1+b/2}m^{b/2}L((b+1)/2,\chi_{\cD'})}{\Gamma(1+b/2)\sqrt{|L'^\vee/L'|}\zeta(b+1)}\left(\sum_{d\mid f}\mu(d)\chi_{\cD}(d)d^{-(b+1)/2}\sigma_{-b}(f/d)\right)\prod_{\ell\mid 2\det L'_1}\Big(\delta(\ell,L',m)/(1-\ell^{-1-b})\Big),\]
where we write $m=m_0f^2$, where $\gcd(f,2\det L'_1)=1$ and $v_\ell(m_0)\in \{0,1\}$ for all $\ell\nmid 2\det L'_1$,  $\mu$ is the Mobius function, and $\cD'=(-1)^{(b-1)/2}2m_0\det L'_1$.

\end{enumerate}
\end{theorem}
\begin{proof}
This theorem is a direct consequence of the Siegel mass formula by the same proof in \cite[Thm.~4.2.2]{MST}.
\end{proof}
We may apply this theorem to $L'=L'_n$ in \S\ref{sec_lat} because all $L'_n$ are maximal at $\ell \neq p$ and thus $L'_n\otimes \bZ_\ell=L'_1\otimes \bZ_\ell$. 

\begin{lemma}\label{den_sm}
For $p\nmid m$, we have that
\[\frac{q_{L'_n}(m)}{|q_{L}(m)|}\leq \frac{2}{\sqrt{|(L'_n\otimes \bZ_p)^\vee/L'_n\otimes \bZ_p|}(1-p^{-[(b+2)/2]})}.\]
Moreover, if $P$ is superspecial, then \[\frac{q_{L'_1}(m)}{|q_{L}(m)|}\leq \frac{1+p^{-1}}{p(1-p^{-[(b+2)/2]})}.\]
\end{lemma}
\begin{proof}
By \cite[Rmk.~7.2.5]{HP}, $L\otimes \bQ_\ell\cong L'_n\otimes \bQ_\ell$ as quadratic spaces for all $\ell\neq p$;  since $L,L'_n$ are both maximal at $\ell\neq p$, then $L\otimes \bZ_\ell\cong L'_n\otimes \bZ_\ell$ as $\bZ_\ell$-quadratic lattices for all $\ell\neq p$. Moreover, since $p\nmid m$, then by \Cref{Eis-cof_L,Ecoeff}, we have that
 \[\frac{q_{L'_n}(m)}{|q_{L}(m)|}=\frac{\delta(p, L'_n,m)}{\sqrt{|(L'_n\otimes \bZ_p)^\vee/L'_n\otimes \bZ_p|}(1-\chi_{(-1)^{1+b/2}4\det L}(p)p^{-1-b/2})} \text{ if } 2\mid b;\]
\[    \frac{q_{L'_n}(m)}{|q_{L}(m)|}=\frac{\delta(p, L',m)(1-\chi_\cD(p)p^{-(b+1)/2})}{\sqrt{|(L'_n\otimes \bZ_p)^\vee/L'_n\otimes \bZ_p|}(1-p^{-1-b})}\text{ if } 2\nmid b.\]
Therefore, \[\frac{q_{L'_n}(m)}{|q_{L}(m)|}\leq \frac{\delta(p,L'_n,m)}{\sqrt{|(L'_n\otimes \bZ_p)^\vee/L'_n\otimes \bZ_p|}(1-p^{-[(b+2)/2]})}.\]

For the first assertion, it remains to show
that $\delta(p, L'_n,m)\leq 2$.
Write the quadratic form $Q$ on $L'_n$ into the diagonal form $\sum_{i=1}^{b+2} a_i x_i^2$ with $a_i\in \bZ_p$ and we may assume that there exists $a_i$ such that $p\nmid a_i$; otherwise $\delta(p,L'_n,m)=0$ then we are done. Now let $\widetilde{L}'_n$ denote the quadratic lattice over $\bZ$ with the quadratic form $\widetilde{Q}$ given by $\sum_{1\leq i\leq b+2, p\nmid a_i} a_i x_i^2$. By \cite[Rmk.~3.4.1(a), Lem.~3.2]{Han04}, we have that \[\delta(p, L'_n,m)=p^{-b-1}\#\{v\in L'_n/pL'_n\mid Q(v)\equiv m \bmod p\}=p^{1-\rk \widetilde{L}'_n}\#\{v\in \widetilde{L}'_n/p\widetilde{L}'_n\mid Q(v)\equiv m \bmod p\},\] where the last equality follows from definition.
If $\rk \widetilde{L}'_n\geq 3$, the $\bF_p$-quadratic form $\widetilde{Q} \bmod p$ is isotropic, then we may write $\widetilde{Q}\bmod p=xy+Q'(z)$. For $ x\in \bF_p^\times$, for any value of $z$, there are at most one $y\in \bF_p$ such that $\widetilde{Q}\equiv m \bmod p$, this yields $(p-1)p^{\rk \widetilde{L}'_n-2}$ solutions; for $x=0$, there are at most $p^{\rk \widetilde{L}'_n-1}$ solutions. Therefore $p^{1-\rk \widetilde{L}'_n}\#\{v\in \widetilde{L}'_n/p\widetilde{L}'_n\mid Q(v)\equiv m \bmod p\}<2$. If $\rk \widetilde{L}'_n=1,2$,  \cite[Table 1]{Han04} implies that $p^{1-\rk \widetilde{L}'_n}\#\{v\in \widetilde{L}'_n/p\widetilde{L}'_n\mid Q(v)\equiv m \bmod p\}\leq 2$. Thus we conclude that $\delta(p,L'_n,m)\leq 2$.

For the second assertion, by definition, for a superspecial point, we have $\sqrt{|(L'_1\otimes \bZ_p)^\vee/L'_1\otimes \bZ_p|}=p^{t_P/2}=p$ and thus it remains to show that $\delta(p, L',m)\leq 1+p^{-1}$.
Since $p^2 || \disc L'_1$ and $\rk L'_1=b+2\geq 5$, then $\rk\widetilde{L}'_1\geq 3$.
If $\rk\widetilde{L}'_1= 3,4$, then the desired bound for $\delta(p,L'_1,m)$ follows from \cite[Table 1]{Han04}. For $\rk\widetilde{L}'_1\geq 5$, we use an inductive argument.
Indeed, $\widetilde{L}'_1$ is isotropic and the discussion for the first assertion implies that there exists an $\bF_p$-lattice $L_0$ (the one corresponds to $Q'$) such that $\rk L_0=\rk \widetilde{L}'_1$ and \[p^{1-\rk \widetilde{L}'_1}\#\{v\in \widetilde{L}'_1/p\widetilde{L}'_1\mid Q(v)\equiv m \bmod p\}=(1-p^{-1})+p^{-\rk L_0}\#\{v\in L_0\mid Q'(v)\equiv m \bmod p\}.\]
By the inductive hypothesis, we have $p^{1-\rk L_0}\#\{v\in L_0\mid Q'(v)\equiv m \bmod p\}\leq 1+p^{-1}$ and then we conclude by the above equation that \[\delta(p,L'_1,m)=p^{1-\rk \widetilde{L}'_1}\#\{v\in \widetilde{L}'_1/p\widetilde{L}'_1\mid Q(v)\equiv m \bmod p\}\leq 1+p^{-1}.\qedhere\]
\end{proof}

\begin{proposition}\label{ssmain}
There exists an absolute constant $0<\alpha<1$ such that 
\[i_P(X,T)_{\mt}=\alpha\sum_{m\in S_X} g_P(m) + O_T(X^{1+(b+2)/4}).\]
\end{proposition}

\begin{proof}
For brevity, we set $h=h_P$ in \Cref{def_gP}; by \S\ref{sec_lat}, it suffices to show that \[\displaystyle \sum_{n=1}^T \frac{q_{L'_n}(m)}{g_P(m)}\leq \alpha \frac{h}{p-1}\] for some constant $0<\alpha<1$. We will prove this claim using the decay statements from Section \ref{sec_heuristic} by a similar computation as in \cite[Cor.~7.2.4, Lem.~8.2.2]{MST}.  
We will apply these here using the fact that $L_n\otimes \bZ_p = L'_n\otimes \bZ_p$
and the identity 
$$\sqrt{|(L'_n\otimes \bZ_p)^\vee/L'_n\otimes \bZ_p|} = \sqrt{|(L'_1\otimes \bZ_p)^\vee/L'_1\otimes \bZ_p|} \cdot |L'_1/L'_n|.$$

If $P$ is an nonsuperspecial supersingular point, then by definition, $\sqrt{|(L'_1\otimes \bZ_p)^\vee/L'_1\otimes \bZ_p|}\geq p^2$.
Moreover, by the above identity and \Cref{cor_decay_sg}, for $h_r+1\leq n\leq h_{r+1}, r\in \bZ_{\geq 0}$, we have $\sqrt{|(L'_n\otimes \bZ_p)^\vee/L'_n\otimes \bZ_p|}\geq p^{4+2r}$. Thus by \Cref{den_sm},
    \[\sum_{n=1}^\infty \frac{q_{L'_n}(m)}{g_P(m)}\leq \frac{2}{1-p^{-2}}\left(\frac{h(p^{-1}+1)}{p^2}+\frac{hp}{p^4}+\frac{hp^2}{p^6}+\cdots\right)\leq \frac{h}{p-1}\cdot\frac{2(p^2-p+1)}{p(p^2-1)}\leq \frac{11}{12}\cdot \frac{h}{p-1}\]
    for all $p\geq 3$.

If $P$ is superspecial and statement (1) in \Cref{cor_decay_ssp} holds for $P$, there for $a = h/2$ such that for $h'_{r-1}+ap^r+1\leq n\leq h'_{r}, r\in \bZ_{\geq 0}$, we have $\sqrt{|(L'_n\otimes \bZ_p)^\vee/L'_n\otimes \bZ_p|}\geq p^{2+2r}$, and for $h'_r+1\leq n\leq h'_{r}+ap^{r+1}, r\in \bZ_{\geq 0}$, we have $\sqrt{|(L'_n\otimes \bZ_p)^\vee/L'_n\otimes \bZ_p|}\geq p^{3+2r}$.
Thus for $b\geq 4$, we have $1-p^{-[(b+2)/2]}\geq 1-p^{-3}$ and by \Cref{den_sm},
\begin{align*}
    \sum_{n=1}^\infty \frac{q_{L'_n}(m)}{g_P(m)} & \leq  \frac{1+p^{-1}}{p(1-p^{-3})}(a(1+p^{-1}))+\frac{2(h-a)}{p^2(1-p^{-3})}+\frac{2ap}{p^3(1-p^{-3})}+\frac{2(h-a)p}{p^4(1-p^{-3})}+\cdots \\
    &\leq  \frac{h}{p-1}\left(\frac{(p+1)^2}{2(p^2+p+1)}+\frac{2p}{p^2+p+1}(1+p^{-1}+p^{-2}+\cdots)\right) \leq \frac{61}{62}\frac{h}{p-1}
\end{align*} 
for all $p\geq 5$. 
For $b=3$, we remark that the proof of \cite[Thm.~5.1.2]{MST} applies to all $(L,Q)$ with $b=3$ and $L$ self-dual at $p$, not just the one associated to principally polarized abelian surfaces. Thus in this case, there is a rank $3$ submodule which decays rapidly in the sense of \Cref{decaydef}. Thus the computation in \cite[\S 9.2 small $n$'s]{MST} proves that $\displaystyle \sum_{n=1}^\infty \frac{q_{L'_n}(m)}{g_P(m)}\leq \frac{11}{12}\frac{h}{p-1}$ for all $p\geq 5$.

If $P$ is superspecial and statement (2) in \Cref{cor_decay_ssp} holds for $P$, then there exists a constant $a\leq h/2$ such that for $ap^{-1}+a+1\leq n \leq ap^{-1}+h$, we have $\sqrt{|(L'_n\otimes \bZ_p)^\vee/L'_n\otimes \bZ_p|}\geq p^{2}$, and for $h'_r+1\leq n\leq h'_{r+1}, r\in \bZ_{\geq 0}$, we have $\sqrt{|(L'_n\otimes \bZ_p)^\vee/L'_n\otimes \bZ_p|}\geq p^{4+2r}$.
Thus by \Cref{den_sm}
\begin{align*}
    \sum_{n=1}^\infty \frac{q_{L'_n}(m)}{g_P(m)} & \leq \frac{1+p^{-1}}{p(1-p^{-2})}(a(1+p^{-1}))+\frac{2(h-a)}{p^2(1-p^{-2})}+\frac{2hp}{p^4(1-p^{-2})}+\frac{2hp^2}{p^6(1-p^{-2})}+\cdots\\
    &\leq \frac{h}{p-1}\left(\frac{1+p^{-1}}{2}+(p+1)^{-1}+\frac{2}{p+1}(p^{-1}+p^{-2}+p^{-3}+\cdots)\right)\leq \frac{17}{20}\frac{h}{p-1}
\end{align*} 
for all $p\geq 5$.
\end{proof}

\begin{theorem}\label{ssbound}
There is an absolute constant $0<\alpha'<1$ such that 
for $S_X$ defined in \S\ref{setS} and for any $P\in C(k)$ supersingular, we have that
\[
\sum_{m \in S_X} i_P(C.Z(m)) = \alpha' \sum_{m\in S_X} g_P(m) +O(X^{(b+1)/2}).
\]
\end{theorem}
Indeed, we may state this theorem without assuming $P$ is supersingular since the statement here for non-supersingular $P$ is a weaker version of \Cref{nonssbound}.
\begin{proof}
We may take $\alpha'$ to be any absolute constant such that $1>\alpha'>\alpha$, where $\alpha$ is given in \Cref{ssmain}. Then we choose $T\in \bZ_{>0}$ such that $\frac{c_3}{T^{2/b}}X^{1+b/2}\leq (\alpha'-\alpha) \sum_{m\in S_X} g_P(m)$; such $T$ exists since $\sum_{m\in S_X} g_P(m)\asymp X^{1+b/2}$ by \Cref{asymp_glo}. Once we fix such a $T$, which may be chosen only depending on $\alpha, \alpha', S$ (not $S_X$), the desired bound follows from \Cref{sserrorbound,ssmain}.
\end{proof}

Now we combine the previous results in this section to prove \Cref{thm_int}.
\begin{proof}[Proof of \Cref{thm_int}]
If there were only finitely many points $P$ in $C\cap (\cup_{p\nmid m} Z(m))(k)$, then by \Cref{nonssbound}, \Cref{ssbound}, and \Cref{def_gP}, we have that
\[\sum_{m\in S_X} C.Z(m)=\sum_{m\in S_X}\sum_{P\in C\cap (\cup_{m\in S_X} Z(m))(k)} i_P(C.Z(m))=\alpha'(\omega.C)\sum_{m\in S_X}|q_L(m)| +O(X^{(b+1)/2}),\]
which contradicts \Cref{asymp_glo}.
\end{proof}

\section{Application to the Hecke orbit problem}\label{sec_Hecke}
We prove \Cref{thm_Hecke} using \Cref{thm_int} in this section. For $x\in \Sh_{\bF_p}(k)$, where $k=\overline{\bF}_p$, we use $T_x$ to denote the set of all prime-to-$p$ Hecke translates of $x$ and let $\overline{T_x}$ denote the Zariski closure of $T_x$ in $\Sh_k$.
We will prove that for $x$ ordinary, $\overline{T_x}=\Sh_k$ by a case-by-case discussion depending on whether we know \emph{a priori} that the Zariski closure $\overline{T_x}^\BB$ of $\overline{T_x}$ in the Bailey--Borel compactification $\Sh^\BB_k$ of $\Sh_k$ hits the boundary $\Sh^\BB_k\setminus \Sh_k$. We will prove the GSpin case first and in the end of this section, we will remark on how to adapt the same line of ideas to the unitary case (see \Cref{rmk_unitary}).

\begin{para}
Recall from \S\ref{def_Sh} that the quadratic lattice $(L,Q)$ is self-dual at $p$ and the level we pick is hyperspecial at $p$. By \cite[Thm 3]{MP19}, the canonical integral model $\Sh$ of the Hodge type Shimura variety $Sh$ admits a projective normal compactification $\Sh^\BB$ over $\bZ_{(p)}$ such that $\Sh^\BB_\bQ$ is the Bailey--Borel/minimal compactification $Sh^\BB$ of $Sh$; moreover, the classical stratification of $Sh^\BB$ by quotients by finite groups of Shimura varieties of Hodge type extends to a stratification of $\Sh^\BB$ by quotients by finite groups of integral models of these Shimura varieties; in particular, the stratification on $\Sh^\BB$ is flat. In addition, the Hecke action of $G(\bA^p_f)$ on $\Sh$ extends naturally to an action on $\Sh^\BB$. Since all these Hecke actions are algebraic correspondences, we have that $\overline{T_x}$ and $\overline{T_x}^\BB$ are stable under the Hecke action of $G(\bA^p_f)$.

Once we choose an admissible complete smooth cone decomposition, by \cite[Thms 1, 2, 4.1.5]{MP19}, the canonical integral model $\Sh$ admits a smooth toroidal compactification $\Sh^{\tor}$ such that $\Sh^{\tor}_\bQ$ is the toroidal compactification of $Sh$ constructed in \cite{AMRT,Pink}. Moreover, the stratification of $\Sh^{\tor}_\bQ$ by quotients by finite groups of mixed Shimura varieties extends to a stratification of $\Sh^{\tor}$ with all boundary components being flat divisors and the formal completions of $\Sh^{\tor}$ along the boundary components of the same shape as that of $\Sh^{\tor}_{\bQ}$. There is also a natural map $\pi: \Sh^{\tor}\rightarrow \Sh^\BB$ which extends the identity map on $\Sh$ and this map is compatible with the stratifications. 

Thus for the rest of this section, we follow \cite[\S\S 3.2, 3.3]{BZ} and \cite[\S 4]{Zemel} for the explicit descriptions of $\Sh^{\tor}_{\bC}, \Sh^\BB_{\bC}$
and use it for $\Sh^{\tor}_{\bF_p}$ and $\Sh^{\BB}_{\bF_p}$ by the work of Madapusi Pera summarized above. In particular, the boundary components (cusps) in $\Sh^\BB_{\bF_p}$ are either $0$-dimensional or $1$-dimensional.
\end{para}

\subsection*{$0$-dimensional cusps}
We first prove \Cref{thm_Hecke} assuming that $\overline{T_x}^\BB$ contains a $0$-dimensional cusp in $\Sh^\BB_{\bF_p}$.  The argument for this is essentially the same as in \cite[\S 2]{Chai95}, and we will follow the approach there closely, indicating the places where modifications are necessary.  The idea of the argument in \cite{Chai95} is as follows.  Given a $0$-dimensional cusp, we study the Hecke-stabilizer of the cusp and its action on the formal neighborhood to argue that any invariant subscheme which is not $\Sh^\BB_{\bF_p}$ is contained in the boundary. 

\begin{para}\label{dim0-coord}
\emph{Coordinates.}
To describe the action in coordinates, we follow the notation in \cite{MP19} and refer to section $2$ there for more details.
We will work with level structure $K_n$ given by embedding into $\mathrm{GSp}$ and restricting the full level $\ell^n$ structure there; let $\Sh_{n,k}$ denote the corresponding special fiber over $k$ of the canonical model of the Shimura variety.   Given a zero-dimensional cusp $x_n$, we fix a cusp label representative $\Phi$ describing the cusp, which includes the data of an admissible parabolic subgroup $P \subset G_{\bQ}$.  As $n$ varies, $\Phi$ defines a compatible system of cusps $\{x_n\}$ in the inverse system $\{\Sh^{\BB}_{n,k}\}$ and a point $x \in \lim_{\leftarrow}\Sh^{\BB}_{n,k}$.

Let $U_P$ denote the unipotent radical of $P$ and $W \subset U_P$ denote the center of $U_P$.
By \cite[\S 2.1.11, \S 2.1.16]{MP19},  we can associate to $K_n$ a lattice
$\mathbf{B}_{K_n} \subset W(\mathbb{Q})$ with dual lattice $\mathbf{S}_{K_n} \subset W(\mathbb{Q})^{\vee}$ and an arithmetic group $\Delta_{K_n}$ acting on $\mathbf{B}_{K_n}$.  We also have an open self-adjoint convex cone $\mathbf{H} \subset W(\mathbb{R})$ preserved by $\Delta_K$ by \cite[\S 2.1.6, \S 2.1.16]{MP19}.\footnote{In \cite{MP19} there is a twist by $2\pi i$ which we are suppressing.}
In terms of this data, by \cite[Cor.~5.1.8, Cor.~5.2.8]{MP19}, the complete local ring of $\Sh^{\BB}_{n,k}$ at $x_n$ is given by the ring of invariants
$$R_{\ell^{n}} = k[[q^{\lambda}]]_{\lambda \geq 0}^{\Delta_{K_n}}$$
where $\lambda \geq 0$ denotes elements of $\mathbf{S}_{K_n}$ which have non-negative pairing with $\mathbf{H}$. 
If we pass to the inverse limit, we get the ring 
$$R_{\ell} = \cup_{n} R_{\ell^{n}}.$$

In order to study Hecke-stable subvarieties, rather than study the full $G(\bA^p_f)$-action, it suffices to study the action of $\bbB_{\ell}:=\mathbf{B}_{K_n}\otimes \bZ[1/\ell] \subset W(\bQ)$ which fixes the point $x$ in the inverse limit and therefore acts on the ring $R_{\ell}$.\footnote{Note that by definition in \cite[\S 2.1.11]{MP19}, $\mathbf{B}_{K_n}\otimes \bZ[1/\ell]$ is independent of $n$ for our $K_n$.}  Given $T \in \bbB_\ell$, its action on $f \in R$ is given by the formula
$$f = \sum_{\lambda} a_{\lambda}q^{\lambda} \mapsto T(f) = \sum_{\lambda} \mathbf{e}((T,\lambda)) a_{\lambda} q^{\lambda}.$$
Here, $(T,\lambda) \in \bZ[1/\ell]$ is the pairing of $T \in W(\mathbb{Q})$ and $\lambda \in W(\mathbb{Q})^{\vee}$ and $\mathbf{e}: \bZ[1/\ell] \rightarrow \mu_{\ell^{\infty}}(k)$ is the group homomorphism given by taking the compatible system of primitive $\ell^n$-th roots of unity determined by the choice of cusp and the full level structure.
\end{para}

\subsubsection*{Invariant ideals of the complete local ring}

In terms of the above coordinates, the main proposition is the following, based on Proposition $2$ of \cite{Chai95}.
\begin{proposition}\label{pf_dim0} 
Let $I_{\ell^{n}} \subset R_{\ell^{n}}$ be a nonzero ideal such that $I = I_{\ell^{n}}R$ is stable under the action of $\bbB_\ell$.  
Then $\operatorname{Spf} R_{\ell^{n}}/I_{\ell^{n}}$ is contained in the formal completion of the boundary of $\Sh^{\BB}_{n,k}$.
\end{proposition}

Again, we merely summarize the argument from \cite{Chai95}.  Rather than work directly with $R$, it is more convenient to pass to a toroidal compactification $\Sh^{\tor}_{n,k}$.
The choice of compactification in particular specifies a smooth cone decomposition of the rational closure of the cone $\mathbf{H}$. By \cite[\S\S 5.1.5, 2.1.17, 2.1.18]{MP19},  the formal completion of $\Sh^{\tor}_{n,k}$ along 
the preimage of $x_n$ is covered by affine formal subschemes $S_\alpha$ parametrized by cones $\sigma_\alpha \subset \mathbf{H}$, with coordinate ring
given by the completion $R_{\sigma, \ell^n}$ of 
$$\oplus_{\lambda \in \mathbf{S}_K \cap \sigma^{\vee}} k[[q^{\lambda}]].$$
along the ideal generated by non-invertible elements of the monoid $\mathbf{S}_K \cap \sigma^{\vee}$.
Let $J_\sigma$ denote the ideal generated by $q^\lambda$ where $\lambda >0$ on $\overline{\sigma} \cap \mathbf{H}$, which is the ideal of the toroidal boundary.

Given $f \in R_{\sigma, \ell^n}$, we say that $f$ has a leading term with respect to $J_\sigma$ if it is a pure monomial $aq^{\lambda}, a\in k^\times$ multiplied by an element in $1+J_\sigma$.
The main claim to be proven is that, given $I$ as in \Cref{pf_dim0}, for each cone $\sigma$ in the decomposition of $\mathbf{H}$, there exists
$f_\sigma \in I$ which has a leading term with respect to $J_\sigma$.  
The proof of this in \cite[pp.~455-456]{Chai95} is purely cone-theoretic, so applies identically in our setting.  The key step (\cite[Lem.~1]{Chai95}) is a cancellation algorithm:  given $f \in I$, and a finite collection $S = \{\lambda_0, \dots, \lambda_r\}$ for which $f$ has nonzero coefficients, there exists $g \in I$ given by a finite linear combination of translates $T(f)$ for which the corresponding coefficients are all zero except for $\lambda_0$.  This is proven using the explicit formula for $T(f)$.

\subsection*{$1$-dimensional cusps}
We now treat the case when $\overline{T_x}^\BB$ contains at least one $k$-point in a $1$-dimensional cusp. We chose an admissible complete smooth cone decomposition and let $\overline{T_x}^{\tor}$ denote the Zariski closure of $T_x$ in $\Sh^{\tor}_k$. We will show that either $\overline{T_x}^{\tor}=\Sh^{\tor}_k$ or $\dim_k \overline{T_x}^\BB\setminus\overline{T_x}=0$ and $\dim \overline{T_x}^\BB\geq 2$.

\begin{para}\label{Hecke-1dim}
By the first paragraph in \cite[\S 3.3]{BZ}, there is a unique cone decomposition for a given $1$-dimensional cusp and the boundary strata in $\Sh^{\tor}$ over $1$-dimensional cusps in $\Sh^\BB$ are canonical. Thus by \cite[Prop.~2.1.19, \S 4.1.12, Prop.~4.1.13]{MP19}, the Hecke action of $G(\bA^p_f)$ on $Sh$ extends uniquely to $\pi^{-1}(\Sh^\BB \setminus \{0\text{-dim cusps}\})$ satisfying certain explicit description of this action on formal completion along boundary components given in \cite[\S 4.1.12]{MP19}. Set $\overline{T_x}^{\tor,1}:=\overline{T_x}^{\tor}\cap \pi^{-1}(\Sh^\BB \setminus \{0\text{-dim cusps}\})$. Then for any $g\in G(\bA^p_f)$, we have $g.\overline{T_x}^{\tor,1}\supset g.\overline{T_x}=\overline{T_x}$ and thus $g.\overline{T_x}^{\tor,1}=\overline{T_x}^{\tor,1}$. In particular, for any $y\in \overline{T_x}^{\tor,1}(k)$, the Zariski closure of all prime-to-$p$ Hecke orbits of $y$ in $\Sh^{\tor}_k$ is contained in $\overline{T_x}^{\tor}$. In particular, we will study the Hecke action on a boundary point $y\in (\overline{T_x}^{\tor,1}\setminus\overline{T_x})(k)$ in order to deduce certain properties for $\overline{T_x}$.
\end{para}

\begin{para}\label{1-dimAVtorsor}
Let $\Upsilon$ be a $1$-dimensional cusp in $\Sh^\BB$. We first follow \cite[\S 4]{Zemel} to give an explicit description of $\pi^{-1}(\Upsilon(\bC))$. By \cite[Prop.~4.3, Thm.~4.5]{Zemel} (see also \cite[Lem.~3.18, Prop.~3.19]{BZ}), 
up to quotient by a finite group, $\pi^{-1}(\Upsilon(\bC))$ is a torsor under an abelian scheme over the modular curve (with suitable level); moreover, let $I\subset L$ be a (saturated) isotropic subspace corresponding to $\Upsilon$ and set $\Lambda=I^\perp/I$, then the above mentioned abelian scheme is given by $\cE\otimes_\bZ \Lambda $, where $\cE$ is the universal family of elliptic curves over the modular curve. Therefore, by \cite[Thm.~4.1.5]{MP19}, $\pi^{-1}(\Upsilon)$ is a quotient by a finite group of a $\cE\otimes \Lambda$-torsor over the modular curve. 
\end{para}

Since the prime-to-$p$ Hecke actions on $\pi^{-1}(\Upsilon)$ is the natural extension of the Hecke actions on $\pi^{-1}(\Upsilon(\bC))$, we first study the Hecke orbits of $y\in \pi^{-1}(\Upsilon(\bC))$.
\begin{proposition}\label{onedimbdrycharzero}
Notation as in \S\ref{1-dimAVtorsor}. 
For $y\in \pi^{-1}(\Upsilon(\bC))$, let $T_{y,\ell}$ denote the set of all $\ell$-power Hecke translates of $y$. Then $T_{y,\ell}$ contains all the translates of $y$ by $\ell$-power torsion points in $\cE_{\pi(y)}\otimes \Lambda$, where $\cE_{\pi(y)}$ denotes the fiber of $\cE$ at $\pi(y)$ (in the modular curve) and recall that $\pi^{-1}(\pi(y))$ is an $\cE_{\pi(y)}\otimes \Lambda$-torsor.
\end{proposition}

\begin{proof}
Recall that $I\subset L$ denotes the (saturated) isotropic subspace corresponding to $\Upsilon$; let $P\subset G_{\bQ}=\GSpin(L\otimes \bQ)$ denote the maximal parabolic which is the stabilizer of $I$, let $U$ denote the unipotent radical of $P$, and let $W$ denote the center of $U$; set $\cV:=U/W$. By \cite[\S 2.1.10]{MP19}, $\cV(\bQ)$ acts on the on the $\cE\otimes \Lambda$-torsor $\pi^{-1}(\Upsilon(\bC))$ over $\Upsilon(\bC)$ and the explicit form of this action is given by \cite[Lem.~3.11]{BZ}. 

More precisely, following \cite[\S 4]{Zemel}, we pick a $\bZ$-basis $\{z,w\}$ of $I$; Using the bilinear form $[-,-]$ induced by the quadratic form $Q$, we naturally identify the dual $L^\vee\subset V=L\otimes \bQ$. Let $\zeta,\omega \in L^\vee$ be a basis dual to $(z,w)$.\footnote{This means that $[-,-]$ induces an isomorphism between $\Span_{\bZ}\{\zeta,\omega\}$ and $\Hom(I,\bZ)$ with $\zeta,\omega$ mapping to the basis dual to $\{z,w\}$; the existence of such a basis is given by \cite[Def.~2.1, Lem.~2.2]{Zemel}.} Recall that $\Upsilon$ is the modular curve with suitable level and let $\tau$ be a lift of $\pi(y)\in \Upsilon(\bC)$ to the upper half plane. Then by \cite[Thm.~4.5, proof of Prop.~4.3, Eqns (25)(26)]{Zemel}, $\pi^{-1}(\pi(y))$ is isomorphic to the quotient of $W_\bC^{1,\tau}:=\{\zeta'+\tau\omega'+e\mid e\in \Lambda\otimes_\bZ \bC\}\subset V_{\bC}/I\otimes_{\bZ}\bC$ by the translation action of $(\Lambda \oplus \tau\Lambda)$. By \cite[Lem.~3.11]{BZ}, $a+b\tau \in \cV(\bZ[1/\ell])\cong \Lambda\otimes \bZ[1/\ell] \oplus \tau \Lambda\otimes \bZ[1/\ell]$ acts by sending $\zeta+\tau\omega+e$ to $\zeta+\tau\omega+(e+a+b\tau)$.
Since $U$ is the Heisenberg group described in \cite[\S 1, Prop.~1.6, Cor.~1.9]{Zemel}, then all elements in $\cV(\bZ[1/\ell])$ lift to elements in $U(\bZ[1/\ell])$; thus the Hecke translates of $y$ by elements in $U(\bZ[1/\ell])$ contains all translates of $y$ by $\ell$-power torsion points in $\cE_{\pi(y)}\otimes \Lambda$.
\end{proof}

\begin{corollary}\label{onedimbdry}
Let $y \in \Upsilon(k)$, where $\Upsilon$ is a $1$-dimensional cusp of $\Sh^{\BB}$, and let $T_{y,\ell}$ denote the set of all $\ell$-power Hecke translates of $y$. Then $T_{y,\ell}\cap \pi^{-1}(\pi(y))$ is Zariski dense in $\pi^{-1}(\pi(y))$.
\end{corollary}
\begin{proof}
By \Cref{onedimbdrycharzero} and its proof, the Hecke action of $U(\bZ[1/\ell])$ on $y$ is given by translates of $y$ by $\ell$-power torsion points $\cE_{\pi(y)}\otimes \Lambda$.\footnote{Following the description in \cite[\S 4.1.12]{MP19}, the extension of Hecke translate from characteristic $0$ as the translation action of $\ell$-power torsion points is still the translation action.} Note that $\pi^{-1}(\pi(y))\simeq \cE_{\pi(y)}\otimes \Lambda$ as varieties over $k$ (this isomorphism is non-canonical) and thus the union of the translates of $\ell$-power torsion points is Zariski dense in $\pi^{-1}(\pi(y))$ since the set of $\ell$-power torsion points of an abelian variety over $k$ is Zariski dense.
\end{proof}

\begin{corollary}\label{pf_Hecke-1dim}
Recall that $x\in \Sh_{\bF_p}(k)$ ordinary and assume that $b\geq 3$. If $\overline{T_x}^\BB$ contains a $k$-point which lies on a $1$-dimensional cusp of $\Sh_{\bF_p}^\BB$. Then
either (1) $\overline{T_x}=\Sh_k$ or (2) $\dim_k \overline{T_x}^\BB\setminus\overline{T_x}=0$ and $\dim \overline{T_x}^\BB\geq 2$.
\end{corollary}
\begin{proof}
Since $\overline{T_x}^\BB$ is stable under Hecke action, then for any $1$-dimensional cusp $\Upsilon$, we have that  $\overline{T_x}^\BB\cap \Upsilon$ is stable under the Hecke action of $\GL_2(\bA^p_f)$ on $\Upsilon$. Thus $\overline{T_x}^\BB\cap \Upsilon=\Upsilon$ or $\dim_k \overline{T_x}^\BB\cap \Upsilon=0$. 

If there exists an $\Upsilon$ such that $\overline{T_x}^\BB\cap \Upsilon=\Upsilon$, then $\pi(\overline{T_x}^{\tor})=\overline{T_x}^\BB\supset \Upsilon$.
By \Cref{onedimbdry} and \S\ref{Hecke-1dim}, we have that $\overline{T_x}^{\tor}\supset \pi^{-1}(\Upsilon)$ and thus $\dim_k \overline{T_x}\geq \dim_k \pi^{-1}(\Upsilon)=\dim \Sh_k$. Moreover, since the $G(\bA^p_f)$-action transitively on $\pi_0(\Sh_k)$ (by the definition of canonical integral models and \cite[Lem.~2.2.5]{Kisin}), then $\overline{T_x}=\Sh_k$.

If for any $1$-dimensional cusp $\Upsilon$, we have $\dim_k \overline{T_x}^\BB\cap \Upsilon=0$, then $\dim_k \overline{T_x}^\BB\setminus\overline{T_x}=0$. On the other hand, by the assumption, there exists $y'\in \Upsilon(k)$ for some $\Upsilon$ such that $y'\in \overline{T_x}^\BB$; then there exists $y\in \pi^{-1}(\Upsilon)(k)$ such that $y\in \overline{T_x}^{\tor}$ and $\pi(y)=y'$. By \Cref{onedimbdry}, we have $\dim_k \overline{T_x}\geq 1+ \overline{T_{y,\ell}}=b-1\geq 2$. Thus we conclude that (2) holds.
\end{proof}

\subsection*{Proof of the Hecke orbit conjecture}
We first recall some results on Hecke orbits which we will need. As the results and their proofs are standard, we will content ourselves with only a sketch of their proofs.

\begin{lemma}\label{basichecke}
Let $f:Sh_1 \rightarrow Sh_2$ be a morphism of Shimura varieties of Hodge type with hyperspecial level at $p$ and let $G_i, i=1,2$ denote the reductive group of $Sh_i$. Let $\Sh_i$ denote the canonical integral model of $Sh_i$ and then $f$ extends naturally as $f:\Sh_1\rightarrow \Sh_2$. 
 Let $X\subset \Sh_{2,k}$ be a subvariety that intersects the ordinary locus (here we assume that the ordinary locus in $\Sh_{2,k}$ is not empty), and let $\overline{T_X}$ denote the Zariski closure of the Hecke orbit $T_X$ of $X$ with respect to the Hecke action by $G_2(\bA^p_f)$. Then
\begin{enumerate}
%
    \item for any Shimura subvariety $Z\subset Sh_2$, we have that $\overline{T_X}\subset \Sh_{2,k}$ is not contained in $\cZ_k$, where $\cZ$ denotes the Zariski closure of $Z$ in $\Sh_2$;
    \item $f^{-1}(\overline{T_X}) \cap \Sh_{1,k}$ is stable under the Hecke action of $G_1(\bA^p_f)$ on $\Sh_{1,k}$. 
%
\end{enumerate}
\end{lemma}
\begin{proof}
\begin{enumerate}
%
    \item The $\ell$-adic monodromy of the $\ell$-adic lisse sheaf given by the relative $H^1_{\ell, \text{\'et}}$ of the universal abelian variety restricted to any Hecke-stable subvariety in  must be Zariski-dense in $G_2(\bQ_{\ell})$. Note that since $T_X$ is Hecke-stable, then $\overline{T_X}$ is Hecke-stable as all Hecke correspondences are algebraic. It then follows that $\overline{T_X}$ is not contained in any $\cZ_k$ since the $\ell$-adic monodromy of the family of abelian varieties over $Z$ is contained in the reductive group associated to $Z$, which is a proper algebraic subgroup of $G_2(\bQ_\ell)$.
    
    \item Note that $\overline{T_X}$ is stable under $G_2(\bA^p_f)$, then it suffices to prove that for any $x\in \Sh_1(k)$ and for any $g\in G_1(\bA^p_f)$, if $x'\in g.x$, then there exists $g'\in G_2(\bA^p_f)$ such that $f(x')\in g'.f(x)$.
    Indeed, we may take $g'=f(g)$, where we view $f:G_1\rightarrow G_2$, and the desired property follows from the definition of Hecke actions via the extension property of canonical integral models given in \cite[Thm.~2.3.8]{Kisin}. \qedhere
%
\end{enumerate}
\end{proof}

\begin{proof}[Proof of \Cref{thm_Hecke} orthogonal case]
We will induct on $\dim_k\Sh_k=b$. When $b=1$, $\Sh_k$ is a curve and since the prime-to-$p$ Hecke orbit of an ordinary point is infinite, and thus the base case is verified. 

Now assume that \Cref{thm_Hecke} holds for all ordinary points in the special fiber of the canonical integral model of GSpin Shimura varieties of dimension $b-1$ with hyperspecial level. Consider $x\in \Sh(k)$ ordinary; since the prime-to-$p$ Hecke orbit $T_x$ of $x$ is infinite and hence its Zariski closure $\overline{T_x} \subset \Sh_k$ is at least $1$-dimensional, and by definition, $\overline{T_x}$ is generically ordinary. If the Zariski closure $\overline{T_x}^\BB$ in $\Sh^\BB_k$ contains a $0$-dimensional cusp, then the theorem follows directly from \Cref{pf_dim0} since $\overline{T_x}^\BB\cap \Sh_k\neq \emptyset$. If $\overline{T_x}^\BB$ in $\Sh^\BB_k$ contains a point in a $1$-dimensional cusp, then by \Cref{pf_Hecke-1dim}, Case (1) is done and we may assume that we are in Case (2). In other words, we remain to prove the theorem for Case (2) in \Cref{pf_Hecke-1dim} and the case when $\overline{T_x}^\BB=\overline{T_x}$.

By \Cref{basichecke}(1), we have that $\overline{T_x}\not\subset Z(m)$ for any $m$ since all $Z(m)$ are (finite unions of) special fibers of the Zariski closure of proper Shimura subvarieties of $Sh$ in $\Sh$ (here we use the fact that $\cZ(m)$ are all flat). For Case (2), since there is always a proper curve in a projective variety of dimension at least $2$ 
avoiding finitely many points, 
then we may always find a proper curve $C'$ in $\overline{T_x}$. On the other hand, when $\overline{T_x}^\BB=\overline{T_x}$, since $\dim_k \overline{T_x}\geq 1$, we may also find a proper curve in $\overline{T_x}$. Since $\overline{T_x}$ is generically ordinary and not contained any $Z(m)$, we may always choose a proper curve $C'$ such that $C'$ is generically ordinary and $C'\not\subset Z(m)$ for any $m$. Then by \Cref{thm_int} applying to the normalization $C$ of $C'$ with the natural map $C\rightarrow C'\rightarrow \Sh_k$, there exists an ordinary point $x'$ on $C'\subset \overline{T_x}$ such that $x'\in Z(m)(k)$ for some $p\nmid m$ representable by $(L,Q)$, as there are only finitely many non-ordinary points on $C'$.

Let $\Sh'\subset \cZ(m)$ denote the canonical integral model of the Shimura subvariety of $\Sh$ which consists some irreducible components of $\cZ(m)$ and $x'\in \Sh'(k)$. Note that since $p\nmid m$, $\Sh'$ has hyperspecial level at $p$ and $\dim_k \Sh'_k=b-1$.
By \Cref{basichecke}(2), $\overline{T_x}\cap \Sh'_k$ is a generically ordinary Hecke-stable subvariety of $\Sh'_k$. Then by the inductive hypothesis, we have that $\overline{T_x} \cap \Sh'_k = \Sh'_k$, and thus $\Sh'_k \subset \overline{T_x}$. In fact, an identical argument yields that $Z'(m) \subset \overline{T_x}$ for infinitely many $m$, where $Z'(m)$ is some irreducible component of $Z(m)$; indeed, if there were only finitely many such $Z'(m)$, they only intersect $C$ at finitely many $k$-points and we may always pick $x'$ different from these finitely many points when we apply \Cref{thm_int}. Since the Zariski closure of infinitely many distinct subvarieties of dimension $b-1$ must be at least $b$-dimensional, we conclude that $\overline{T_x}$ must contain at least one irreducible component of $\Sh_k$. Moreover, since the Hecke action $G(\bA^p_f)$ on $\Sh_k$ permutes all its irreducible/connected components, we conclude that $\overline{T_x}=\Sh_k$.
\end{proof}

\begin{remark}\label{rmk_unitary}
By \cite[\S 9.3]{SSTT}, as a direct consequence of \Cref{thm_int}, we have that \Cref{thm_int} still holds for $\Sh_k$ being the $\bmod\, \fp$ special fiber of the canonical integral model over $\Spec \cO_{K,(\fp)}$ of the PEL type unitary Shimura variety considered in \cite[\S 2.1]{BHKRY} and \cite[\S 2.1, Notation 2.6]{KR14} with the special divisors described in \cite[\S 2.5]{BHKRY} and \cite[\S 2.2, Def.~2.8]{KR14}, where $\fp\mid p$ and $p$ splits in $K/\bQ$. Therefore, we adapt the above inductive proof for the orthogonal case to the unitary case if $\overline{T_x}$ is proper; thus to finish the proof of \Cref{thm_Hecke} unitary case, it remains to treat the case when $\overline{T_x}^\BB$ hits the boundary of $\Sh^\BB_k$. 

The arithmetic compactifications of $\Sh$ are described in \cite[\S 3]{BHKRY}. More precisely, by \cite[Thm.~3.7.1, Prop.~3.4.4]{BHKRY}, the boundary components of $\Sh^\BB$ are $0$-dimensional (relative to $\Spec \cO_{K,\fp}$); the toroidal compactification $\Sh^{\tor}$ is canonical and the fibers over the cusps of $\Sh^{\tor}\rightarrow \Sh^\BB$ are abelian schemes and each of these abelian schemes, up to quotient by a finite group, is isomorphic (over some finite extension of $\cO_K$) to $E\otimes_{\cO_K}\Lambda_0$, where $E$ is an elliptic curve CM by $\cO_K$ and $\Lambda_0$ is an $\cO_K$-lattice of rank $n-1$. Since $\Sh^{\tor}$ is canonical, the Hecke action $G(\bA^p_f)$ (here $G$ denotes the reductive group associated to $\Sh$) extends to $\Sh^{\tor}$. For each cusp in $\Sh^\BB$, we may choose an isotropic line $J\subset W$, where $W$ is the Hermitian space over $K$ of signature $(n,1)$ used to define $\Sh$. The admissible parabolic associated to the cusp is the stabilizer of $J$ and by \cite[\S 3.3, p.~673]{Howard}, the $\bZ[1/\ell]$-points of the unipotent part of this parabolic acts on $E\otimes_{\cO_K}\Lambda_0$ by translations of $\ell$-power torsion points and thus we prove the analogous statement of \Cref{onedimbdrycharzero} for the unitary case. Therefore we prove the unitary case of \Cref{thm_Hecke} by the proof of \Cref{pf_Hecke-1dim}.
\end{remark}

\end{document}